\documentclass[11pt]{amsart}
\usepackage{amsmath}
\usepackage{amssymb}
\usepackage{amsthm}
\usepackage{latexsym}
\usepackage{hyperref}
\usepackage{enumerate}
\usepackage{graphicx}
\usepackage[all]{xy}

\newtheorem{thm}{Theorem}[section]
\newtheorem{cor}[thm]{Corollary}
\newtheorem{lem}[thm]{Lemma}
\newtheorem{prop}[thm]{Proposition}

\newtheorem{thmintro}{Theorem}

\theoremstyle{definition}
\newtheorem{defn}[thm]{Definition}
\newtheorem{rem}[thm]{Remark}
\newtheorem{cond}[thm]{Condition}

\newcommand{\enuma}[1]{\begin{enumerate}[\textup{(}a\textup{)}] {#1} \end{enumerate}}

\newcommand{\mb}{\mathbf}

\newcommand{\mh}{\mathbb}
\newcommand{\mr}{\mathrm}
\newcommand{\mc}{\mathcal}
\newcommand{\mf}{\mathfrak}

\newcommand{\isom}{\xrightarrow{\;\sim\;}}

\newcommand{\N}{\mathbb N}
\newcommand{\Z}{\mathbb Z}

\newcommand{\R}{\mathbb R}
\newcommand{\C}{\mathbb C}

\newcommand{\inp}[2]{\langle #1 \,,\, #2 \rangle}

\newcommand{\matje}[4]{\left(\begin{smallmatrix} #1 & #2 \\ 
#3 & #4 \end{smallmatrix}\right)}

\newcommand{\q}{/\!/}

\def\For{{\rm For}}
\def\Hom{{\rm Hom}}
\def\End{{\rm End}}

\def\Irr{{\rm Irr}}

\def\Unip{{\mathcal U}}

\def\SL{{\rm SL}}

\def\cC{{\mathcal C}}
\def\cG{{\mathcal G}}

\def\cS{{\mathcal S}}

\def\cV{{\mathcal V}}
\def\cL{{\mathcal L}}
\def\cH{{\mathcal H}}
\def\cB{{\mathcal B}}
\def\cO{{\mathcal O}}

\def\cF{{\mathcal F}}
\def\cP{{\mathcal P}}
\def\cR{{\mathfrak R}}

\def\Fr{{\rm Frob}}

\def\reg{{\rm reg}}

\def\ind{{\rm ind}}

\def\Mod{{\rm Mod}}
\def\nr{{\rm nr}}

\def\fs{{\mathfrak s}}

\def\Mod{{\rm Mod}}

\def\Stab{{\rm Stab}}

\def\Ad{{\rm Ad}}
\def\Lie{{\rm Lie}}

\def\der{{\rm der}}
\def\ad{{\rm ad}}
\def\sc{{\rm sc}}
\def\Aut{{\rm Aut}}

\def\temp{{\rm temp}}

\def\cusp{{\rm cusp}}

\def\pr{{\rm pr}}

\def\fB{{\mathfrak B}}
\def\IC{{\rm{IC}}}
\def\IM{{\rm{IM}}}

\begin{document}

\title[Graded Hecke algebras for disconnected reductive groups]{Graded Hecke algebras \\
for disconnected reductive groups}

\author[A.-M. Aubert]{Anne-Marie Aubert}
\address{Institut de Math\'ematiques de Jussieu -- Paris Rive Gauche, 
U.M.R. 7586 du C.N.R.S., U.P.M.C., 4 place Jussieu 75005 Paris, France}
\email{anne-marie.aubert@imj-prg.fr}
\author[A. Moussaoui]{Ahmed Moussaoui}
\address{Department of Mathematics and Statistics, University of Calgary, 
2500 University Drive NW, Calgary, Alberta, Canada}
\email{ahmed.moussaoui@ucalgary.ca}
\author[M. Solleveld]{Maarten Solleveld}
\address{IMAPP, Radboud Universiteit Nijmegen, Heyendaalseweg 135, 
6525AJ Nijmegen, the Netherlands}
\email{m.solleveld@science.ru.nl}
\date{\today}
\maketitle

\begin{abstract}
We introduce graded Hecke algebras $\mh H$ based on a (possibly disconnected)
complex reductive group $G$ and a cuspidal local system $\cL$ on a unipotent
orbit of a Levi subgroup $M$ of $G$. These generalize the graded Hecke algebras
defined and investigated by Lusztig for connected $G$. 

We develop the representation theory of the algebras $\mh H$. obtaining complete 
and canonical parametrizations of the irreducible, the irreducible tempered and 
the discrete series representations. All the modules are constructed in 
terms of perverse sheaves and equivariant homology, relying on work of Lusztig.
The parameters come directly from the data $(G,M,\cL)$ and they are closely
related to Langlands parameters. 

Our main motivation for considering these graded Hecke algebras is that the
space of irreducible $\mh H$-representations is canonically in bijection with
a certain set of "logarithms" of enhanced L-parameters. Therefore we expect
these algebras to play a role in the local Langlands program. We will make
their relation with the local Langlands correspondence, which goes via affine 
Hecke algebras, precise in a sequel to this paper.\\
\textbf{Erratum.} Theorem \ref{thm:2.24} and Proposition \ref{prop:2.25} were not 
entirely correct as stated. This is repaired in a new appendix.
\end{abstract}

\vspace{7mm}

\tableofcontents

\newpage

\section*{Introduction}

The study of Hecke algebras and more specifically their simple modules is 
a powerful tool in representation theory. They can be used to build bridges 
between different objects. Indeed they can arise arithmetically (as endomorphism 
algebras of a parabolically induced representation) or geometrically (using 
K-theory or equivariant homology). For example, this strategy was successfully 
used by Lusztig in his Langlands parametrization of unipotent representations 
of a connected, adjoint simple unramified group over a nonarchimedean local 
field \cite{Lus6,Lus8}. This paper is part of a series, whose final goal is
to generalize these methods to arbitrary irreducible representations of
arbitrary reductive $p$-adic groups. In the introduction we discuss the 
results proven in the paper, and in Section \ref{sec:outline} we shed some light 
on the envisaged relation with the Langlands parameters. 

After \cite{AMS}, where the authors extended the generalized Springer 
correspondence in the context of a reductive disconnected complex group, 
this article is devoted to generalize in this context several results of the 
series of papers of Lusztig \cite{Lus3,Lus5,Lus7}. Let $G$ be an complex reductive 
algebraic group with Lie algebra $\mf g$. Although we do not assume that G is connected, 
it has only finitely components because it is algebraic. Let $L$ be a Levi factor 
of a parabolic subgroup $P$ of $G^\circ$, $T=Z(L)^{\circ}$ the connected center of 
$L$, $\mf t$ its Lie algebra and $v \in \mf l = \mathrm{Lie}(L)$ be nilpotent. 
Let $\cC_v^L$ be the adjoint orbit of $v$ and let $\cL$ be an irreducible 
$L$-equivariant cuspidal local system on $\cC_v^L$. The triples $(L,\cC_v^L,\cL)$ 
(or more precisely their $G$-conjugacy classes) defined by data of the above 
kind will be called \emph{cuspidal supports} for $G$. 
We associate to $\tau=(L,\cC_v^L,\cL)_G$ a twisted version $\mh H (G,L,\cL)=\mh H (G,\tau)$ 
of a graded Hecke algebra  
and study its simple modules. More precisely, let $W_\tau= N_G (\tau) / L$, 
$W_\tau^\circ = N_{G^\circ}(\tau) / L$ and $\cR_\tau = N_G (P,\cL) / L$. 
Then $W_\tau = W_\tau^\circ \rtimes \cR_\tau$. Let ${\mb r}$ be an indeterminate 
and $\natural_{\tau} \colon \cR_\tau^2 \to \C^\times$ be a (suitable) 2-cocycle. 
The twisted graded Hecke algebra associated to $\tau$ is the vector space
\[
\mh H (G,\tau) =\C [W_\tau,\natural_{\tau}] \otimes S(\mf t^*) \otimes \C[{\mb r}] ,
\] with multiplication as in Proposition \ref{prop:1.4}. 
As $W_\tau= W_\tau^\circ \rtimes \cR_\tau$ and $W_\tau$ plays the role of 
$W_\tau^\circ$ in the generalized Springer correspondence for disconnected groups, 
the algebra $\mh H (G,\tau)$ contains the graded Hecke algebra 
$\mh H (G^{\circ},\tau)$ defined by Lusztig in \cite{Lus3} and plays the role of
the latter in the disconnected context. More precisely, let $y \in \mf g$ be 
nilpotent and let $(\sigma,r) \in \mf g \oplus \C$ be semisimple such that 
$[\sigma,y]=2ry$. Let $\sigma_0=\sigma-r h \in \mf t$ with $h \in \mf g$ 
a semisimple element which commutes with $\sigma$ and which arises in a 
$\mf{sl}_2$-triple containing $y$. 
Then we have $\pi_0 (Z_G (\sigma,y))=\pi_0 (Z_G (\sigma_0,y))$, where $Z_G (\sigma,y)$ denotes 
the simultaneous centralizer of $\sigma$ and $y$ in $G$, and respectively for $\sigma_0$. 
We also denote by $\Psi_G$ the cuspidal support map defined in \cite{Lus2,AMS}, 
which associates to every pair $(x,\rho)$ with $x \in \mf g$ nilpotent and 
$\rho \in \Irr \big( \pi_0 (Z_G (x)) \big)$ (with $Z_G (x)$ the centralizer of $x$ in $G$) 
a cuspidal support $(L',\cC_{v'}^{L'},\cL')$.\\

Using equivariant homology methods, we define standard modules in the same 
way as in \cite{Lus3} and denote by $E_{y,\sigma,r}$ (resp. 
$E_{y,\sigma,r,\rho}$) the one which is associated to $y,\sigma,r$ (resp. $y,\sigma,r$ 
and $\rho \in \Irr \big( \pi_0 (Z_G (\sigma,y)) \big)$). They are modules over 
$\mh H (G,\tau)$ and we have the following theorem:

\begin{thmintro}\label{thmintro:1}
Fix $r \in \C$. 
\enuma{
\item Let $y,\sigma \in \mf g$ with $y$ nilpotent, $\sigma$ semisimple
and $[\sigma,y] = 2r y$. Let $\rho\in \\ \Irr \big( \pi_0 (Z_G (\sigma_0,y)) \big)$
such that $\Psi_{Z_G (\sigma_0)} (y,\rho) = \tau=(L,\cC_v^L,\cL)_G$. With these data 
we associate a $\mh H (G,\tau)$-module 
$E_{y,\sigma,r,\rho}$. The $\mh H (G,\tau)$-module $E_{y,\sigma,r,\rho}$ has a 
distinguished irreducible quotient $M_{y,\sigma,r,\rho}$, which appears with 
multiplicity one in $E_{y,\sigma,r,\rho}$.
\item The map $M_{y,\sigma,r,\rho} \longleftrightarrow (y,\sigma,\rho)$ gives
a bijection between $\Irr_r (\mh H (G,\tau))$ and $G$-conjugacy classes of
triples as in part (a).
\item The set $\Irr_r (\mh H (G,\tau))$ is also canonically in bijection with
the following two sets:
\begin{itemize}
\item $G$-orbits of pairs $(x,\rho)$ with $x \in \mf g$ and $\rho \in \Irr \big(
\pi_0 (Z_G (x)) \big)$ such that $\Psi_{Z_G (x_S)} (x_N,\rho) = \tau$,
where $x=x_S+x_N$ is the Jordan decomposition of~$x$.
\item $N_G (L)/L$-orbits of triples $(\sigma_0,\cC,\cF)$, with $\sigma_0 \in \mf t$,
$\cC$ a nilpotent $Z_G (\sigma_0)$-orbit in $Z_{\mf g}(\sigma_0)$ and $\cF$ a
$Z_G (\sigma_0)$-equivariant cuspidal local system on $\cC$ such that
$\Psi_{Z_G (\sigma_0)} (\cC,\cF) = \tau$.
\end{itemize}
}
\end{thmintro}

Next we investigate the questions of temperedness and discrete series of 
$\mh H (G,\tau)$-modules. Recall that the vector space 
$\mf t=X_*(T) \otimes_{\Z} \C$ has a decomposition 
$\mf t=\mf t_{\R} \oplus i \mf t_{\R}$ with $\mf t_{\R}=X_*(T) \otimes_{\Z} \R$. 
Hence any $x \in \mf t$ can be written uniquely as $x = \Re (x)+i \Im(x)$. 
We obtain the following:

\begin{thmintro}(see Theorem \ref{thm:2.19}) \\
\label{thmintro:2}
Let $y,\sigma,\rho$ be as above with $\sigma, \sigma_0 \in \mf t$.
\enuma{
\item Suppose that $\Re (r) \leq 0$. The following are equivalent:
\begin{itemize}
\item $E_{y,\sigma,r,\rho}$ is tempered;
\item $M_{y,\sigma,r,\rho}$ is tempered;
\item $\sigma_0 \in i \mf t_\R$.
\end{itemize}
\item Suppose that $\Re (r) \geq 0$. Then part (a) remains valid if we
replace tempered by anti-tempered.\\

Assume further that $G^\circ$ is semisimple.
\item Suppose that $\Re (r) < 0$. The following are equivalent:
\begin{itemize}
\item $M_{y,\sigma,r,\rho}$ is discrete series;
\item $y$ is distinguished in $\mf g$, that is, it is not contained in any
proper Levi subalgebra of $\mf g$.
\end{itemize}
Moreover, if these conditions are fulfilled, then $\sigma_0 = 0$ and
$E_{y,\sigma,r,\rho} = M_{y,\sigma,r,\rho}$.
\item Suppose that $\Re (r) > 0$. Then part (c) remains valid if we replace
(i) by: $M_{y,\sigma,r,\rho}$ is anti-discrete series.
\item For $\Re (r) = 0$ there are no (anti-)discrete series representations
on which $\mb r$ acts as $r$.
} 
\end{thmintro}

Moreover, using the Iwahori--Matsumoto involution we give another description 
of tempered modules when $\Re (r)$ is positive, and this is more suitable in 
the context of the Langlands correspondence. 

The last section consists 
of the formulation of the previous results in terms of cuspidal quasi-supports, 
which is more adapted than cuspidal supports in the context of Langlands 
correspondence, as it can be seen in \cite[\textsection 5--6]{AMS}.

Recall that a quasi-Levi subgroup of $G$ is a group of the form
$M = Z_G (Z(L)^\circ)$, where $L$ is a Levi subgroup of $G^\circ$. Thus
$Z (M)^\circ = Z(L)^\circ$ and $M \longleftrightarrow L = M^\circ$ is a bijection
between quasi-Levi subgroups of $G$ and the Levi subgroups of $G^\circ$.

A \emph{cuspidal quasi-support} for $G$ is the $G$-conjugacy class of $q\tau$ 
of a triple $(M,\cC_v^M,q\cL)$, where $M$ is 
a quasi-Levi subgroup of $G$, $\cC_v^M$ is a nilpotent $\Ad(M)$-orbit in 
$\mf m = \mathrm{Lie}(M)$ and $q \cL$ is a $M$-equivariant cuspidal local 
system on $\cC_v^M$, i.e. as  $M^\circ$-equivariant local system it is a 
direct sum of cuspidal local systems. We denote by $q \Psi_G$ the cuspidal 
quasi-support map defined in \cite[\textsection 5]{AMS}. With the cuspidal 
quasi-support $q\tau=(M,\cC_v^M,q\cL)_G$, we associate a twisted graded Hecke algebra 
denoted $\mh H (G,q\tau)$.

\begin{thmintro}\label{thmintro:3}
The analog of Theorem with quasi cuspidal supports instead of cuspidal ones holds true.
\end{thmintro}

The article is organized as follows. 
The first section is introductory, it explains why and how the study of enhanced Langlands
parameters motivated this paper.
The second section contains the definition of the twisted graded Hecke algebra 
associated to a cuspidal support. After that we study the representations of these 
Hecke algebras in the third section. To do that we define the standard modules and 
we relate them to the standard modules defined in the connected case by Lusztig. 
As preparation we study precisely the modules annihilated by ${\mb r}$. By Clifford 
theory, as explained in \cite[\textsection 1]{AMS}, we show then that the simple modules 
over $\mh H (G,\tau)$ can be parametrized in a compatible way by the objects in 
part (c) and (d) of the first theorem in this introduction. We deduce then the first 
theorem. The second part consist of the study of temperedness and discrete series. Note 
that we show a version of the ABPS conjecture for these Hecke algebra. To conclude, 
the last section is devoted to the adaption of the previous results for a cuspidal 
quasi-support as described above.

\smallskip

\textbf{Acknowledgements.}\\
We thank Dan Ciubotaru and George Lusztig for helpful discussions.

\section{The relation with Langlands parameters}
\label{sec:outline}

This article is part of a series the main purpose of which is to construct a bijection
between enhanced Langlands parameters for $\cG (F)$ and a certain collection of
irreducible representations of twisted affine Hecke algebras, with possibly unequal parameters.
The parameters appearing in Theorems \ref{thmintro:1} and \ref{thmintro:3} are quite close
to those in the local Langlands correspondence, and with the exponential map one can make
that precise. To make optimal use of Theorem \ref{thmintro:3}, we will show that the parameters
over there constitute a specific part of one Bernstein component in the space of enhanced
L-parameters for one group. Let us explain this in more detail.

Let $F$ be a local non-archimedean field,
let $\mb W_F$ be the Weil group of $F$, $\mb I_F$ the inertia subgroup of $\mb W_F$, and 
$\Fr_F \in \mb W_F$ an arithmetic Frobenius element. Let $\cG$ be a connected reductive 
algebraic group defined over $F$,
and $\cG^\vee$ be the complex dual group of $\cG$. The latter is endowed with an action of
$\mb W_F$, which preserves a pinning of $\cG^\vee$. The Langlands
dual group of the group $\cG(F)$ of the $F$-rational points of $\cG$ is 
${}^L \cG := \cG^\vee \rtimes \mb W_F$. 

A Langlands parameter (L-parameter for short) for ${}^L \cG$ is a continuous group homomorphism
\[
\phi \colon \mb W_F \times \SL_2 (\C) \to \cG^\vee \rtimes \mb W_F 
\]
such that $\phi (w) \in \cG^\vee w$ for all $w \in \mb W_F$, the image of $\mb W_F$ under 
$\phi$ consists of semisimple elements, and $\phi |_{\SL_2 (\C)}$ is algebraic.

We call a L-parameter \emph{discrete}, if $Z_{\cG^\vee}(\phi)^\circ = Z(\cG^\vee)^{\mb W_F,\circ}$.
With \cite[\S 3]{Bor} it is easily seen that this definition of discreteness is equivalent
to the usual one with proper Levi subgroups.

Let $\cG_\sc^\vee$ be the simply connected cover of the derived group $\cG^\vee_\der$.
Let $Z_{\cG^\vee_\ad}(\phi)$ be the image of $Z_{\cG^\vee}(\phi)$ in the adjoint
group $\cG^\vee_\ad$. We define
\[
Z^1_{\cG^\vee_\sc}(\phi) = \text{ inverse image of } Z_{\cG^\vee_\ad}(\phi)
\text{ under } \cG^\vee_\sc \to \cG^\vee . 
\]
To $\phi$ we associate the finite group $\cS_\phi := \pi_0 (Z^1_{\cG^\vee_\sc}(\phi))$.
An \emph{enhancement} of $\phi$ is an irreducible representation of $\cS_\phi$. The group 
$\cS_\phi$ coincides with the group considered by both Arthur in \cite{Art} and Kaletha in 
\cite[\S 4.6]{Kal}.

The group $\cG^\vee$ acts on the collection of enhanced L-parameters for
${}^L \cG$ by 
\[
g \cdot (\phi,\rho) =  (g \phi g^{-1}, g \cdot \rho).
\]
Let $\Phi_e ({}^L \cG)$ denote the collection of $\cG^\vee$-orbits of enhanced L-parameters.

Let us consider $\cG (F)$ as an inner twist of a quasi-split group. Via the Kottwitz
isomorphism it is parametrized by a character of $Z(\cG^\vee_\sc)^{\mb W_F}$, say
$\zeta_\cG$. We say that $(\phi,\rho) \in \Phi_e ({}^L \cG)$ is relevant for $\cG (F)$ 
if $Z(\cG^\vee_\sc)^{\mb W_F}$ acts on $\rho$ as $\zeta_\cG$. The subset of 
$\Phi_e ({}^L \cG)$ which is relevant for $\cG (F)$ is denoted $\Phi_e (\cG (F))$.

As it is well-known $(\phi,\rho) \in \Phi_e ({}^L \cG)$ is already determined by
$\phi |_{\mb W_F}, u_\phi := \phi \big(1, \matje{1}{1}{0}{1} \big)$ and $\rho$.
Sometimes we will also consider $\cG^\vee$-conjugacy classes of such triples\\
$(\phi |_{\mb W_F}, u_\phi, \rho)$ as enhanced L-parameters. An enhanced L-parameter
$(\phi |_{\mb W_F}, v, q\epsilon)$ will often be abbreviated to $(\phi_v, q \epsilon)$.

For $(\phi,\rho) \in \Phi_e ({}^L \cG)$ we write 
\begin{equation} \label{eqn:Gphi}
G_\phi := Z^1_{\cG^\vee_\sc}(\phi |_{\mb W_F}),
\end{equation} 
a complex (possibly disconnected) reductive group. We say that $(\phi,\rho)$
is \emph{cuspidal} if $\phi$ is discrete and $(u_\phi = \phi \big(1, \matje{1}{1}{0}{1} \big),\rho)$ 
is a cuspidal pair for $G_\phi$: this means that $\rho$ corresponds to a 
$G_\phi$-equivariant cuspidal local system $\cF$ on $\cC^{G_\phi}_{u_\phi}$. 
We denote the collection of cuspidal L-parameters for ${}^L \cG$ by $\Phi_\cusp ({}^L \cG)$, 
and the subset which is relevant for $\cG (F)$ by $\Phi_\cusp (\cG (F))$. 

Let $G$ be  a complex (possibly disconnected) reductive group. 
We define the \emph{enhancement of the unipotent variety} of $G$ as the set:
\[
\Unip_e(G):=\{(\cC^{G}_u,\rho)\,:\,\text{with $u\in G$ unipotent and  $\rho\in\Irr(\pi_0(Z_{G}(u))$}\},
\]
and call a pair $(\cC^{G}_u,\rho)$ an \emph{enhanced unipotent class}.
Let $\fB(\Unip_e(G))$ be the set of $G$-conjugacy classes of triples $(M,\cC^M_v,q\epsilon)$, 
where $M$ is a quasi-Levi subgroup of $G$, and $(\cC^M_v,q\epsilon)$ 
is a cuspidal enhanced unipotent class in $M$.

In \cite[Theorem 5.5]{AMS}, we have attached to every element $q\tau\in\fB(\Unip_e(G))$ a $2$-cocycle
\[\kappa_{q\tau}\colon W_{q\tau}/W_{q\tau}^\circ
\times W_{q\tau}/W_{q\tau}^\circ\to\C^\times\]
where $W_{q\tau}:=N_{G}(q\tau)/M$ and $W_{q\tau}^\circ:=N_{G^\circ}(M^\circ)/M^\circ$,
and constructed a \emph{cuspidal support map} 
\[q\Psi_G\colon \Unip_e(G)\to \fB(\Unip_e(G))\]
such that 
\begin{equation} \label{eqn:Unip}
\Unip_e(G)=\bigsqcup_{q\tau\in\fB(\Unip_e(G))}q\Psi_G^{-1}(q\tau),
\end{equation}
where $q\Psi_G^{-1}(q\tau)$ is in bijection with the set of isomorphism classes of irreducible
representations of twisted algebra $\C[W_{q\tau},\kappa_{q\tau}]$. Our construction is an 
extension of, and is based on, the Lusztig's construction of the generalized Springer correspondence 
for $G^\circ$ in \cite{Lus2}. 

Let $(\phi,\rho) \in \Phi_e (\cG (F))$. We will first apply the construction above to the group
$G=G_\phi$ in order to obtain a partition of $\Phi_e (\cG (F))$ in the spirit of (\ref{eqn:Unip}). 
We write $q\Psi_{G_\phi} = [M, v,q\epsilon]_{G_\phi}$.
We showed in \cite[Proposition 7.3]{AMS} that, upon replacing $(\phi,\rho)$ by $\cG^\vee$-conjugate, 
there exists a Levi subgroup
$\cL (F) \subset \cG (F)$ such that $(\phi |_{\mb W_F},v,q\epsilon)$ is a cuspidal 
L-parameter for $\cL (F)$. Moreover, 
\[
\cL^\vee \rtimes \mb W_F = Z_{\cG^\vee \rtimes \mb W_F}(Z(M)^\circ) .
\]
We set
\[
{}^L \Psi (\phi,\rho) := (\cL^\vee \rtimes \mb W_F, \phi |_{\mb W_F}, v,q\epsilon) .
\]
The right hand side consists of a Langlands dual group and a cuspidal L-parameter for 
that. Every enhanced L-parameter for ${}^L \cG$ is conjugate to one as above, so the map
${}^L \Psi$ is well-defined on the whole of $\Phi_e ({}^L \cG)$.

In \cite{AMS}, we defined Bernstein components of enhanced L-parameters. Recall from 
\cite[\S 3.3.1]{Hai} that the group of unramified characters of $\cL (F)$ is naturally
isomorphic to $Z (\cL^\vee \rtimes \mb I_F )^\circ_{\mb W_F}$. We consider this as an
object on the Galois side of the local Langlands correspondence and we write
\[
X_\nr ({}^L \cL) = Z (\cL^\vee \rtimes \mb I_F )^\circ_{\mb W_F} .
\]
Given $(\phi',\rho') \in \Phi_e (\cL (F))$ and $z \in X_\nr ({}^L \cL)$, we define 
$(z \phi',\rho') \in \Phi_e (\cL (F))$ by
\[
z \phi' = \phi' \text{ on } \mb I_F \times \SL_2 (\C) \text{ and }
(z \phi')(\Fr_F) = \tilde z \phi' (\Fr_F) ,
\]
$\tilde z \in Z (\cL^\vee \rtimes \mb I_F )^\circ$ represents $z$.
By definition, an \emph{inertial equivalence class} for $\Phi_e (\cG (F))$ consists of 
a Levi subgroup $\cL (F) \subset \cG (F)$ and a $X_\nr ({}^L \cL)$-orbit $\mf s^\vee_\cL$ in 
$\Phi_\cusp (\cL (F))$. 
Another such object is regarded as equivalent if the two are conjugate by an 
element of $\cG^\vee$. The equivalence class is denoted $\mf s^\vee$.

The Bernstein component of $\Phi_e (\cG (F))$ associated to $\mf s^\vee$ is defined as
\begin{equation} 
\Phi_e (\cG (F))^{\mf s^\vee} := {}^L \Psi^{-1} (\cL \rtimes \mb W_F, \mf s^\vee_\cL).
\end{equation}
In particular $\Phi_e (\cL (F))^{\mf s^\vee_\cL}$ is diffeomorphic to a quotient of
the complex torus $X_\nr ({}^L \cL)$ by a finite subgroup, albeit not in a canonical way.

With an inertial equivalence class $\fs^\vee$ for $\Phi_e (\cG (F))$ we associate the
finite group
\[
W_{\fs^\vee} := \text{ stabilizer of } \fs^\vee_\cL \text{ in } 
N_{\cG^\vee}(\cL^\vee \rtimes \mb W_F) / \cL^\vee .
\]
It plays a role analogous to that of the finite groups appearing in the description
of the Bernstein centre of $\cG (F)$. 
We expect  that the local Langlands correspondence for $\cG(F)$ 
matches every Bernstein component $\Irr^\fs(\cG(F))$ for $\cG(F)$, where $\fs=[\cL(F),\sigma]_{\cG(F)}$, 
with $\cL$ an $F$-Levi subgroup of an $F$-parabolic subgroup of $\cG$ and $\sigma$ 
an irreducible supercupidal smooth representation of $\cL(F)$, with a Bernstein component
$\Phi_e (\cG (F))^{\mf s^\vee}$, where $\fs^\vee=[\cL(F),\fs_\cL^\vee]_{\cG^\vee}$, and 
that the (twisted) affine Hecke algebras on both sides will correspond.   

Let $W_{\fs^\vee,\phi_v,q\epsilon}$ be the isotropy group of $(\phi_v, q\epsilon) \in 
\fs^\vee_\cL$. Let $\cL^\vee_c \subset \cG^\vee_\sc$ denote the preimage of $\cL^\vee$ under $\cG^\vee_\sc 
\to \cG^\vee$. With the generalized Springer correspondence, applied to the group $G_\phi \cap \cL^\vee_c$, 
we can attach to any element of ${}^L \Psi^{-1} (\cL^\vee \rtimes \mb W_F, \phi_v, 
q\epsilon)$ an irreducible projective representation of $W_{\fs^\vee,\phi_v,q\epsilon}$.
More precisely, set 
\[
q \tau := [G_\phi \cap \cL^\vee_c, v ,q \epsilon ]_{G_\phi}.
\]
By \cite[Lemma 8.2]{AMS} $W_{q \tau}$ is canonically isomorphic to
$W_{\fs^\vee,\phi_v,q\epsilon}$. To the data $q \tau$ we will attach (in Section \ref{sec:3})
twisted graded Hecke algebras, whose irreducible representations are parametrized by 
triples $(y,\sigma_0,\rho)$ related to $\Phi_e (\cG (F))^{\mf s^\vee}$. Explicitly, using the exponential 
map for the complex reductive group $Z_{\cG^\vee}(\phi (\mb W_F))$, we can construct 
$(\phi',\rho') \in \Phi_e (\cG (F))^{\mf s^\vee}$ with $u_{\phi'} = \exp (y)$ and 
$\phi' (\Fr_F) = \phi_v (\Fr_F) \exp (\sigma_0)$.

In the sequel \cite{AMS2} to this paper, we associate to every Bernstein component 
$\Phi_e (\cG (F))^{\mf s^\vee}$ a twisted affine Hecke algebra $\cH(\cG(F),\mf s^\vee,\vec z)$
whose irreducible representations are naturally parametrized by $\Phi_e (\cG (F))^{\mf s^\vee}$. 
Here $\vec z$ is an abbreviation for an array of complex parameters. 

For general linear groups (and their inner forms) and classical groups, it is proved in 
\cite{AMS2} that there are specializations $\vec {z}$ such that the algebras 
$\cH(\cG(F),\mf s^\vee,\vec{z})$ are those computed for representations.
In general, we expect that the simple modules of $\cH(\cG(F),\mf s^\vee,\vec{z})$ should be in 
bijection with that of the Hecke algebras for types in reductive $p$-adic groups (which is the 
case for special linear groups and their inner forms), and in this way they should contribute 
to the local Langlands correspondence.

\section{The twisted graded Hecke algebra of a cuspidal support}

Let $G$ be a complex reductive algebraic group with Lie algebra 
$\mf g$. Let $L$ be a
Levi subgroup of $G^\circ$ and let $v \in \mf l = \mathrm{Lie}(L)$ be nilpotent.
Let $\cC_v^L$ be the adjoint orbit of $v$ and let $\cL$ be an irreducible
$L$-equivariant cuspidal local system on $\cC_v^L$. Following \cite{Lus2,AMS} 
we call $(L,\cC_v^L,\cL)$ a cuspidal support for $G$.

Our aim is to associate to these data a graded Hecke algebra, possibly extended
by a twisted group algebra of a finite group, generalizing \cite{Lus3}. Since most 
of \cite{Lus3} goes through without any problems if $G$ is disconnected, we focus 
on the parts that do need additional arguments.

Let $P = L U$ be a parabolic subgroup of $G^\circ$ with Levi factor $L$ and
unipotent radical $U$. Write $T = Z(L)^\circ$ and $\mf t = \mathrm{Lie}(T)$.
By \cite[Theorem 3.1.a]{AMS} the group $N_G (L)$ stabilizes $\cC_v^L$.
Let $N_G (\cL)$ be the stabilizer in $N_G (L)$ of the local system $\cL$ on
$\cC_v^L$. It contains $N_{G^\circ}(L)$ and it is the same as $N_G (\cL^*)$,
where $\cL^*$ is the dual local system of $\cL$. Similarly, let
$N_G (P,\cL)$ be the stabilizer of $(P,L,\cL)$ in $G$. We write
\begin{align*}
& W_\cL = N_G (\cL) / L , \\
& W_\cL^\circ = N_{G^\circ}(L) / L ,\\
& \cR_\cL = N_G (P,\cL) / L , \\
& R(G^\circ,T) = \{ \alpha \in X^* (T) \setminus \{0\} : \alpha
\text{ appears in the adjoint action of } T \text{ on } \mf g \} .
\end{align*}

\begin{lem}\label{lem:1.1}
\enuma{
\item The set $R(G^\circ,T)$ is (not necessarily reduced) root system
with Weyl group $W_\cL^\circ$.
\item The group $W_\cL^\circ$ is normal in $W_\cL$ and
$W_\cL = W_\cL^\circ \rtimes \cR_\cL$.
}
\end{lem}
\begin{proof}
(a) By \cite[Proposition 2.2]{Lus3} $R(G^\circ,T)$ is a root system, and
by \cite[Theorem 9.2]{Lus2} $N_{G^\circ}(L) / L$ is its Weyl group.\\
(b) Also by \cite[Theorem 9.2]{Lus2}, $W_\cL^\circ$ stabilizes $\cL$, so
it is contained in $W_\cL$. Since $G^\circ$ is 
normal in $G$, $W_\cL^\circ$ is normal in $W_\cL$. The group $\cR_\cL$
is the stabilizer in $W_\cL$ of the positive system $R(P,T)$ of
$R(G^\circ,T)$. Since $W_\cL^\circ$ acts simply transitively on the collection
of positive systems, $\cR_\cL$ is a complement for $W_\cL^\circ$.
\end{proof}

Now we give a presentation of the algebra that we want to study.
Let $\{ \alpha_i : i \in I\}$ be the set of roots in $R(G^\circ,T)$ which
are simple with respect to $P$. Let $\{s_i : i \in I\}$ be the associated set 
of simple reflections in the Weyl group $W_\cL^\circ = N_{G^\circ}(L) / L$. 
Choose $c_i \in \C \: (i \in I)$ such that $c_i = c_j$ if $s_i$ and $s_j$ are 
conjugate in $W_\cL$. We can regard $\{ c_i : i \in I \}$ as a $W_\cL$-invariant
function $c : R(G^\circ,T)_{\mathrm{red}} \to \C$, where the subscript "red" 
indicates the set of indivisible roots.

Let $\natural : (W_\cL / W_\cL^\circ)^2 \to \C^\times$ be a 2-cocycle.
Recall that the twisted group algebra $\C[W_\cL,\natural]$ has a $\C$-basis
$\{ N_w : w \in W_\cL \}$ and multiplication rules
\[
N_w \cdot N_{w'} = \natural (w,w') N_{w w'} . 
\]
In particular it contains the group algebra of $W_\cL^\circ$.

\begin{prop}\label{prop:1.4}
Let ${\mb r}$ be an indeterminate, identified with the coordinate function on $\C$. 
There exists a unique associative algebra structure
on $\C [W_\cL,\natural] \otimes S(\mf t^*) \otimes \C[{\mb r}]$ such that:
\begin{itemize}
\item the twisted group algebra $\C[W_\cL,\natural]$ is embedded as subalgebra;
\item the algebra $S(\mf t^*) \otimes \C[{\mb r}]$ of polynomial functions on $\mf t 
\oplus \C$ is embedded as a subalgebra;
\item $\C[{\mb r}]$ is central;
\item the braid relation $N_{s_i} \xi - {}^{s_i}\xi N_{s_i} = c_i {\mb r} (\xi - {}^{s_i} 
\xi) / \alpha_i$ holds for all $\xi \in S(\mf t^*)$ and all simple roots $\alpha_i$;
\item $N_w \xi N_w^{-1} = {}^w \xi$ for all $\xi \in S (\mf t^*)$ and $w \in \cR_\cL$. 
\end{itemize}
\end{prop}
\begin{proof}
It is well-known that there exists such an algebra with $W_\cL^\circ$ instead of 
$W_\cL$, see for instance \cite[\S 4]{Lus4}. It is called the graded Hecke algebra, 
over $\C[{\mb r}]$ with parameters $c_i$, and we denote it by 
$\mh H (\mf t,W_\cL^\circ,c \mb r )$.

Let $\cR_\cL^+$ be a finite central extension of $\cR_\cL$ such that the 2-cocycle
$\natural$ lifts to the trivial 2-cocycle of $\cR_\cL^+$. 
For $w^+ \in W_\cL^\circ \rtimes \cR_\cL^+$ with image $w \in W_\cL$ we put
\[
\phi_{w^+}(N_{w'} \xi) = N_{w w' w^{-1}} {}^w \xi \qquad 
w' \in W_\cL^\circ, \xi \in S(\mf t^*) \otimes \C[{\mb r}] .  
\]
Because of the condition on the $c_i$, $w^+ \mapsto \phi_{w^+}$ defines an action of
$\cR_\cL^+$ on $\mh H (\mf t,W_\cL^\circ,c \mb r)$ by algebra automorphisms. 
Thus the crossed product algebra 
\[
\cR_\cL^+ \ltimes \mh H (\mf t,W_\cL^\circ,c \mb r) = 
\C[\cR_\cL^+] \ltimes \mh H (\mf t,W_\cL^\circ,c \mb r)
\]
is well-defined. Let $p_\natural \in \C [\ker (\cR_\cL^+ \to \cR_\cL)]$ 
be the central idempotent such that 
\[
p_\natural \C [\cR_\cL^+] \cong \C [\cR_\cL,\natural] .
\]
The isomorphism is given by $p_\natural w^+ \mapsto \lambda (w^+) N_w$ for a
suitable $\lambda (w^+) \in \C^\times$. Then 
\begin{equation}\label{eq:1.10}
p_\natural \C[\cR_\cL^+] \ltimes \mh H (\mf t,W_\cL^\circ, c\mb r) \; \subset \;
\C[\cR_\cL^+] \ltimes \mh H (\mf t,W_\cL^\circ,c \mb r )
\end{equation}
is an algebra with the required relations.
\end{proof}

We denote the algebra of Proposition \ref{prop:1.4} by $\mh H (\mf t,W_\cL,c\mb r,
\natural)$. It is a special case of the algebras considered in \cite{Wit}, namely
the case where the 2-cocycle $\natural_\cL$ and the braid relations live only
on the two different factors of the semidirect product $W_\cL = 
W_\cL^\circ \rtimes \cR_\cL$.
Let us mention here some of its elementary properties.
\begin{lem}\label{lem:1.6}
$S(\mf t^*)^{W_\cL} \otimes \C[{\mb r}]$ is a central subalgebra of 
$\mh H (\mf t,W_\cL,c \mb r,\natural)$. If $W_\cL$ acts faithfully on $\mf t$, 
then it equals the centre $Z(\mh H (\mf t,W_\cL,c \mb r,\natural))$.
\end{lem}
\begin{proof}
The case $W_\cL = W_\cL^\circ$ is \cite[Theorem 6.5]{Lus3}. For $W_\cL \neq 
W_\cL^\circ$ and $\natural = 1$ see \cite[Proposition 5.1.a]{Sol1}. 
The latter argument also works if $\natural$ is nontrivial.
\end{proof}
If $V$ is a $\mh H (\mf t,W_\cL,c \mb r,\natural)$-module on which
$S(\mf t^*)^{W_\cL} \otimes \C[{\mb r}]$ acts by a character $(W_\cL x,r)$,
then we will say that the module admits the central character $(W_\cL x,r)$.

A look at the defining relations reveals that there is a unique anti-isomorphism
\begin{equation}\label{eq:1.7}
* : \mh H (\mf t,W_\cL, c\mb r,\natural) \to \mh H (\mf t,W_\cL,c \mb r,\natural^{-1}) 
\end{equation}
such that * is the identity on $S(\mf t^*) \otimes \C[{\mb r}]$ and 
$N_w^* = (N_w)^{-1}$, the inverse of the basis element $N_w \in 
\mh H (\mf t,W_\cL,c \mb r,\natural^{-1})$. Hence $\mh H (\mf t,W_\cL,c \mb r,
\natural^{-1})$ is the opposite algebra of $\mh H (\mf t,W_\cL, c \mb r,\natural)$, 
and $\mh H (\mf t,W_\cL^\circ,c \mb r)$ is isomorphic to its opposite.

Suppose that $\mf t = \mf t' \oplus \mf z$ is a decomposition of 
$W_\cL$-representations such that Lie$(Z(L) \cap G_\der) \subset \mf t'$ and 
$\mf z \subset \mf t^{W_\cL}$. Then
\begin{equation}\label{eq:1.11}
\mh H (\mf t,W_\cL,c \mb r,\natural) = 
\mh H (\mf t',W_\cL,c \mb r,\natural) \otimes_\C S(\mf z^*) . 
\end{equation}
For example, if $W_\cL = W_\cL^\circ$ we can take $\mf t' = 
\mathrm{Lie}(Z(L) \cap G_\der)$ and $\mf z = \mathrm{Lie}(Z(G))$.

Now we set out to construct $\mh H (\mf t, W_\cL,c \mb r,\natural)$ geometrically. 
In the process we will specify the parameters $c_i$ and the 2-cocycle $\natural$. \\

If $X$ is a complex variety equiped with a continuous action of $G$ and stratified by 
some algebraic stratification, we denote by $\mathcal{D}_{c}^{b}(X)$ the bounded derived 
category of constructible sheaves on $X$ and by $\mathcal{D}_{G,c}^{b}(X)$ the $G$-equivariant 
bounded derived category as defined in \cite{BeLu}. We denote by $\mathcal{P}(X)$ 
(resp. $\mathcal{P}_{G}(X)$) the category of perverse sheaves (resp. $G$-equivariant 
perverse sheaves) on $X$. Let us recall briefly how $\mathcal{D}_{G,c}^{b}(X)$ is defined. 
First, if $p\cP \rightarrow X$ is a $G$-map where $P$ is a free $G$-space and 
$q\colon P\rightarrow G \backslash P$ is the quotient map, then the category 
$\mathcal{D}_{G}^{b}(X,P)$ consists in triples 
$\mathcal{F}=(\mathcal{F}_{X},\overline{\mathcal{F}},\beta)$ with 
$\mathcal{F}_X \in \mathcal{D}^{b}(X)$, $\overline{\mathcal{F}} \in 
\mathcal{D}^{b}(G \backslash P)$, and an isomorphism 
$\beta : p^{*}\mathcal{F}_{X} \simeq q^{*}\overline{\mathcal{F}}$. 
Let $I \subset \Z$ be a segment. If $p \colon P \rightarrow X$ is an $n$-acyclic resolution 
of $X$ with $n \geqslant | I |$, then $\mathcal{D}_{G}^{I}(X)$ is defined to be 
$\mathcal{D}_{G}^{b}(X,P)$ and this does not depend on the choice of $P$. 
Finally, the $G$-equivariant derived category $\mathcal{D}_{G}^{b}(X)$ is defined as the 
limit of the categories $\mathcal{D}_{G}^{I}(X)$. Moreover, $\mathcal{P}_{G}(X)$ is the 
subcategory of $\mathcal{D}_{G}^{b}(X)$ consisting of objects $\mathcal{F}$ such that 
$\mathcal{F}_{X} \in \mathcal{P}(X)$. All the usual functors, Verdier duality, 
intermediate extension, etc., exist and are well-defined in this category. We will denote 
by $\For \colon \mathcal{D}_{G}^{b}(X) \rightarrow \mathcal{D}^{b}(X)$ the functor which 
associates to every $\mathcal{F} \in \mathcal{D}_{G}^{b}(X) $ the complex $\mathcal{F}_{X}$.\\

Consider the varieties
\begin{align*}
& \dot{\mf g} = \{ (x,gP) \in \mf g \times G/P : 
\Ad (g^{-1}) x \in \cC_v^L + \mf t + \mf u \} , \\
& \dot{\mf g}^{\circ} = \{ (x,gP) \in \mf g \times G^{\circ}/P : 
\Ad (g^{-1}) x \in \cC_v^L + \mf t + \mf u \}, \\
& \dot{\mf g}_{RS} = \{ (x,g P) \in \mf g \times G/P : 
\Ad (g^{-1}) x \in \cC_v^L + \mf t_\reg + \mf u \} , \\
& \dot{\mf g}_{RS}^{\circ} = \{ (x,g P) \in \mf g \times G^{\circ}/P : 
\Ad (g^{-1}) x \in \cC_v^L + \mf t_\reg + \mf u \} 
\end{align*}
where $\mf t_\reg = \{ x \in \mf t : Z_{\mf g}(x) = \mf l \}$.
Assume first that $\dot{\mf g}=\dot{\mf g}^{\circ}$ and so $\dot{\mf g}_{RS}=\dot{\mf g}_{RS}^{\circ}$.
Consider the maps
\begin{align*}
& \cC_v^L \xleftarrow{f_1} \{ (x,g) \in \mf g \times G : \Ad (g^{-1}) x \in 
\cC_v^L + \mf t + \mf u\} \xrightarrow{f_2} \dot{\mf g} , \\
& f_1 (x,g) = \mathrm{pr}_{\cC_v^L}(\Ad (g^{-1}) x) , \hspace{2cm} f_2 (x,g) = (x,gP) .
\end{align*}
The group $G \times P$ acts on $\{ (x,g) \in \mf g \times G : \Ad (g^{-1}) x \in 
\cC_v^L + \mf t + \mf u\}$ by 
\[
(g_1,p) \cdot (x,g) = ( \Ad (g_1)x,g_1 g p).
\]
Let $\dot{\cL}$ be the unique $G$-equivariant local system $\dot{\mf g}$ such that $f_2^* \dot{\cL} = 
f_1^* \cL$. The map
\[
\mathrm{pr}_1 : \dot{\mf g}_{RS} \to 
\mf g_{RS} := \Ad (G) (\cC_v^L + \mf t_\reg + \mf u) 
\]
is a fibration with fibre $N_G (L) / L$, so $(\pr_1)_! \dot{\cL}$ is a local 
system on $\mf g_{RS}$. Let $\cV := \Ad (G) (\overline{\cC_v^L} + \mf t + \mf u)$, 
$j : \cC_v^L \hookrightarrow \overline{\cC_v^L}$ and 
$\widehat{j} :  \dot{\mf g}_{RS} \hookrightarrow \cV $. 
Since $\cL$ is a cuspidal local system, by \cite[2.2.b)]{Lus3} it is clean, so 
$j_{!} \cL = j_{*} \cL \in \mathcal{D}_{L}^{b}(\overline{\cC_v^L})$. It follows 
(by unicity and base changes) that $\widehat{j}_{!}\dot{\cL} =
\widehat{j}_{*}\dot{\cL} \in \mathcal{D}_{G}^{b}(\widehat{\mf g}_{RS})$.
Let $K_1=\IC_{G}(\mf g_{RS}, (\pr_1)_! \dot{\cL})$ be the equivariant intersection 
cohomology complex defined by  $(\pr_1)_! \dot{\cL}$.

Considering $\pr_1$ as a map $\dot{\mf g} \to \mf g$, we get (up to a shift) a 
$G$-equivariant perverse sheaf $K = (\pr_1)_! \dot{\cL}=i_{!}K_1$ on $\mf g$, 
where $i : \cV \hookrightarrow \mf g$. Indeed, by definition it is enough to show 
that $\For(K_1) \in \mathcal{D}^{b}(\mf g)$ is perverse. But the same arguments of 
\cite[3.4]{Lus3} apply here (smallness of $\pr_1 : \dot{\mf g} \to \cV$, equivariant 
Verdier duality, etc) and the forgetful functor commutes with $(\pr_1)_!$ by \cite[3.4.1]{BeLu}.

Now, if $\dot{\mf g} \neq \dot{\mf g}^{\circ}$, then 
$\dot{\mf g}=G \times_{G^S} \dot{\mf g}^{\circ}$ where $G^S$ is the largest 
subgroup of $G$ which preserves $\dot{\mf g}^{\circ}$. Using 
\cite[5.1. Proposition (ii)]{BeLu} it follows that $K$ is a perverse sheaf. 
Notice that $(\pr_1)_! \dot{\cL}^*$ is another local system on $\mf g_{RS}$.
In the same way we construct $K_1^*$ and $K^* = (\pr_1)_! \dot{\cL}^*$.

\begin{rem}
In \cite[\S 4]{AMS} the authors consider a perverse sheaf $\pi_* \tilde{\mathcal E}$
on a subvariety $Y$ of $G^\circ$. The perverse sheaves $K$ and $K^*$ are the direct
analogues of $\pi_* \tilde{\mathcal E}$, when we apply the exponential map to replace 
$G^\circ$ by its Lie algebra $\mf g$. As Lusztig notes in \cite[2.2]{Lus3} (for
connected $G$), this allows us transfer all the results of \cite{AMS} to the 
current setting. In this paper we will freely make use of \cite{AMS} in the 
Lie algebra setting as well.
\end{rem}
In \cite[Proposition 4.5]{AMS} we showed that the $G$-endomorphism algebras of\\ 
$K = (\pr_1)_! \dot{\cL}$ and $K^* = (\pr_1)_! \dot{\cL}^*$, in the category 
$\mathcal{P}_{G}(\mf g_{RS})$ of equivariant perverse sheaves, are isomorphic to 
twisted group algebras: 
\begin{equation}\label{eq:1.1}
\begin{array}{lll}
\End_{\mathcal{P}_{G}(\mf g_{RS})} \big( (\mr{pr}_1)_! \dot{\cL} \big) & \cong & 
\C [W_\cL, \natural_\cL] ,\\
\End_{\mathcal{P}_{G}(\mf g_{RS})} \big( (\mr{pr}_1)_! \dot{\cL}^* \big) & \cong & 
\C [W_\cL, \natural_\cL^{-1}] ,
\end{array}
\end{equation}
where $\natural_\cL : ( W_\cL / W_\cL^\circ )^2 \to \C^\times$ is a 2-cocycle.
The cocycle $\natural_\cL^{-1}$ in \eqref{eq:1.1} is the inverse of $\natural_\cL$, 
necessary because we use the dual $\cL^*$. 

\begin{rem} \label{rem:1.L}
In fact there are two good ways to let \eqref{eq:1.1} act on
$(\pr_1)_! \dot{\cL}$. For the moment we subscribe to the normalization of 
Lusztig from \cite[\S 9]{Lus2}, which is based on identifying a suitable 
cohomology space with the trivial representation of $W_\cL^\circ$. However, later
we will switch to a different normalization, which identifies the same space
with the sign representation of the Weyl group $W_\cL^\circ$.
\end{rem}

According to \cite[3.4]{Lus3} this gives rise to an action of $\C [W_\cL, 
\natural_\cL^{-1}]$ on $K^*_1$ and then on $K^*$. (And similarly without duals, 
of course.) Applying the above with the group $G \times \C^\times$ and the cuspidal 
local system $\cL$ on $\cC_v^L \times \{0\} \subset \mf l \oplus \C$, we see that 
all these endomorphisms are even $G \times \C^\times$-equivariant.

Define $\End^+_{\mathcal{P}_{G}(\mf g_{RS})}((\pr_1)_! \dot{\cL})$ as the subalgebra of 
$\End_{\mathcal{P}_{G}(\mf g_{RS})}((\pr_1)_! \dot{\cL})$ which also preserves Lie$(P)$. Then
\begin{equation}\label{eq:1.E}
\begin{array}{lll}
\End^+_{\mathcal{P}_{G}(\mf g_{RS})} \big( (\mr{pr}_1)_! \dot{\cL} \big) & \cong &
\C [\cR_\cL, \natural_\cL] ,\\
\End^+_{\mathcal{P}_{G}(\mf g_{RS})} \big( (\mr{pr}_1)_! \dot{\cL}^* \big) & \cong &
\C [\cR_\cL, \natural_\cL^{-1}] ,
\end{array}
\end{equation}
The action of the subalgebra $\C [\cR_\cL,\natural_\cL^{-1}]$ on $K^*$ admits a 
simpler interpretation. For any representative $\bar{w} \in N_G (P,L)$ of $w \in
\cR_\cL$, the map $\Ad (\bar{w}) \in \Aut_\C (\mf g)$ stabilizes $\mf t = 
\mathrm{Lie}(Z(L))$ and $\mf u = \mathrm{Lie}(U) \vartriangleleft \mathrm{Lie}(P)$. 
Furthermore $\cC_v^L$ supports a cuspidal local system, so by \cite[Theorem 3.1.a]{AMS} it 
also stable under the automorphism $\Ad (\bar{w})$. Hence $\cR_\cL$ acts on $\dot{\mf g}$ by 
\begin{equation}\label{eq:1.2}
w \cdot (x,gP) = (x,g w^{-1}P) .
\end{equation}
The action of $w \in \cR_L$ on $(\dot{\mf g},\dot{\cL}^*)$ lifts \eqref{eq:1.2}, 
extending the automorphisms of $(\dot{\mf g}_{RS},\dot{\cL}^*)$ constructed in 
\cite[(44) and Proposition 4.5]{AMS}. By functoriality this induces an action 
of $w$ on $K^* = (\pr_1)_! \dot{\cL}^*$.

For $\Ad (G)$-stable subvarieties $\cV$ of $\mf g$, we define, 
as in \cite[\S 3]{Lus3},
\begin{align*}
& \dot{\cV} = \{ (x,gP) \in \dot{\mf g} : x \in \cV \} , \\
& \ddot{\cV} = \{ (x,gP,g' P) : (x,g P) \in \dot{\cV}, (x,g' P) \in \dot{\cV} \} .
\end{align*}
The two projections $\pi_{12}, \pi_{13} : \ddot{\cV} \to \dot{\cV}$ give rise 
to a $G \times \C^\times$-equivariant local system $\ddot{\cL} = \dot{\cL} 
\boxtimes \dot{\cL}^*$ on $\ddot{\cV}$. As in \cite{Lus3}, the action of 
$\C [W_\cL,\natural_\cL^{-1}]$ on $K^*$ leads to 
\begin{equation}\label{eq:1.3}
\text{actions of } \C [W_\cL,\natural_\cL] \otimes \C [W_\cL,\natural_\cL^{-1}]
\text{ on } \ddot{\cL} \text{ and on } 
H_j^{G^\circ \times \C^\times} (\ddot{\cV},\ddot{\cL}) ,
\end{equation}
denoted $(w,w') \mapsto \Delta (w) \otimes \Delta (w')$.
By \cite[Proposition 4.2]{Lus3} there is an isomorphism of graded algebras
\[
H^*_{G \times \C^\times}(\dot{\mf g}) \cong S (\mf t^* \oplus \C) =
S (\mf t^*) \otimes \C [{\mb r}] ,
\]
where $\mf t^* \oplus \C$ lives in degree 2. This algebra acts naturally on 
$H_*^{G \times \C^\times}(\dot{\cV},\dot{\cL})$ and that yields two 
actions $\Delta (\xi)$ (from $\pi_{12}$) and $\Delta' (\xi)$ (from $\pi_{13}$)
of $\xi \in S (\mf t^* \oplus \C)$ on $H_*^{G \times \C^\times}(\dot{\mf g})$.

Let $\Omega \subset G$ be a $P-P$ double coset and write
\[
\ddot{\mf g}^\Omega = \{(x,gP,g' P) \in \ddot{\mf g} : g^{-1} g' \in \Omega \} .
\]
Given any sheaf $\cF$ on a variety $\cV$, we denote
its stalk at $v \in \cV$ by $\cF_v$ or $\cF |_v$.
\begin{prop}\label{prop:1.2}
\enuma{
\item The $S(\mf t^* \oplus \C)$-module structures $\Delta$ and $\Delta'$ define
isomorphisms
\[
S(\mf t^* \oplus \C) \otimes H_0^{G \times \C^\times}(\ddot{\mf g},\ddot{\cL})
\rightrightarrows H_*^{G \times \C^\times}(\ddot{\mf g},\ddot{\cL}) .
\]
\item As $\C[W_\cL,\natural_\cL]$-modules
\[
H_0^{G \times \C^\times}(\ddot{\mf g},\ddot{\cL}) = 
\bigoplus\nolimits_{w \in W_\cL} \Delta (w) H_0^{G \times \C^\times}
(\ddot{\mf g}^P,\ddot{\cL}) \cong \C [W_\cL,\natural_\cL] .
\]
} 
\end{prop}
\begin{proof}
We have to generalize \cite[Proposition 4.7]{Lus3} to the case where $G$ 
is disconnected. We say that a $P-P$ double coset $\Omega \subset G$ 
is good if it contains an element of $N_G (L,\cL)$, and bad otherwise. Recall 
from \cite[Theorem 9.2]{Lus2} that $N_{G^\circ}(L,\cL) = N_{G^\circ}(T)$.
Let us consider $H_0^{G \times \C^\times}(\ddot{\mf g}^\Omega,\ddot{\cL})$.
\begin{itemize}
\item If $\Omega$ is good, then Lusztig's argument proves that $H_0^{G
\times \C^\times}(\ddot{\mf g}^\Omega ,\ddot{\cL}) \cong S(\mf t^* \oplus \C)$.
\item If $\Omega$ does not meet $P N_G (L) P$, then Lusztig's argument goes
through and shows that $H_0^{G \times \C^\times}
(\ddot{\mf g}^\Omega ,\ddot{\cL}) = 0$.
\item Finally, if $\Omega \subset P N_G (L) P \setminus P N_G (L,\cL) P$, we
pick any $g_0 \in \Omega$. Then \cite[p. 177]{Lus3} entails that 
\begin{multline}\label{eq:1.4}
H_0^{G \times \C^\times}(\ddot{\mf g}^\Omega,\ddot{\cL}) \cong 
H_0^{L \times \C^\times} (\cC_v^L, \cL \boxtimes \Ad (g_0)^* \cL^*) \cong \\
H_0^{Z_{L \times \C^\times}(v)} \big( \{v\}, (\cL \boxtimes \Ad (g_0)^* \cL^* 
)_v \big) \cong \\
H_0^{Z_{L \times \C^\times}(v)} ( \{v\} ) \otimes
(\cL \boxtimes \Ad (g_0)^* \cL^* )_v^{Z_{L \times \C^\times}(v)} = 0,
\end{multline}
because $\Ad (g_0)^* \cL^* \not\cong \cL^*$. 
\end{itemize}
We also note that \eqref{eq:1.4} with $P$ instead of $\Omega$ gives
\begin{multline*}
H_*^{G \times \C^\times}(\ddot{\mf g}^P,\ddot{\cL}) \cong 
H_*^{L \times \C^\times} (\cC_v^L, \cL \boxtimes \cL^*) \cong \\
H_*^{Z_{L \times \C^\times}(v)} ( \{v\} ) \otimes
(\cL \boxtimes \cL^* )_v^{Z_{L \times \C^\times}(v)} = H_*^{Z_{L \times 
\C^\times}(v)} ( \{v\} ) \otimes \End_{Z_{L \times \C^\times}(v)}(\cL_v) .
\end{multline*}
By the irreducibility of $\cL$ the right hand side is isomorphic to 
$H_*^{Z_{L \times \C^\times}(v)} ( \{v\} )$, which by \cite[p. 177]{Lus3} is
\[
S\big(\mathrm{Lie} (Z_{L \times \C^\times}(v))^* \big) = S(\mf t^* \oplus \C).
\]
In particular $H_*^{G \times \C^\times}(\ddot{\mf g}^P,\ddot{\cL})$
is an algebra contained in $H_*^{G \times \C^\times}(\ddot{\mf g},\ddot{\cL})$.

These calculations suffice to carry the entire proof of 
\cite[Proposition 4.7]{Lus3} out. It establishes (a) and
\[
\dim H_0^{G \times \C^\times}(\ddot{\mf g},\ddot{\cL}) = |W_\cL| .
\]
Then (b) follows in the same way as \cite[4.11.a]{Lus3}.
\end{proof}

The $W_\cL$-action on $T$ induces an action of $W_\cL$ on $S(\mf t^*) \otimes 
\C[{\mb r}]$, which fixes ${\mb r}$. For $\alpha$ in the root system $R(G^\circ,T)$, 
let $\mf g_\alpha \subset \mf g$ be the associated eigenspace for the $T$-action.
Let $\alpha_i \in R(G^\circ,T)$ be a simple root (with respect to $P$) and let 
$s_i \in W_\cL^\circ$ be the corresponding simple reflection. 
We define $c_i \in \Z_{\geq 2}$ by
\begin{equation}\label{eq:1.5}
\begin{array}{ccc}
\ad (v)^{c_i - 2} : & \mf g_{\alpha_i} \oplus \mf g_{2 \alpha_i} \to \mf g_{\alpha_i} 
\oplus \mf g_{2 \alpha_i} & \text{ is nonzero,} \\
\ad (v)^{c_i - 1} : & \mf g_{\alpha_i} \oplus \mf g_{2 \alpha_i} \to \mf g_{\alpha_i} 
\oplus \mf g_{2 \alpha_i} & \text{ is zero.}
\end{array}
\end{equation}
By \cite[Proposition 2.12]{Lus3} $c_i = c_j$ if $s_i$ and $s_j$ are conjugate in 
$N_G (L) / L$. According to \cite[Theorem 5.1]{Lus3}, for all $\xi \in 
S (\mf t^* \oplus \C) = S (\mf t^*) \otimes \C [{\mb r}]$:
\begin{equation}\label{eq:1.6}
\begin{array}{ccc}
\Delta (s_i) \Delta (\xi) - \Delta ({}^{s_i}\xi) \Delta (s_i) & = & 
c_i \Delta ({\mb r} (\xi - {}^{s_i}\xi) / \alpha_i) , \\
\Delta' (s_i) \Delta' (\xi) - \Delta' ({}^{s_i}\xi) \Delta' (s_i) & = & 
c_i \Delta' ({\mb r} (\xi - {}^{s_i}\xi) / \alpha_i) .
\end{array}
\end{equation}

\begin{lem}\label{lem:1.3}
For all $w \in \cR_\cL$ and $\xi \in S(\mf t^* \oplus \C)$:
\[
\begin{array}{ccc}
\Delta (w) \Delta (\xi) & = & \Delta ({}^w \xi) \Delta (w) , \\
\Delta' (w) \Delta' (\xi) & = & \Delta' ({}^w \xi) \Delta' (w) .
\end{array}
\]
\end{lem}
\begin{proof}
Recall that $\Delta (\xi)$ is given by $S(\mf t^* \oplus \C) \cong 
H^*_{G \times \C^\times}(\dot{\mf g})$ and the product in equivariant (co)homology
\[
H^*_{G \times \C^\times}(\dot{\mf g}) \otimes H_*^{G \times 
\C^\times}(\dot{\mf g} ,\dot{\cL}) \to H_*^{G \times 
\C^\times}(\dot{\mf g} ,\dot{\cL}) .
\]
As explained after \eqref{eq:1.2}, the action of $w \in \cR_\cL$ on 
$(\dot{\mf g},\dot{\cL})$ is a straightforward lift of the action \eqref{eq:1.2}
on $\dot{\mf g}$. It follows that 
\[
\Delta (w) \Delta (\xi) \Delta (w)^{-1} = \Delta ({}^{\bar w} \xi) , 
\]
where $\xi \mapsto {}^{\bar w} \xi$ is the action induced by \eqref{eq:1.2}.
Working through all the steps of the proof of \cite[Proposition 4.2]{Lus3}, we
see that this corresponds to the natural action $\xi \mapsto {}^w \xi$ of
$\cR_\cL$ on $S(\mf t^* \oplus \C)$.
\end{proof}

Let $\mh H (G,L,\cL)$ be the algebra $\mh H (\mf t,W_\cL,c \mb r,\natural_\cL)$, 
with the 2-cocycle $\natural_\cL$ and the parameters $c_i$ from \eqref{eq:1.5}.
By \eqref{eq:1.7} its opposite algebra is
\begin{equation}\label{eq:1.8}
\mh H (G,L,\cL)^{op} \cong \mh H (G,L,\cL^*) = 
\mh H (\mf t,W_\cL,c \mb r,\natural_\cL^{-1}) .
\end{equation}
Using \eqref{eq:1.E} we can interpret
\begin{equation}
\mh H (G,L,\cL) = \mh H (\mf t, W_\cL, c \mb r) \rtimes 
\End^+_{\mathcal{P}_{G}(\mf g_{RS})} \big( (pr_1 )_! \dot{\cL} \big) .
\end{equation}
\begin{lem}\label{lem:1.7}
With the above interpretation $\mh H (G,L,\cL)$ is determined uniquely by
$(G,L,\cL)$, up to canonical isomorphisms.
\end{lem}
\begin{proof}
The only arbitrary choices are $P$ and $\natural_\cL : \cR_\cL^2 \to \C^\times$.

A different choice of a parabolic subgroup $P' \subset G$ with Levi factor $L$
would give rise to a different algebra $\mh H (G,L,\cL)'$. However, Lemma \ref{lem:1.1}.a
guarantees that there is a unique (up to $P$) element $g \in G^\circ$ with
$g P g^{-1} = P'$. Conjugation with $g$ provides a canonical isomorphism between the
two algebras under consideration.

The 2-cocycle $\natural_\cL$ depends on the choice of elements $N_\gamma \in
\End^+_{\mathcal{P}_{G}(\mf g_{RS})} \big( (pr_1 )_! \dot{\cL} \big)$. This choice is not
canonical, only the cohomology class of $\natural_\cL$ is uniquely determined.
Fortunately, this indefiniteness drops out when we replace $\C [\cR_\cL,\natural_\cL]$
by $\End^+_{\mathcal{P}_{G}(\mf g_{RS})} \big( (pr_1 )_! \dot{\cL} \big)$. Every element of
$C^\times N_\gamma \subset \End^+_{\mathcal{P}_{G}(\mf g_{RS})} \big( (pr_1 )_! \dot{\cL} \big)$
has a well-defined conjugation action on $\mh H (\mf t, W_\cL,c \mb r)$, depending
only on $\gamma \in \cR_\cL$. This suffices to define the crossed product
$\mh H (\mf t, W_\cL, c \mb r) \rtimes \End^+_{\mathcal{P}_{G}(\mf g_{RS})} 
\big( (pr_1 )_! \dot{\cL} \big)$ in a canonical way.
\end{proof}

The group $W_\cL$ and its 2-cocycle $\natural_\cL$ from \cite[\S 4]{AMS} can 
be constructed using only the finite index subgroup 
$G^\circ N_G (P,\cL) \subset G$. Hence
\begin{equation}\label{eq:1.12}
\mh H (G,L,\cL) = \mh H (G^\circ N_G (P,\cL),L,\cL) . 
\end{equation}
With \eqref{eq:1.3}, \eqref{eq:1.6} and Lemma \ref{lem:1.3} we can define
endomorphisms $\Delta (h)$ and $\Delta' (h')$ of $H_*^{G \times \C^\times}
(\ddot{\mf g},\ddot{\cL})$ for every $h \in \mh H (G,L,\cL)$ and every 
$h' \in \mh H (G,L,\cL^*)$.

Let $1 \in H_0^{G \times \C^\times}(\ddot{\mf g}^P,\ddot{\cL}) \cong
S(\mf t^* \oplus \C)$ be the unit element.

\begin{cor}\label{cor:1.5}
\enuma{
\item The map $\mh H (G,L,\cL) \to H_*^{G \times \C^\times}
(\ddot{\mf g},\ddot{\cL}) : h \mapsto \Delta (h) 1$ is bijective.
\item The map $\mh H (G,L,\cL^*) \to H_*^{G \times \C^\times}
(\ddot{\mf g},\ddot{\cL}) : h' \mapsto \Delta' (h') 1$ is bijective.
\item The operators $\Delta (h)$ and $\Delta' (h')$ commute, and
$(h,h') \mapsto \Delta (h) \Delta' (h')$ identifies $H_*^{G \times \C^\times} 
(\ddot{\mf g},\ddot{\cL})$ with the biregular representation of $\mh H (G,L,\cL)$.
} 
\end{cor}
\begin{proof}
This follows in the same way as \cite[Corollary 6.4]{Lus3}, when we take
Proposition \ref{prop:1.2} and \eqref{eq:1.8} into account.
\end{proof}

\section{Representations of twisted graded Hecke algebras}
\label{sec:reps}

We will extend the construction and parametrization of $\mh H(G,L,\cL)$-modules
from \cite{Lus3,Lus5} to the case where $G$ is disconnected. In this section will 
work under the following assumption:

\begin{cond}\label{cond:1}
The group $G$ equals $N_G (P,\cL) G^\circ$. 
\end{cond}
In view of \eqref{eq:1.12} this does not pose any restriction on the collection
of algebras that we consider.

\subsection{Standard modules} \
\label{par:standard}

Let $y \in \mf g$ be nilpotent and define
\[
\cP_y = \{ gP \in G / P : \Ad (g^{-1}) y \in \cC_v^L + \mf u \} . 
\]
The group 
\[
M(y) = \{ (g_1,\lambda) \in G \times \C^\times : \Ad (g_1) y = \lambda^2 y \} 
\]
acts on $\cP_y$ by $(g_1,\lambda) \cdot gP = g_1 g P$.
Clearly $\cP_y$ contains an analogous variety for $G^\circ$:
\[
\cP_y^\circ := \{ gP \in G^\circ / P : \Ad (g^{-1}) y \in \cC_v^L + \mf u \} .
\]
Since $\cC_v^L$ is stable under $\Ad (N_G (L))$, $\cC_v^L + \mf u$ is stable under
$\Ad (N_G (P))$. As $N_{G^\circ}(P) = P$ and $N_G (P,\cL) P / P \cong \cR_\cL$,
there is an isomorphism of $M(y)$-varieties
\begin{equation}\label{eq:2.1}
\cP_y^\circ \times \cR_\cL \to \cP_y : (gP,w) \mapsto g w^{-1} P .
\end{equation}
The local system $\dot{\cL}$ on $\dot{\mf g}$ restricts to a local system on
$\cP_y \cong \{y\} \times \cP_y \subset \dot{\mf g}$. We will endow the space
\begin{equation}\label{eq:2.2}
H_*^{M(y)^\circ}(\cP_y,\dot{\cL})
\end{equation}
with the structure of an $\mh H (G,L,\cL)$-module. With the method of 
\cite[p. 193]{Lus3}, the action of $\C [W_\cL,\natural_\cL^{-1}]$
on $K^*$ from \eqref{eq:1.1} gives rise to an action $\tilde \Delta$ on
the dual space of \eqref{eq:2.2}. With the aid of \eqref{eq:1.7}, the map 
\begin{equation}\label{eq:2.68}
\Delta : \C [W_\cL,\natural_\cL] \to \End_\C \big( H_*^{M(y)^\circ} (\cP_y,
\dot{\cL}) \big), \quad \Delta (N_w) = \tilde{\Delta} \big( (N_w)^{-1} \big)^* 
\end{equation}
makes \eqref{eq:2.2} into a graded $\C [W_\cL,\natural_\cL]$-module.

We describe the action of $S(\mf t^* \oplus \C) \cong H^*_{G \times 
\C^\times}(\dot{\mf g})$ in more detail. The inclusions
\[
\{y\} \times \cP_y \subset (G \times \C^\times)\cdot (\{y\} \times \cP_y)
\subset \dot{\mf g} 
\]
give maps
\begin{equation}\label{eq:2.3}
H^*_{G \times \C^\times}(\dot{\mf g}) \to H^*_{G \times \C^\times}
(G \times \C^\times \cdot \{y\} \times \cP_y) \to
H^*_{M(y)}(\cP_y) .
\end{equation}
Here $(G \times \C^\times)\cdot (\{y\} \times \cP_y) \cong
(G \times \C^\times) \times_{M(y)} \cP_y$, so by \cite[1.6]{Lus3} the
second map in \eqref{eq:2.3} is an isomorphism. 
Recall from \cite[1.9]{Lus3} that
\[
H^*_{M(y)}(\cP_y) \cong H^*_{M(y)^\circ}(\cP_y)^{M(y) / M(y)^\circ} . 
\]
The product
\begin{equation}\label{eq:2.9}
H^*_{M(y)^\circ}(\cP_y) \otimes H_*^{M(y)^\circ}(\cP_y,\dot{\cL}) \to 
H_*^{M(y)^\circ}(\cP_y,\dot{\cL}) 
\end{equation}
gives an action of the graded algebras in \eqref{eq:2.3} on the graded 
vector space $H_*^{M(y)^\circ}(\cP_y,\dot{\cL})$. We denote the operator 
associated to $\xi \in S(\mf t^* \oplus \C)$ by $\Delta (\xi)$.

The projection $\{y\} \times \cP_y \to \{y\}$ induces an algebra homomorphism
$H^*_{M(y)^\circ}(\{y\}) \to H^*_{M(y)^\circ}(\cP_y)$. With \eqref{eq:2.9}
this also gives an action of $H^*_{M(y)^\circ}(\{y\})$ on $H_*^{M(y)^\circ}
(\cP_y,\dot{\cL})$. Furthermore $M(y)$ acts naturally on $H^*_{M(y)^\circ}(\{y\})$ 
and on $H_*^{M(y)^\circ}(\cP_y,\dot{\cL})$, and this actions factors through 
the finite group $\pi_0 (M(y)) = M(y) / M(y)^\circ$.

\begin{thm}\label{thm:2.4}[Lusztig] 
\enuma{
\item The above operators $\Delta (w)$ and $\Delta (\xi)$ make 
$H_*^{M(y)^\circ}(\cP_y,\dot{\cL})$ into a graded $\mh H (G,L,\cL)$-module.
\item The actions of $H^*_{M(y)^\circ}(\{y\})$ and $\mh H (G,L,\cL)$ commute.
\item $H_*^{M(y)^\circ}(\cP_y,\dot{\cL})$ is finitely generated and projective as 
 $H^*_{M(y)^\circ}(\{y\})$-module.
\item The action of $\pi_0 (M(y))$ commutes with the $\mh H (G,L,\cL)$-action.
It is semilinear with respect to $H^*_{M(y)^\circ}(\{y\})$, that is,
for $m \in \pi_0 (M(y)), \mu \in H^*_{M(y)^\circ}(\{y\})$, 
$h \in \mh H (G,L,\cL)$ and $\eta \in H_*^{M(y)^\circ} (\cP_y,\dot{\cL})$:
\[
m \cdot (\mu \otimes \Delta (h) \eta ) = (m \cdot \mu) \otimes \Delta (h) 
(m \cdot \eta) = \Delta (h) \big( (m \cdot \mu) \otimes (m \cdot \eta) \big) .
\]
}
\end{thm}
\begin{proof}
(b) The actions of $S(\mf t^* \oplus \C)$ and $H^*_{M(y)^\circ}(\{y\})$ both come
from \eqref{eq:2.9}. The algebra $H^*_{M(y)^\circ}(\cP_y)$ is graded commutative
\cite[1.3]{Lus3}. However, since $H^{M(y)^\circ}_j (\cP_y,\dot{\cL}) = 0$ for
odd $j$ \cite[Propostion 8.6.a]{Lus3}, only the action of the subalgebra
$H_{M(y)^\circ}^{\text{even}}(\cP_y)$ matters. Since this is a commutative algebra,
the actions of $S(\mf t^* \oplus \C)$ and $H^*_{M(y)^\circ}(\{y\})$ commute.

Write $\tilde{\cO} = (G \times \C^\times) / M(y)^\circ$ and define
\[
h : \tilde{\cO} \to \mf g, \quad (g,\lambda) \mapsto \lambda^{-2} \Ad (g) y .
\]
There are natural isomorphisms
\begin{align*}
& H^*_{M(y)^\circ}(\{y\}) \cong H^*_{G \times \C^\times}(\tilde{\cO}) ,\\
& H_j^{M(y)^\circ}(\cP_y, \dot{\cL})^* \cong 
H_{G \times \C^\times}^{2\dim (\tilde{\cO}) - j} (\tilde{\cO}, h^* K^*) .
\end{align*}
The dual of the action of $H^*_{M(y)^\circ}(\{y\})$ on 
$H_*^{M(y)^\circ}(\cP_y,\dot{\cL})$ becomes the product
\[
H^*_{G \times \C^\times}(\tilde{\cO}) \otimes 
H_{G \times \C^\times}^* (\tilde{\cO}, h^* K^*) \to
H_{G \times \C^\times}^* (\tilde{\cO}, h^* K^*) .
\]
From the proof of \cite[4.4]{Lus3} one sees that this action commutes with the 
operators $\tilde{\Delta}(w)$. Hence the $\Delta (w)$ also commute with the  
$H^*_{M(y)^\circ}(\{y\})$-action.\\
(c) See \cite[Proposition 8.6.c]{Lus3}.\\
(d) The semilinearity is a consequence of the functoriality of the product in 
equivariant homology. Since the action of $S(\mf t^* \oplus \C)$ factors via 
\[
H_{M(y)}^* (\cP_y) \cong H^*_{M(y)^\circ}(\cP_y)^{\pi_0 (M(y))} ,
\] 
it commutes with the action of $\pi_0 (M(y))$ on 
$H_*^{M(y)^\circ}(\cP_y,\dot{\cL})$. 

The algebra $\C[W_\cL,\natural_\cL^{-1}]$ acts on $(\mf g,K^*)$ and on 
$(\tilde{\cO},h^* K^*)$ by $G \times \C^\times$-equivariant endomorphisms. 
In other words, the operators $\tilde \Delta (w)$ on $H_{G \times \C^\times}^* 
(\tilde{\cO}, h^* K^*)$ commute with the natural action of $M(y) \subset G \times 
\C^\times$. Consequently the operators $\Delta (w)$ on $H_{G \times \C^\times}^* 
(\tilde{\cO}, h^* K^*)^* \cong H_*^{M(y)^\circ}(\cP_y,\dot{\cL})$
commute with the action of $M(y)$.\\
(a) For $G = G^\circ$ this is \cite[Theorem 8.13]{Lus3}. That proof also works if
$G$ is disconnected. We note that it uses parts (b), (c) and (d).
\end{proof}

In the same way $H_*^{M(y)^\circ}(\cP_y^\circ,\dot{\cL})$ becomes a 
$\mh H (G^\circ,L,\cL)$-module.

\begin{lem}\label{lem:2.1}
There is an isomorphism of $\mh H (G,L,\cL)$-modules
\[
H_*^{M(y)^\circ}(\cP_y,\dot{\cL}) \cong \ind_{\mh H (G^\circ,L,\cL)}^{\mh H 
(G,L,\cL)} H_*^{M(y)^\circ}(\cP_y^\circ,\dot{\cL}) 
\]
\end{lem}
\begin{proof}
Recall from \eqref{eq:1.10} that
\[
\mh H (G,L,\cL) = \C [\cR_\cL,\natural_\cL] \ltimes \mh H (G^\circ,L,\cL) . 
\]
It follows from \eqref{eq:2.1} that
\[
H_j^{M(y)^\circ} (\cP_y,\dot{\cL}) = \bigoplus\nolimits_{\gamma \in \cR_\cL} 
H_j^{M(y)^\circ} (\cP_y^\circ \gamma^{-1},\dot{\cL}) .
\]
In \eqref{eq:1.2} we saw that the action of $\C [\cR_\cL,\natural]$ on 
$(\dot{\mf g},\dot{\cL})$ lifts the action
\[
w \cdot (x,gP) = (x,g w^{-1}P) \quad w \in \cR_\cL, (x,gP) \in \dot{\mf g} .  
\]
Hence, for all $w,r \in \cR_\cL$:
\begin{equation}\label{eq:2.4}
\Delta (w) H_j^{M(y)^\circ} (\cP_y^\circ \gamma^{-1},\dot{\cL}) =
H_j^{M(y)^\circ} (\cP_y^\circ \gamma^{-1} w^{-1},\dot{\cL}) .
\end{equation}
Therefore the action map
\[
\C [\cR_\cL,\natural_\cL] \underset{\C}{\otimes} H_*^{M(y)^\circ} 
(\cP_y^\circ,\dot{\cL}) = \mh H (G,L,\cL) \underset{\mh H (G^\circ,L,\cL)}{\otimes}
H_*^{M(y)^\circ} (\cP_y^\circ,\dot{\cL}) \to H_*^{M(y)^\circ} (\cP_y,\dot{\cL})
\]
is an isomorphism of $\mh H (G,L,\cL)$-modules.
\end{proof}

From the natural isomorphism 
\[
H^*_{M(y)^\circ}(\{y\}) \cong \cO (\mathrm{Lie}(M(y)^\circ) )^{M(y)^\circ}
\]
one sees that the left hand side is the coordinate ring of the variety $V_y$ 
of semisimple adjoint orbits in
\[
\mathrm{Lie}(M(y)^\circ) = 
\{ (\sigma,r) \in \mf g \oplus \C : [\sigma,y] = 2 r y \} .
\]
For any $(\sigma,r) /\!\!\sim \;\in V_y$ let $\C_{\sigma,r}$ be the one-dimensional
$H_*^{M(y)^\circ}(\{y\})$-module obtained by evaluating functions at the
$\Ad (M(y)^\circ$-orbit of $(\sigma,r)$. We define
\begin{align*}
& E_{y,\sigma,r} = \C_{\sigma,r} \underset{H_*^{M(y)^\circ}(\{y\})}{\otimes}
H_*^{M(y)^\circ}(\cP_y,\dot{\cL}) , \\
& E_{y,\sigma,r}^\circ = \C_{\sigma,r} \underset{H_*^{M(y)^\circ}(\{y\})}{\otimes}
H_*^{M(y)^\circ}(\cP_y^\circ,\dot{\cL}) .
\end{align*}
These are $\mh H (G,L,\cL)$-modules (respectively $\mh H (G^\circ,L,\cL)$-modules).
In general they are reducible and not graded (in contrast with Theorem \ref{thm:2.4}.a).
These modules, and those in Lemma \ref{lem:2.1}, are compatible with parabolic induction
in a sense which we will describe next. 

Let $Q \subset G$ be an algebraic subgroup
such that $Q \cap G^\circ$ is a Levi subgroup of $G^\circ$ and $L \subset Q^\circ =
Q \cap G^\circ$. Assume that $y \in \mf q = \Lie (Q)$. Let $\mc P_y^Q$ and 
$\mc P_y^{Q^\circ}$ be the versions of $\mc P_y$ for $Q$ and $Q^\circ$. 
The role of $P$ is now played by $P \cap Q$. There is a natural map
\begin{equation}\label{eq:2.54}
\mc P_y^Q \to \mc P_y : g (P \cap Q) \mapsto g P . 
\end{equation}
By \cite[1.4.b]{Lus3} it induces, for every $n \in \Z$, a map
\begin{equation}\label{eq:2.67}
H^{M(y)^\circ}_{n + 2 \dim \mc P_y^Q}(\mc P_y^Q,\dot{\cL}) \; \to \;
H^{M(y)^\circ}_{n + 2 \dim \mc P_y}(\mc P_y,\dot{\cL}) .
\end{equation}

\begin{thm}\label{thm:2.24}
Let $Q$ and $y$ be as above, and let $C$ be a maximal torus of $M^Q (y)^\circ$.
\enuma{
\item The map \eqref{eq:2.54} induces an isomorphism of $\mh H (G,L,\cL)$-modules
\[
\mh H (G,L,\cL) \underset{\mh H (Q,L,\cL)}{\otimes} H_*^C (\mc P_y^Q,\dot{\cL})
\to H_*^C (\mc P_y,\dot{\cL}) ,
\]
which respects the actions of $H_C^* (\{y\})$.
\item Let $(\sigma,r) / \!\! \sim \in V_y^Q$. The map \eqref{eq:2.54} induces an 
isomorphism of $\mh H (G,L,\cL)$-modules
\[
\mh H (G,L,\cL) \underset{\mh H (Q,L,\cL)}{\otimes} E^Q_{y,\sigma,r} 
\to E_{y,\sigma,r} ,
\]
which respects the actions of $\pi_0 (M^Q (y))_\sigma$.
}
\end{thm}
\textbf{Erratum.} Unfortunately the above theorem is incorrect. The map in part (a)
is usually not surjective, and for part (b) we need an extra condition $\epsilon (\sigma,r)
\neq 0$ or $r=0$. This condition holds for almost all parameters, see the appendix.
\begin{proof}
(a) It was noted in \cite[1.16]{Lus7} that the map of the theorem is well-defined,
$\mh H (G,L,\cL)$-linear and $H_C^* (\{y\})$-linear. 

Let us consider the statement for $G^\circ$ and $Q^\circ$ first. In \cite[\S 2]{Lus7} 
a $C$-variety $\dot{\mc A}$, which contains $\mc P_y^\circ$, is studied. 
Consider the diagram of $\mh H (G^\circ,L,\cL)$-modules
\begin{equation}\label{eq:2.55}
\xymatrix{
& \mh H (G^\circ,L,\cL) \underset{\mh H (Q^\circ,L,\cL)}{\otimes} 
H_*^C (\mc P_y^{Q^\circ},\dot{\cL}) \ar[dl] \ar[dr] \\
H_*^C (\mc P_y^\circ,\dot{\cL}) \ar[rr] & & H_*^C (\dot{\mc A},\dot{\cL}) 
}
\end{equation}
with maps coming from the theorem, from $\mc P_y^\circ \to \dot{\mc A}$ and from
\cite[2.15.(c)]{Lus7}. According to \cite[2.19]{Lus7} the diagram commutes, and by 
\cite[2.8.(g)]{Lus7} the horizontal map is injective. Moreover 
\cite[Theorem 2.16]{Lus7} says that the right slanted map is an isomorphism
of $\mh H (G^\circ,L,\cL)$-modules. Consequently the horizontal map of \eqref{eq:2.55}
is surjective as well, and the entire diagram consists of isomorphisms. 

Combining this result with Lemma \ref{lem:2.1}, we get isomorphisms
\begin{align*}
& \mh H (G,L,\cL) \underset{\mh H (Q,L,\cL)}{\otimes} H_*^C (\mc P_y^Q,\dot{\cL}) \cong \\
& \mh H (G,L,\cL) \underset{\mh H (Q,L,\cL)}{\otimes} \mh H (Q,L,\cL) 
\underset{\mh H (Q^\circ,L,\cL)}{\otimes} H_*^C (\mc P_y^{Q^\circ},\dot{\cL}) \cong \\
& \mh H (G,L,\cL) \underset{\mh H (G^\circ,L,\cL)}{\otimes} \mh H (G^\circ,L,\cL) 
\underset{\mh H (Q^\circ,L,\cL)}{\otimes} H_*^C (\mc P_y^{Q^\circ},\dot{\cL}) \cong \\
& \mh H (G,L,\cL) \underset{\mh H (G^\circ,L,\cL)}{\otimes} 
H_*^C (\mc P_y^\circ,\dot{\cL}) \qquad \cong \qquad H_*^C (\mc P_y,\dot{\cL}) .
\end{align*}
(b) Since $(\sigma,r) \in \Lie (M^Q (y))$ is semisimple, we may assume that
$(\sigma,r) \in \Lie (C)$. By \cite[Proposition 7.5]{Lus3} there exist natural
isomorphisms
\begin{multline*}
\C_{\sigma,r} \underset{H_C^* (\{y\})}{\otimes} H_*^C (\mc P_y, \dot{\cL}) \cong 
\C_{\sigma,r} \underset{H_C^* (\{y\})}{\otimes} H^*_C (\{y\}) \underset{
H_{M(y)^\circ}^* (\{y\})}{\otimes} H_*^{M (y)^\circ} (\mc P_y, \dot{\cL}) \cong \\
\C_{\sigma,r} \underset{H_{M(y)^\circ}^* (\{y\})}{\otimes} H_*^{M (y)^\circ} 
(\mc P_y, \dot{\cL}) = E_{y,\sigma,r} .
\end{multline*}
The actions of $\mh H (G,L,\cL)$ and $H_C^* (\{y\})$ commute, so we also get 
\begin{multline*}
\C_{\sigma,r} \underset{H_C^* (\{y\})}{\otimes} \mh H (G,L,\cL) 
\underset{\mh H (Q,L,\cL)}{\otimes} H_*^C (\mc P_y^Q,\dot{\cL}) \cong \\
\mh H (G,L,\cL) \underset{\mh H (Q,L,\cL)}{\otimes} \C_{\sigma,r} 
\underset{H_C^* (\{y\})}{\otimes} H_*^C (\mc P_y^Q,\dot{\cL}) =
\mh H (G,L,\cL) \underset{\mh H (Q,L,\cL)}{\otimes} E^Q_{y,\sigma,r} .
\end{multline*}
Now we can apply part (a) to obtain the desired isomorphism. Since the map 
\eqref{eq:2.54} is $M^Q (y)$-equivariant, this isomorphism preserves the 
$\pi_0 (M^Q (y))_\sigma$-actions.
\end{proof}

It is possible to choose an algebraic homomorphism $\gamma_y : \SL_2 (\C) \to
G^\circ$ with d$\gamma_y \matje{0}{1}{0}{0} = y$. It will turn out that often it
is convenient to consider the element
\begin{equation}\label{eq:2.52}
\sigma_0 := \sigma + \textup{d}\gamma_y \matje{-r}{0}{0}{r} \in Z_{\mf g}(y) . 
\end{equation}
instead of $\sigma$.

\begin{prop}\label{prop:2.3}
Assume that $\cP_y$ is nonempty. 
\enuma{ 
\item $\Ad (G)(\sigma) \cap \mf t$ is a single $W_\cL$-orbit in $\mf t$.
\item The $\mh H (G,L,\cL)$-module $E_{y,\sigma,r}$ admits the central character
$(\Ad (G)(\sigma) \cap \mf t,r) \in \mf t / W_\cL \times \C$. 
\item The pair $(y,\sigma)$ is $G^\circ$-conjugate to one with $\sigma, \sigma_0$
and $\textup{d}\gamma_y \matje{-1}{0}{0}{1}$ all three in $\mf t$.
}
\end{prop}
\begin{proof}
(a) and (b) According to \cite[8.13.a]{Lus5} there is a canonical surjection
\begin{equation}\label{eq:2.8}
H^*_{G^\circ \times \C^\times}(\text{point}) \cong \cO (\mf g \oplus \C)^{G^\circ 
\times \C^\times} = \cO (\mf g)^{G^\circ} \otimes \C[{\mb r}] \to 
Z (\End_{D_{G^\circ \times \C^\times}(\mf g)} (K^*)).
\end{equation}
By \cite[Theorem 8.11]{Lus5} the endomorphism algebra of $K^*$, in the category 
of $G^\circ \times \C^\times$-equivariant perverse sheaves on $\mf g$, is 
canonically isomorphic to $\mh H (G^\circ,L,\cL)$. Together with Lemma 
\ref{lem:1.6} it follows that the right hand side of \eqref{eq:2.8} is
\[
Z( \mh H (G^\circ,L,\cL)) \cong S(\mf t^*)^{W_\cL^\circ} \otimes \C [{\mb r}] . 
\]
By \cite[8.13.b]{Lus5} the surjection \eqref{eq:2.8} corresponds to an injection
\[
\mf t / W_\cL^\circ \to \Irr (\cO (\mf g)^{G^\circ}) , 
\]
where the right hand side is the variety of semisimple adjoint orbits in $\mf g$.
Hence $\Ad (G^\circ) (\sigma) \cap \mf t$ is either empty or a single 
$W_\cL^\circ$-orbit. By Condition \ref{cond:1} $G / G^\circ \cong W_\cL / 
W_\cL^\circ$, so all these statements remain valid if we replace $G^\circ$ by $G$.

The action of $S(\mf t^*)^{W_\cL^\circ} \otimes \C [{\mb r}]$ on $E_{y,\sigma,r}$ 
can be realized as
\[
H^*_{G \times \C^\times}(\text{point}) \to H^*_{M(y)^\circ}(\{y\}) \to
H^*_{M(y)^\circ}(\cP_y)
\]
and then the product \eqref{eq:2.9}. By construction $H^*_{M(y)^\circ}(\{y\})$ 
acts on $E_{y,\sigma,r}$ via the character $(\sigma,r)/\!\! \sim \; \in V_y$. 
Hence $H^*_{G \times \C^\times}(\text{point})$ acts via the character 
$\Ad (G \times \C^\times) (\sigma,r)$.

The assumption $\cP_y \neq \emptyset$ implies that $H_*^{M(y)^\circ}(\cP_y, 
\dot{\cL})$ is nonzero. By Theorem \ref{thm:2.4}.c, and because $V_y$ is an 
irreducible variety, $E_{y,\sigma,r} \neq 0$ for all $(\sigma,r)/\!\!\sim \; 
\in V_y$. Thus the above determines a unique character of $Z(\mh H (G,L,\cL))$ 
via \eqref{eq:2.8}, which must be $(\Ad (G)(\sigma) \cap \mf t,r)$. 
In particular the intersection is nonempty and constitutes one $W_\cL$-orbit.\\
(c) By part (b) with $r=0$ we may assume that $\sigma_0 \in \mf t$. Then $M$ is
contained in the reductive group $Z_G(\sigma_0)$, so we can arrange that
the image of $\gamma_y$ lies in $Z_G (\sigma_0)$. Applying part (b) to this
group, with $r \neq 0$, we see that there exists a $g \in Z_{G^\circ}(\sigma_0)$ 
such that 
\[
\Ad (g) \sigma = \sigma_0 + \Ad (g) \textup{d}\gamma_y \matje{1}{0}{0}{-1}
\text{ lies in } \mf t .
\]
Now the pair $(\Ad (g) y, \Ad (g) \sigma)$ has the required properties.
\end{proof}

Let $\pi_0 (M(y))_\sigma$ be the stabilizer of $(\sigma,r)/\!\!\sim \; \in V_y$ 
in $\pi_0 (M(y))$. (It does not depend on $r$ because $\C^\times$ is central in 
$G \times \C^\times$.) It follows from Theorem \ref{thm:2.4}.d that 
$\pi_0 (M(y))_\sigma$ acts on $E_{y,\sigma,r}$ by $\mh H (G,L,\cL)$-module 
homomorphisms. Similarly, let $\pi_0 (M)_\sigma^\circ$ be the stabilizer of 
$(\sigma,r)/\!\!\sim$ in $\pi_0 (M(y) \cap G^\circ)$. It acts on 
$E^\circ_{y,\sigma,r}$ by $\mh H (G^\circ,L,\cL)$-module maps. 
To analyse these components groups we use \eqref{eq:2.52}.

\begin{lem}\label{lem:2.2}
\enuma{
\item There are natural isomorphisms
\[
\pi_0 (M(y))_\sigma \cong \pi_0 (Z_G (\sigma,y)) \cong \pi_0 (Z_G (\sigma_0,y)) . 
\]
\item Fix $r \in \C$. The map $\sigma \mapsto \sigma_0$ and part (a) induce a 
bijection between
\begin{itemize}
\item $G$-conjugacy classes of triples $(y,\sigma,\rho)$ with $y \in \mf g$ 
nilpotent,\\ $(\sigma,r) \in \mathrm{Lie}(M(y))$ semisimple and 
$\rho \in \Irr (\pi_0 (M(y))_\sigma)$;
\item $G$-conjugacy classes of triples $(y,\sigma_0,\rho)$ with $y \in \mf g$ 
nilpotent,\\ $\sigma_0 \in \mf g$ semisimple, $[\sigma_0,y] = 0$ and 
$\rho \in \Irr (\pi_0 (M(y))_{\sigma_0})$.
\end{itemize}
}
\end{lem}
\textbf{Remark.} Via the Jordan decomposition the second set in part (b) is 
canonically in bijection with the $G$-orbits of pairs $(x,\rho)$ where 
$x \in \mf g$ and $\rho \in \pi_0 (Z_G (x))$. Although that is a more elegant 
description we prefer to keep the semisimple and nilpotent parts separate, 
because only the $(y,\sigma_0)$ with $\cP_y \neq \emptyset$ are relevant for 
$\mh H (G,L,\cL)$.
\begin{proof}
(a) By definition
\[
\pi_0 (M(y))_\sigma = \Stab_{\pi_0 (M(y))}(\Ad (M(y)^\circ) (\sigma,y)) \cong
Z_{M(y)}(\sigma,r) / Z_{M(y)^\circ}(\sigma,r) .
\]
Since $(\sigma,r)$ is a semisimple element of Lie$(G \times \C^\times)$, 
taking centralizers with $(\sigma,r)$ preserves connectedness. 
Hence the right hand side is
\begin{equation}\label{eq:2.6}
\big( Z_{G \times \C^\times}(\sigma,r) \cap M(y) \big) / 
\big( Z_{G \times \C^\times}(\sigma,r) \cap M(y) \big)^\circ = 
\pi_0 \big( Z_{G \times \C^\times}(\sigma,r) \cap M(y) \big) .
\end{equation}
We note that $Z_{G \times \C^\times}(\sigma,r) = Z_G (\sigma) \times \C^\times$
and that there is a homeomorphism
\[
Z_G (y) \times \C^\times \to M(y) : 
(g,\lambda) \mapsto g \gamma_y \matje{\lambda}{0}{0}{\lambda^{-1}} .
\]
It follows that the factor $\C^\times$ can be omitted from \eqref{eq:2.6} 
without changing the quotient, and we obtain
\[
\pi_0 (M(y))_\sigma \cong Z_G (\sigma,y) / Z_G (\sigma,y)^\circ = 
\pi_0 (Z_G (\sigma,y)) . 
\]
By \cite[\S 2.4]{KaLu} the inclusion maps
\[
Z_G (\sigma,y) \leftarrow Z_G \big( \sigma,\textup{d}\gamma_y 
(\mf{sl}_2 (\C))\big) \to Z_G (\sigma_0,y)
\]
induce isomorphisms on component groups.\\
(b) Again by \cite[\S 2.4]{KaLu}, the $Z_G (y)$-conjugacy class of $\sigma_0$ 
is uniquely determined by $\sigma$. The reason is that the homomorphism 
d$\gamma_y : \mf{sl}_2 (\C) \to \mf g$ is unique up to the adjoint action of 
$Z_G (y)$. By the same argument $\sigma_0$ determines the $Z_G (y)$-adjoint 
orbit of $\sigma$. Thus $\sigma \mapsto \sigma_0$ gives a bijection between 
adjoint orbits of pairs $(\sigma,y)$ and of pairs $(\sigma_0,y)$. The 
remainder of the asserted bijection comes from part (a).
\end{proof}

Applying Lemma \ref{lem:2.2} with $G^\circ$ instead of 
$G$ gives natural isomorphisms
\begin{equation}\label{eq:2.7}
\pi_0 (M(y))^\circ_\sigma \cong \pi_0 (Z_{G^\circ} (\sigma,y)) \cong 
\pi_0 (Z_{G^\circ} (\sigma_0,y)) .  
\end{equation}
For $\rho \in \Irr (\pi_0 (M(y))_\sigma)$ and 
$\rho^\circ \in \Irr (\pi_0 (M(y))_\sigma^\circ)$ we write
\begin{align*}
& E_{y,\sigma,r,\rho} = \Hom_{\pi_0 (M(y))_\sigma}(\rho,E_{y,\sigma,r}) , \\
& E^\circ_{y,\sigma,r,\rho^\circ} = 
\Hom_{\pi_0 (M(y))^\circ_\sigma}(\rho^\circ,E^\circ_{y,\sigma,r}) .
\end{align*}
It follows from Theorem \ref{thm:2.4}.d that these vector spaces are modules for
$\mh H (G,L,\cL)$, respectively for $\mh H (G^\circ,L,\cL)$. When they are
nonzero, we call them standard modules.

Recall the cuspidal support map $\Psi_G$ from \cite{Lus2,AMS}. It associates 
a cuspidal support $(L',\cC_{v'}^{L'},\cL')$ to every pair $(x,\rho)$ with 
$x \in \mf g$ nilpotent and $\rho \in \Irr (\pi_0 (Z_G (x)))$.

\begin{prop}\label{prop:2.5}
The $\mh H (G^\circ,L,\cL)$-module $E^\circ_{y,\sigma,r,\rho^\circ}$ is nonzero 
if and only if \\ $\Psi_{ Z_{G^\circ}(\sigma_0)} (y,\rho^\circ)$ is 
$G^\circ$-conjugate to $(L,\cC_v^L,\cL)$. Here $\rho^\circ$ is considered as 
an irreducible representation of $\pi_0 (Z_{Z_{G^\circ}(\sigma_0)}(y))$ 
via Lemma \ref{lem:2.2}.
\end{prop}
\begin{proof}
Assume first that $r \neq 0$.
Unravelling the definitions in \cite{Lus5}, one sees that $K^*$ is called $B$ 
in that paper. We point out that the proof of \cite[Proposition 10.12]{Lus5} 
misses a *-sign in equation (c), the correct statement involves the dual space 
of $E^\circ_{y,\sigma,r}$. It implies that
$E^\circ_{y,\sigma,r,\rho^\circ} \neq 0$ if and only if 
\begin{equation}\label{eq:2.10}
\Hom_{\pi_0 (Z_{G^\circ}(\sigma,y))} \big( (\rho^\circ)^*, 
\bigoplus\nolimits_n \cH^n (i_y^! \tilde B) \big) \neq 0 . 
\end{equation}
Here $i_y : \{y\} \to \tilde{\mf g} = \{x \in \mf g : [\sigma,x] = 2 r x\}$ 
is the inclusion and $\tilde B$ is the restriction of $K^*$ to $\tilde{\mf g}$.
In the notation of \cite[Corollary 8.18]{Lus5}, \eqref{eq:2.10} means that 
$(y,(\rho^\circ)^*)$ (or more precisely the associated local system on 
$\tilde{\mf g}$) is an element of $\mathcal M_{o,\cF}$. By 
\cite[Proposition 8.17]{Lus5}, that is equivalent to the existence 
of a $(\rho^\circ_G)^* \in \Irr \big( \pi_0 (Z_{G^\circ}(y)) \big)$ such that:
\begin{itemize}
\item $(\rho^\circ_G)^* \big|_{\pi_0 (Z_{G^\circ}(\sigma,y))}$ contains 
$(\rho^\circ)^*$,
\item the local system $(\cF^\circ)^*$ on $\cC_y^{G^\circ}$ with fibre 
$(\rho^\circ_G)^*$ at $y$ is a direct summand of 
$\bigoplus_{n \in \Z} \cH^n (K^*)|_{\cC_y^{G^\circ}}$.
\end{itemize}
The natural pairing between $\cL$ and $\cL^*$ induces a pairing between $K$
and $K^*$. This allows us to identify each fibre $(K^*)_y$ with the dual space 
of $K_y = ((\pr_1)! \dot{\cL})_y$, and it gives an isomorphism
\begin{equation}\label{eq:2.11}
(\cH^n (K)|_y )^* \cong \cH^n (K^*)|_y . 
\end{equation}
This shows that the previous condition is equivalent to:\\ 
there exists a $\rho^\circ_G \in \Irr \big( \pi_0 (Z_{G^\circ}(y)) \big)$ 
which contains $\rho^\circ$ and such that $\cF^\circ$ is a direct summand 
of $\bigoplus_{n \in \Z} \cH^n (K)|_{\cC_y^{G^\circ}}$. 

According to \cite[Proposition 8.16]{Lus5} there is a unique $B'$, among the 
possible choices of $(L',\cC_{v'}^{L'},\cL')$, such that $K^* = B'$ fulfills 
this condition. By \cite[Theorem 6.5]{Lus2} it can be fulfilled with the 
cuspidal support of $(\cC_y^{G^\circ},\cF^\circ)$ and $n$ equal to
\begin{equation} \label{eqn:degree}
2d_{\cC_y^{G^\circ},\cC_v^L}:=\dim Z_{G^\circ}(y)-\dim Z_L(v).
\end{equation} 
Hence we may restrict $n$ to $2 d_{\cC_y^{G^\circ},\cC_v^L}$ without 
changing the last condition.

By \cite[Proposition 5.6.a]{AMS} the second condition of the proposition is 
equivalent to: there exists a $\rho^\circ_G \in \Irr \big( \pi_0 
(Z_{G^\circ}(y)) \big)$ such that $\rho^\circ_G \big|_{\pi_0 (Z_{G^\circ}
(\sigma,y))}$ contains $\rho^\circ$ and  
$\Psi_{G^\circ}(y,\rho_G^\circ) = (L,\cC_v^L,\cL)$. 

Let $\cF^\circ$ be the local system on $\cC_y^{G^\circ}$ with 
$(\cF^\circ )_y = \rho^\circ_G$. By \cite[Theorem 6.5]{Lus2} its cuspidal
support is $(L,\cC_v^L,\cL)$ if and only if $\cF^\circ$ is a direct
summand of $\cH^{2d}(K )|_{\cC_y^{G^\circ}}$, where $d = 
d_{\cC_y^{G^\circ},\cC_v^L}$. Hence the second condition of the propostion 
is equivalent to all the above conditions, if $r \neq 0$.

For any $r \in \C$, as $\C [W_\cL^\circ]$-modules:
\begin{equation}\label{eq:2.15}
E^\circ_{y,\sigma,r,\rho^\circ} = \Hom_{\pi_0 (M(y))_\sigma^\circ} 
\big( \rho^\circ, H_* (\cP_y,\dot{\cL}) \big) ,
\end{equation}
see \cite[10.12.(d)]{Lus5}. Recall from \eqref{eq:2.7} that 
$\pi_0 (M(y))_\sigma^\circ \cong \pi_0 (Z_{G^\circ}(\sigma_0,y))$. 
For any $t \in \C$ we obtain
\[
E^\circ_{y,\sigma_0 + r \text{d}\gamma_y \matje{t}{0}{0}{-t},t r,
\rho^\circ} = \Hom_{\pi_0 (Z_{G^\circ}(\sigma_0,y))} 
\big( \rho^\circ, H_* (\cP_y,\dot{\cL}) \big) .
\]
The right hand side is independent of $t \in \C$ and for $t r \neq 0$ 
it is nonzero if and only if $\Psi_{ Z_{G^\circ}(\sigma_0)} (y,\rho^\circ) 
= (L,\cC_v^L,\cL)$ (up to $G^\circ$-conjugacy). Hence the same goes for 
$E^\circ_{y,\sigma_0,0,\rho^\circ}$. We have $\sigma_0 = \sigma$ if $r = 0$,
so this accounts for all $(\sigma,r)/\!\! \sim \; \in V_y$ with $r \in \C$.
\end{proof}

\subsection{Representations annihilated by ${\mb r}$} \ 
\label{par:ann}

The representations of $\mh H (G,L,\cL)$ which are annihilated by ${\mb r}$ 
can be identified with representations of 
\[
\mh H (G,L,\cL) / ({\mb r}) = \C[W_\cL,\natural_\cL] \ltimes S(\mf t^*). 
\]
Se will study the irreducible representations of this algebra in a straightforward
way: we give ad-hoc definitions of certain modules, then we show that these
exhaust $\Irr \big( \C[W_\cL,\natural_\cL] \ltimes S(\mf t^*) \big)$, and we
provide a parametrization.

The generalized Springer correspondence \cite{Lus2} associates to 
$(y,\rho^\circ)$ an irreducible representation $M_{y,\rho^\circ}$ of a 
suitable Weyl group. It is a representation of $W_\cL^\circ$ if the cuspidal 
support $\Psi_{G^\circ}(y,\rho^\circ)$ is $(L,\cC_v^L,\cL)$. If that is the 
case and $\sigma_0 \in \mathrm{Lie}(Z(G^\circ))$, we let 
$M^\circ_{y,\sigma_0,0,\rho^\circ}$ be the irreducible 
$\mh H(G^\circ,L,\cL)$-module on which $S( \mf t^* \oplus \C )$ acts via
the character $(\sigma_0,0) \in \mf t \oplus \C$ and 
\begin{equation}\label{eq:2.74}
M^\circ_{y,\sigma_0,0,\rho^\circ} = M_{y,\rho^\circ} \text{ as } 
\C[W_\cL^\circ] \text{-modules.}
\end{equation}
For a general $\sigma_0 \in Z_{\mf g}(y)$ we can define a similar 
$W_\cL^\circ \ltimes S (\mf t^*)$-module. We may assume that $\cP_y^\circ$ 
is nonempty, for otherwise $H_*^{M(y)^\circ}(\cP_y^\circ,\dot{\cL}) = 0$. 
Upon replacing $(y,\sigma_0)$ by a suitable $G^\circ$-conjugate, we may also 
assume that $L$ centralizes $\sigma_0$. Write ${Q^\circ} = Z_{G^\circ}
(\sigma_0)$, a Levi subgroup of $G^\circ$ containing $L$. Notice that 
$W_\cL^{Q^\circ} = W({Q^\circ},T)$ is a Weyl group, the stabilizer of 
$\sigma_0$ in $W_\cL$. Then $\pi_0 (M(y))_{\sigma_0}^\circ \cong 
\pi_0 (Z_{Q^\circ} (y))$, so $(y,\sigma_0,\rho^\circ)$ determines the 
irreducible $\mh H ({Q^\circ},L,\cL)$-module $M^{Q^\circ}_{y,\sigma_0,
0,\rho^\circ}$. We define
\begin{equation}\label{eq:2.18}
M^\circ_{y,\sigma_0,0,\rho^\circ} = \ind_{W_\cL^{Q^\circ} \ltimes 
S (\mf t^*)}^{W_\cL^\circ \ltimes S(\mf t^*)} (M^{Q^\circ}_{y,\sigma_0,
0,\rho^\circ}) = \ind_{\mh H ({Q^\circ},L,\cL)}^{\mh H (G^\circ,L,\cL)} 
(M^{Q^\circ}_{y,\sigma_0,0,\rho^\circ}) .
\end{equation}
\begin{prop}\label{prop:2.7}
The map $(y,\sigma_0,\rho^\circ) \mapsto M^\circ_{y,\sigma_0,0,\rho^\circ}$ 
induces a bijection between:
\begin{itemize}
\item $G^\circ$-conjugacy classes of triples $(y,\sigma_0,\rho^\circ)$ 
such that $y \in Z_{\mf g}(\sigma_0)$ nilpotent, $\rho^\circ \in \Irr 
\big( \pi_0 (Z_{G^\circ}(\sigma_0,y)) \big)$ and $\Psi_{Z_{G^\circ}
(\sigma_0)}(y,\rho^\circ)$ is $G^\circ$-conjugate to $(L,\cC_v^L,\cL)$;
\item $\Irr (W_\cL^\circ \ltimes S(\mf t^*)) = 
\Irr (\mh H (G^\circ,L,\cL) / ({\mb r}))$.
\end{itemize}
\end{prop}
\begin{proof}
By definition $S(\mf t^* )$ acts on $M^{Q^\circ}_{y,\sigma_0,0,\rho^\circ}$ 
via the character $\sigma$. For $w \in W_\cL^\circ$ it acts on 
$w M^{Q^\circ}_{y,\sigma_0,0,\rho^\circ} \subset M^\circ_{y,\sigma_0,0,
\rho^\circ}$ as the character $w \sigma$. Since $W_\cL^{Q^\circ}$ is the 
centralizer of $\sigma$ in $W_\cL$, the $S(\mf t^*)$-weights $w \sigma$ with 
$w \in W_\cL^\circ / W_\cL^{Q^\circ}$ are all different. As vector spaces
\[
M^\circ_{y,\sigma_0,0,\rho^\circ} = \ind_{W_\cL^{Q^\circ} \ltimes 
S (\mf t^*)}^{W_\cL^\circ \ltimes S(\mf t^*)} \big( M^{Q^\circ}_{y,\sigma_0,0,
\rho^\circ} \big) = \bigoplus\nolimits_{w \in W_\cL^\circ / W_\cL^{Q^\circ}} 
M^{Q^\circ}_{y,\sigma_0,0,\rho^\circ},
\]
and $M^{Q^\circ}_{y,\sigma_0,0,\rho^\circ}$ is irreducible. With Frobenius 
reciprocity we see that $M^\circ_{y,\sigma_0,0,\rho^\circ}$ is also irreducible.

Recall that the generalized Springer correspondence \cite{Lus2} provides 
a bijection between $\Irr (W_\cL^{Q^\circ})$ and the ${Q^\circ}$-conjugacy 
classes of pairs $(y,\rho^\circ)$ where $y \in \mathrm{Lie}({Q^\circ})$ is 
nilpotent and $\rho^\circ \in \Irr \big( \pi_0 (Z_{Q^\circ} (y)) \big)$ 
such that $\Psi_{Q^\circ} (y,\rho^\circ) = (L,\cC_v^L,\cL)$. We obtain 
a bijection between $G^\circ$-conjugacy classes of triples 
$(y,\sigma_0,\rho^\circ)$ and $W_\cL^\circ$-association classes of pairs 
$(\sigma_0,\pi)$ with $\sigma_0 \in \mf t$ and $(\pi,V_\pi) \in \Irr 
((W_\cL^\circ)_{\sigma_0})$. It is well-known, see for example 
\cite[Theorem 1.1]{Sol1}, that the latter set in a bijection with 
$\Irr (W_\cL^\circ \times S(\mf t^*))$ via
\[
(\sigma_0,\pi) \mapsto \ind_{(W_\cL^\circ)_{\sigma_0} \ltimes 
S(\mf t^*)}^{W_\cL^\circ \ltimes S(\mf t^*)}(\C_{\sigma_0} \otimes V_\pi) .
\]
In other words, the map of the proposition is a bijection.
\end{proof}

We woud like to relate the above irreducible representations of $W_\cL^\circ \ltimes 
S(\mf t^*)$ to the standard modules from the previous paragraph. To facilitate this
we first exhibit some properties of standard modules, which are specific
for the case $r = 0$.

\begin{lem}\label{lem:2.6}
Assume that $\Psi_{G^\circ}(y,\rho^\circ) = (L,\cC_v^L,\cL)$. The standard 
$\mh H(G^\circ,L,\cL)$-module $E^\circ_{y,\sigma_0,0,\rho^\circ}$ is completely 
reducible and admits a module decomposition by homological degree:
\[
E^\circ_{y,\sigma_0,0,\rho^\circ} = \bigoplus\nolimits_n
\Hom_{\pi_0 (M(y))_\sigma^\circ} \big( \rho^\circ, H_n (\mc P_y^\circ,\dot{\cL}) \big).
\]
\end{lem}
\begin{proof}
First we assume that $\sigma_0$ is central in Lie$(G^\circ)$. Then the action of 
$S(\mf t^* \oplus \C)$ simplifies. Indeed, from \eqref{eq:2.9} we see that it given 
just by evaluation at $(\sigma_0,0)$. Hence the structure of $E^\circ_{y,\sigma_0,0}$
as a $\mh H (G^\circ,L,\cL)$-module is completely determined by the action of
$\C [W_\cL^\circ]$. That is a semisimple algebra, so 
\begin{equation}\label{eq:2.60}
E^\circ_{y,\sigma_0,0} \text{ is completely reducible.}
\end{equation}
Then the direct summand $E^\circ_{y,\sigma_0,0,\rho^\circ}$ is also
completely reducible.

By \cite[10.12.(d)]{Lus5} $E^\circ_{y,\sigma_0,0}$ can be identified with 
$H_* (\cP_y^\circ,\dot{\cL})$, as $W_\cL^\circ$-representations. In 
\eqref{eq:2.68} we observed that the action of $\C[W_\cL^\circ]$ preserves
the homological degree, so
\begin{equation}\label{eq:2.69}
E^\circ_{y,\sigma_0,0} \cong \bigoplus\nolimits_n H_n (\cP_y^\circ,\dot{\cL})
\text{ as } W_\cL^\circ \ltimes S(\mf t^*) \text{-representations.}
\end{equation}
This decomposition persists after applying $\Hom (\rho^\circ,?)$.

Now we lift the condition on $\sigma_0$, and we consider the Levi subgroup 
$Q^\circ = Z_{G^\circ}(\sigma_0)$ of $G^\circ$. 
As explained before Proposition \ref{prop:2.7}, we may assume that 
$L \subset {Q^\circ}$. By \cite[Corollary 1.18]{Lus7} there is a natural 
isomorphism of $\mh H(G^\circ,L,\cL)$-modules
\begin{equation}\label{eq:2.12} 
W_\cL^\circ \ltimes S(\mf t^*) \underset{W_\cL^{Q^\circ} \ltimes S (\mf t^*)}{
\otimes} E^{Q^\circ}_{y,\sigma_0,0} =
\mh H (G^\circ,L,\cL) \underset{\mh H ({Q^\circ},L,\cL)}{\otimes} 
E^{Q^\circ}_{y,\sigma_0,0} \longrightarrow E^\circ_{y,\sigma_0,0} .
\end{equation}
We note that \cite[Corollary 1.18]{Lus7} is applicable because $r = 0$ and 
$\ad (\sigma_0)$ is an invertible linear transformation of Lie$(U_{Q^\circ})$, 
where $U_{Q^\circ}$ is the unipotent radical of a parabolic subgroup of 
$G^\circ$ with Levi factor ${Q^\circ}$. 

For later use we remark that the map \eqref{eq:2.12} comes from a morphism 
$\cP_y^{Q^\circ} \to \cP_y^\circ$. Hence it changes all homological degrees by 
the same amount, namely $\dim \cP_y^\circ - \dim \cP_y^{Q^\circ}$.

In \eqref{eq:2.60} we saw that the $W_\cL^{Q^\circ} \ltimes S (\mf t^*)$-module 
$E^{Q^\circ}_{y,\sigma_0,0}$ is completely reducible. Above we also showed that 
$S(\mf t^* \oplus \C)$ acts on $E^{Q^\circ}_{y,\sigma_0,0,\rho^\circ}$ via the 
character $(\sigma_0,0)$.
With the braid relation from Propostion \ref{prop:1.4} we see that
\begin{equation}\label{eq:2.63}
S(\mf t^* \oplus \C) \text{ acts on } w E^{Q^\circ}_{y,\sigma_0,0}
\text{ via the character } (w \sigma_0,0).
\end{equation}
As $W_\cL^{Q^\circ}$ is the stabilizer of $\sigma_0$ in $W_\cL$, this brings
the reducibility question for $E^\circ_{y,\sigma_0,0}$ back to that for 
$E^{Q^\circ}_{y,\sigma_0,0}$, which we already settled. Thus
\begin{equation}\label{eq:2.61}
E^\circ_{y,\sigma_0,0} \text{ is completely reducible.} 
\end{equation}
This implies that the direct summand $E^\circ_{y,\sigma_0,0,\rho^\circ}$
is also completely reducible.

It follows from \eqref{eq:2.69} and \eqref{eq:2.63} that $E^\circ_{y,\sigma_0,0}
= H_* (\mc P_y^\circ, \dot{\cL})$ and that the action of $W_\cL^\circ \ltimes
S(\mf t^*)$ preserves the homological degree. The same goes for the action of
$\pi_0 (M(y))_\sigma^\circ$, which yields the desired module decomposition of
$E^\circ_{y,\sigma_0,0,\rho^\circ}$.
\end{proof}

In terms of Lemma \ref{lem:2.6} we can describe explicitly how a standard
module for $\mh H (G,L,\cL) / ({\mb r})$ contains the irreducible module
with the same parameter.

\begin{lem}\label{lem:2.8}
The $W_\cL^\circ \ltimes S(\mf t^*)$-module $E^\circ_{y,\sigma_0,0,\rho^\circ}$
has a unique irreducible subquotient isomorphic to 
$M^\circ_{y,\sigma_0,0,\rho^\circ}$. It is the component of
$E^\circ_{y,\sigma_0,0,\rho^\circ}$ in the homological degree
\[
\dim \cP_y^\circ - \dim \cP_y^{Z_G (\sigma_0)^\circ} +
\dim Z_{G^\circ}(\sigma_0,y)-\dim Z_L(v).
\]
\end{lem}
\begin{proof}
For the moment we assume that $\sigma_0$ is central in Lie$(G^\circ)$.
According to \cite[Theorem 8.15]{Lus3} every irreducible 
$\mh H (G^\circ,L,\cL)$-module is a quotient of some standard module. The 
central character of $M^\circ_{y,\sigma_0,0, \rho^\circ}$ is $(\sigma_0,0) 
\in \mf t / W_\cL^\circ \times \C$. In view of Proposition \ref{prop:2.3}.b, 
$M^\circ_{y,\sigma_0,0,\rho^\circ}$ cannot be a subquotient of a standard
module $E^\circ_{y,\sigma,r,\rho^\circ}$ with $(\sigma,r) \neq (\sigma_0,0)$. 
Therefore it must be a quotient of $E^\circ_{y,\sigma_0,0,\rho'}$ for some 
\[
\rho' \in \Irr (\pi_0 (M(y))_{\sigma_0}^\circ = 
\Irr \big( \pi_0 (Z_{G^\circ}(y)) \big) .
\]
By definition \cite[1.5.(c)]{Lus5} the dual space of \eqref{eq:2.69} is 
\begin{equation}\label{eq:2.16}
H_* (\cP^\circ_y,\dot{\cL})^* \cong H^* (\cP_y^\circ,\dot{\cL}^*) .
\end{equation}
Since we sum over all degrees, we may ignore changes in the grading for now. 
By \cite[10.12.(c)]{Lus5}, in which a *-sign is missing, \eqref{eq:2.16} 
is isomorphic to $H^* (\{y\},i_y^! (K^*))$, where $i_y : \{y\} \to \mf g$ 
is the inclusion. From \cite[1.3.(d) and 1.4.(a)]{Lus5} we see that
\begin{equation}\label{eq:2.17}
H^* (\{y\},i_y^! (K^*)) \cong H^* (\{y\},i_y^* (K^*)) \cong \cH^* (K^*) |_y .
\end{equation}
From \eqref{eq:2.16}, \eqref{eq:2.17} and \eqref{eq:2.11} we get isomorphisms
\[
E^\circ_{y,\sigma_0,0} \cong H_* (\cP_y^\circ,\dot{\cL}) \cong 
(\cH^* (K^*) |_y )^* \cong \cH^* (K )|_y . 
\]
The generalized Springer correspondence, which in \cite{Lus2} comes from 
sheaves on subvarieties of $G^\circ$, can also be obtained from sheaves on 
subvarieties of $\mf g$, see \cite[2.2]{Lus3}. In that version it is given by
\[
(y,\rho^\circ) \mapsto \Hom_{\pi_0 (Z_{G^\circ}(y))}
\big( \rho^\circ,\cH^{2d}(K) |_y \big),
\]
where $d = d_{\cC_y^{G^\circ},\cC_v^L}$ is as in (\ref{eqn:degree}). 
More precisely \cite[Theorem 6.5]{Lus2}:
\begin{equation}\label{eq:2.73}
\cH^{2d}(K )|_y \cong \bigoplus\nolimits_{\rho'} V_{\rho'} \otimes M_{y,\rho'}
\text{ as } \pi_0 (Z_{G^\circ}(y)) \times W_\cL^\circ \text{-representations,}
\end{equation}
where the sum runs over all $(\rho',V_{\rho'}) \in \Irr \big( \pi_0 
(Z_{G^\circ}(y))\big)$ with $\Psi_{G^\circ}(y,\rho') = (L,\cC_v^L,\cL)$. 

Let $I$ denote the set of all pairs $i=(\cC_y^{G^\circ},\cF^\circ)$ where $\cC_y^{G^\circ}$  
is the adjoint orbit of a nilpotent element $y$ in $\mf g$, and $\cF^\circ$ is an 
irreducible $G^\circ$-equivariant local system (given up to isomorphism) on $\cC_y^{G^\circ}$. 
In \cite[Theorem 24.8]{LusCS}, Lusztig has proved that for any 
$i=(\cC_y^{G^\circ},\cF^\circ)\in I$:
\begin{itemize}
\item $\cH^n \big(\IC(\overline \cC_y^{G^\circ},\cF^{\circ})\big)=0$ if $n$ is odd.
\item for $i'=(\cC_{y'}^{G^\circ},\cF^{\circ,\prime}) \in I$ the polynomial 
\[ \Pi_{i,i'}:=\sum\nolimits_m \big(\cF^{\circ,\prime}:
\cH^{2m}\big(\IC(\overline \cC_y^{G^\circ},\cF^{\circ})\big)|_{\cC_{y'}^{G^\circ}}\big)\,\bf{q}^m,
\]
in the indeterminate $\bf q$, satisfies $\Pi_{i,i}=1$.
\end{itemize}
From the second bullet we obtain 
\begin{equation} \label{eqn:cH}
\big(\cF^{\circ}:\cH^{2m}\big(\IC(\overline \cC_y^{G^\circ},\cF^{\circ})\big)|_{\cC_{y}^{G^\circ}})
=\begin{cases} 0 \quad \text{ if $m\ne 0$}\cr
1 \quad \text{ if $m=0$.}
\end{cases}
\end{equation}
By combining \eqref{eqn:cH} with \cite[Theorem 6.5]{Lus2}, where the considered complex 
is shifted in degree $2d_{\cC_y^{G^\circ},\cC_v^L}=\dim Z_{G^\circ}(y)-\dim Z_L(v)$, we obtain that 
\[
E^\circ_{y,\sigma_0,0,\rho^\circ} \cong \Hom_{\pi_0 (Z_{G^\circ}(y))} 
\big( \rho^\circ, \cH^* (K ) |_y \big) 
\]
contains $M_{y,\rho^\circ}$ with multiplicity one, as the component in
the homological degree $2d_{\cC_y^{G^\circ},\cC_v^L}$.

Now consider a general $\sigma_0 \in \Lie (G^\circ)$, and we write
$Q = Z_G (\sigma_0)$. By Lemma \ref{lem:2.2}
\begin{equation}\label{eq:2.37}
\pi_0 (M(y))_{\sigma_0}^\circ \cong \pi_0 (Z_{G^\circ}(\sigma_0,y)) =
\pi_0 (Z_{Q^\circ} (\sigma_0,y)) = \pi_0 (Z_{Q^\circ} (y)) .
\end{equation}
By Theorem \ref{thm:2.4}.d the action of this group commutes with that of 
$\mh H (G^\circ,L,\cL)$, so \eqref{eq:2.12} contains an isomorphism of 
$\mh H(G^\circ,L,\cL)$-modules
\begin{equation}\label{eq:2.38}
\mh H (G^\circ,L,\cL) \underset{\mh H ({Q^\circ},L,\cL)}{\otimes} 
E^{Q^\circ}_{y,\sigma_0,0,\rho^\circ} \to E^\circ_{y,\sigma_0,0,\rho^\circ} . 
\end{equation}
The argument for the irreducibility of $M^\circ_{y,\sigma_0,0,\rho^\circ}$ 
in the proof of Proposition \ref{prop:2.7} also applies here, when we use 
Proposition \ref{prop:2.3}.b. It shows that the $S(\mf t^*)$-modules 
$w E^\circ_{y,\sigma_0,0,\rho^\circ}$ with $w \in W_\cL^\circ / 
W_\cL^{Q^\circ}$ contain only different $S(\mf t^*)$-modules, so they have
no common irreducible constituents. It follows that the functor 
$\ind^{\mh H (G^\circ,L,\cL)}_{\mh H ({Q^\circ},L,\cL)}$ provides a bijection 
between $\mh H ({Q^\circ},L,\cL)$-subquotients of $E^{Q^\circ}_{y,\sigma_0,
0,\rho^\circ}$ and $\mh H(G^\circ,L,\cL)$-subquotients of $E^\circ_{y,\sigma_0,
0,\rho^\circ}$. Together with the statement of the lemma for $(Q^\circ,\sigma_0)$, 
we see that  $E^\circ_{y,\sigma_0,0,\rho^\circ}$ has a unique quotient isomorphic to 
$M^\circ_{y,\sigma_0,0,\rho^\circ}$ and no other constituents isomorphic to that.

As remarked before, the maps \eqref{eq:2.12} and \eqref{eq:2.38} come from a morphism 
$\cP_y^{Q^\circ} \to \cP_y^\circ$, so they change all homological degrees by 
$\dim \cP_y^\circ - \dim \cP_y^{Q^\circ}$.
With the result for $(Q^\circ,\sigma_0)$ at hand, it follows that the image of 
$W_\cL^\circ \ltimes S(\mf t^*) \underset{W_\cL^{Q^\circ} \ltimes S(\mf t^*) 
}{\otimes} M^{Q^\circ}_{y,\sigma_0,0,\rho^\circ}$ is the full component of 
$E^\circ_{y,\sigma_0,0,\rho^\circ}$ in the stated homological degree. By definition this 
image is also (isomorphic to) $M^\circ_{y,\sigma_0,0, \rho^\circ}$.
\end{proof}

\subsection{Intertwining operators and 2-cocycles} \
\label{par:intop}

For $r \in \C$ we let $\Irr_{r}(\mh H (G,L,\cL))$ be the set of (equivalence 
classes of) irreducible $\mh H (G,L,\cL)$-modules on which ${\mb r}$ acts as $r$.

The irreducible representations of $\mh H(G,L,\cL)$ are built from those of 
$\mh H(G^\circ,L,\cL)$. Let us collect some available information about 
the latter here.

\begin{thm}\label{thm:2.9}
Let $y \in \mf g$ be nilpotent and let $(\sigma,r)/\!\!\sim \; \in V_y$ be 
semisimple. Let $\rho^\circ \in \Irr \big(\pi_0 (Z_{G^\circ}(\sigma,y))\big)$ 
be such that $\Psi_{Z_{G^\circ}(\sigma_0)}(y,\rho^\circ) = (L,\cC_v^L,\cL)$ 
(up to $G^\circ$-conjugation).
\enuma{
\item If $r \neq 0$, then $E^\circ_{y,\sigma,r,\rho^\circ}$ has a unique 
irreducible quotient $\mh H (G^\circ,L,\cL)$-module. We call it 
$M^\circ_{y,\sigma,r,\rho^\circ}$.
\item If $r = 0$, then $E^\circ_{y,\sigma_0,r, \rho^\circ}$ has a unique 
irreducible summand isomorphic to $M^\circ_{y,\sigma_0,0,\rho^\circ}$.
\item Parts (a) and (b) set up a canonical bijection between 
$\Irr_{r}(\mh H (G^\circ,L,\cL))$ and the $G^\circ$-orbits of triples 
$(y,\sigma,\rho^\circ)$ as above.
\item Every irreducible constituent of $E^\circ_{y,\sigma,r,\rho^\circ}$, 
different from $M^\circ_{y,\sigma,r,\rho^\circ}$, is isomorphic to a 
representation $M^\circ_{y',\sigma',r,\rho'}$ with 
$\dim \cC_y^{G^\circ} < \dim \cC_{y'}^{G^\circ}$.
}
\end{thm}
\begin{proof}
(a) is \cite[Theorem 1.15.a]{Lus7}.\\
(b) is a less precise version of Lemma \ref{lem:2.8}. \\
(c) For $r \neq 0$ see \cite[Theorem 1.15.c]{Lus7} and for $r = 0$ see
Proposition \ref{prop:2.7}. \\
(d) As noted in \cite[\S 3]{Ciu}, this follows from \cite[\S 10]{Lus5}.
\end{proof}

Our goal is to generalize Theorem \ref{thm:2.9} from $G^\circ$ to $G$. 
To this end we have to extend both $\rho^\circ$ and $M^\circ_{y,\sigma,
\rho^\circ}$ to representations of larger algebras. That involves the 
construction of some intertwining operators, followed by Clifford theory for 
representations of crossed product algebras. Although all our intertwining 
operators are parametrized by some group, they typically do not arise from
a group homomorphism. Instead they form twisted group algebras, and we will 
have to determine the associated group cocycles as well.

The group $\cR_\cL$ acts on the set of 
$\mh H(G^\circ,L,\cL)$-representations $\pi$ by
\[
(w \cdot \pi)(h) = 
\pi (N_w^{-1} h N_w) \qquad w \in \cR_\cL, h \in \mh H (G^\circ,L,\cL).
\]
Let $\cR_{\cL,y,\sigma}$ (respectively $\cR_{\cL,y,\sigma,\rho^\circ}$) 
be the stabilizer of $E^\circ_{y,\sigma,r}$ (respectively 
$E^\circ_{y,\sigma,r\rho^\circ}$) in $\cR_\cL$. Similarly the group 
$\pi_0 (Z_G (\sigma,y)))$ acts the set of $\pi_0 (Z_{G^\circ}(\sigma,y))
$-representations. Let $\pi_0 (Z_G (\sigma,y))_{\rho^\circ}$ 
be the stabilizer of $\rho^\circ$ in $\pi_0 (Z_G (\sigma,y))$.

\begin{lem}\label{lem:2.10}
There are natural isomorphisms
\enuma{
\item $\cR_{\cL,y,\sigma} \cong \pi_0 (Z_G (\sigma,y)) / 
\pi_0 (Z_{G^\circ}(\sigma,y))
\cong \pi_0 (Z_G (\sigma_0,y)) / \pi_0 (Z_{G^\circ}(\sigma_0,y))$,
\item $\cR_{\cL,y,\sigma,\rho^\circ} \cong \pi_0 (Z_G (\sigma,y))_{
\rho^\circ} / \pi_0 (Z_{G^\circ}(\sigma,y))$.
}
\end{lem}
\begin{proof}
(a) Since all the constructions are algebraic and $\cR_\cL$ acts by algebraic 
automorphisms, $w \cdot E^\circ_{y,\sigma,r} \cong E^\circ_{w(y),w(\sigma),r}$.
By Theorem \ref{thm:2.9}.c $E^\circ_{w(y),w(\sigma),r} \cong 
E^\circ_{y,\sigma,r}$ if and only if $(y,\sigma)$ and $(w(y),w(\sigma))$ 
are in the same $\Ad(G^\circ)$-orbit. We can write this condition as 
$w G^\circ \subset G_{\Ad (G^\circ)}(y,\sigma)$. Next we note that
\[
G_{\Ad (G^\circ)}(y,\sigma) / G^\circ \cong 
Z_G (\sigma,y) / Z_{G^\circ}(\sigma,y) .
\]
Since $G / G^\circ$ is finite, the right hand side is isomorphic to 
$\pi_0 (Z_G (\sigma,y)) / \pi_0 (Z_{G^\circ}(\sigma,y))$. By Lemma 
\ref{lem:2.2}.a we can replace $\sigma$ by $\sigma_0$ without changing 
these groups.\\
(b) Consider the stabilizer of $\rho^\circ$ in $\pi_0 (Z_G (\sigma,y)) / 
\pi_0 (Z_{G^\circ}(\sigma,y))$. By part (a) it is isomorphic to the 
stabilizer of $\rho^\circ$ in $\cR_{\cL,y,\sigma}$. As $E^\circ_{y,\sigma,
r,\rho^\circ} = \Hom_{\pi_0 (Z_{G^\circ}(\sigma,y))}(\rho^\circ, 
E^\circ_{y,\sigma,r})$ and $\cR_{\cL,y,\sigma}$ stabilizes 
$E^\circ_{y,\sigma,r}$, this results in the desired isomorphism.
\end{proof}

Next we parametrize the relevant representations of $\pi_0 (Z_G (\sigma,y))$.

\begin{lem}\label{lem:2.11}
There exists a bijection
\[
\begin{array}{ccc}
\Irr (\C [\cR_{\cL,y,\sigma,\rho^\circ},\natural_\cL^{-1}]) & \to &
\big\{ \rho \in \Irr \big( \pi_0 (Z_G (\sigma,y)) \big) : 
\rho |_{\pi_0 (Z_{G^\circ} (\sigma,y))} \text{ contains } \rho^\circ \big\} \\
(\tau,V_\tau) & \mapsto & \tau \ltimes \rho^\circ 
\end{array}.
\]
Here $\tau \ltimes \rho^\circ = \ind^{\pi_0 (Z_{G^\circ} (\sigma,y))}_{\pi_0 
(Z_{G^\circ} (\sigma,y))_{\rho^\circ}} (V_\tau \otimes V_{\rho^\circ})$, 
where $V_\tau \otimes V_{\rho^\circ}$ is the tensor product of two 
projective representations of the stabilizer of $\rho^\circ$ in 
$\pi_0 (Z_{G^\circ} (\sigma,y))$.
\end{lem}
\begin{proof}
For $\gamma \in \pi_0 (Z_G (\sigma,y))_{\rho^\circ}$ we choose 
$I^\gamma \in \Aut_\C (V_{\rho^\circ})$ such that
\begin{equation}\label{eq:2.13}
I^\gamma \circ \rho^\circ (\gamma^{-1} z \gamma) = \rho^\circ (z) 
\circ I^\gamma \qquad z \in \pi_0 (Z_{G^\circ}(\sigma,y)) .
\end{equation}
To simplify things a little, we may and will assume that $I^{\gamma z} = 
I^{\gamma} \circ \rho^\circ (z)$ for all $\pi_0 (Z_G (\sigma,y))_{\rho^\circ}, 
z \in \pi_0 (Z_{G^\circ}(\sigma,y))$. Then \eqref{eq:2.13} implies that also
$I^{z \gamma} = \rho^\circ (z) \circ I^\gamma$. By Schur's lemma there exist
unique $\kappa_{\rho^\circ}(\gamma,\gamma') \in \C^\times$ such that
\begin{equation}\label{eq:2.14}
I^{\gamma \gamma'} = \kappa_{\rho^\circ}(\gamma,\gamma') I^\gamma \circ 
I^{\gamma'} \qquad \gamma,\gamma' \in \pi_0 (Z_G (\sigma,y))_{\rho^\circ}.
\end{equation}
Then $\kappa_{\rho^\circ}$ is a 2-cocycle of 
$\pi_0 (Z_G (\sigma,y))_{\rho^\circ}$. The above assumption and Lemma 
\ref{lem:2.10}.b implies that it factors via 
\[
\pi_0 (Z_G (\sigma,y))_{\rho^\circ} / \pi_0 (Z_{G^\circ} 
(\sigma,y))_{\rho^\circ} \cong \cR_{\cL,y,\sigma,r,\rho^\circ} . 
\]
Let $\C [\cR_{\cL,y,\sigma,r,\rho^\circ},\kappa_{\rho^\circ}]$ be the 
associated twisted group algebra, with basis $\{ T_\gamma : \gamma \in 
\cR_{\cL,y,\sigma,r,\rho^\circ} \}$. Then $\pi_0 (Z_G (\sigma,y))$ acts on 
\[
\C [\cR_{\cL,y,\sigma,r,\rho^\circ},\kappa_{\rho^\circ}] \otimes_\C 
V_{\rho^\circ} \quad \text{by} \quad \gamma \cdot (T_{\gamma'} \otimes v) = 
T_{\gamma \gamma'} \otimes I^\gamma (v).
\]
By Clifford theory (see \cite[\S 1]{AMS}) there is a bijection
\[
\begin{array}{ccc}
\Irr (\C [\cR_{\cL,y,\sigma,\rho^\circ},\kappa_{\rho^\circ}]) & \to &
\big\{ \rho \in \Irr \big( \pi_0 (Z_G (\sigma,y)) \big) : 
\rho |_{\pi_0 (Z_{G^\circ} (\sigma,y))} \text{ contains } \rho^\circ \big\} \\
(\tau,V_\tau) & \mapsto & \tau \ltimes \rho^\circ 
\end{array}.
\]
It remains to identify $\kappa_{\rho^\circ}$. By Proposition \ref{prop:2.5} 
the cuspidal support of $(y,\rho^\circ)$ is $(L,\cC_v^L,\cL)$, which means 
that it is contained in $\cH^* (K) |_y$. Hence the 2-cocycle $\natural_\cL$, 
used to extend the action of $W_\cL^\circ$ on $K$ to $\C [W_\cL,\natural_\cL]$, 
also gives an action on $V_{\rho^\circ}$. Comparing the multiplication 
relations in $\C [W_\cL,\natural_\cL]$ with \eqref{eq:2.14}, we see that we 
can arrange that $\kappa_{\rho^\circ}$ is the restriction of 
$\natural_\cL^{-1}$ to $\cR_{\cL,y,\sigma,\rho^\circ}$.
\end{proof}

The analogue of \ref{lem:2.11} for $M^\circ_{y,\sigma,\rho^\circ}$ is more 
difficult, we need some technical preparations.
Since $N_{G^\circ}(P) = P$, we can identify $G^\circ / P$ with a variety 
$\cP^\circ$ of parabolic subgroups $P'$ of $G^\circ$. For $g \in G$ and 
$P' \in \cP$ we write 
\[
\Ad (g) P' = g P' g^{-1} . 
\]
This extends the left multiplication action of $G^\circ$ on $\cP^\circ$ 
and it gives rise to an action of $G$ on $\dot{\mf g}$ by
\[
\Ad (g)(x,P') = (\Ad (g)x,g P' g^{-1}).
\]
By Condition \ref{cond:1} every element of $G$ stabilizes $\cL$, so 
$\Ad (g)^* \dot{\cL} \cong \dot{\cL}$. Lift $\Ad (g)$ to an isomorphism of 
$G^\circ$-equivariant sheaves
\begin{equation}\label{eq:2.20}
\Ad_\cL (g) :  \dot{\cL} \to \Ad (g)^* \dot{\cL} .
\end{equation}
(Although there is more than one way to do so, we will see in Proposition 
\ref{prop:2.13} that in relevant situations $\Ad_\cL (g)$ is unique up to 
scalars.) Thus $\Ad_\cL (g)$ provides a system of linear bijections
\begin{equation}\label{eq:2.21}
(\dot{\cL})_{(x,P')} \to (\dot{\cL})_{(\Ad (g) x,\Ad(g) P')} \quad \text{such 
that} \quad \Ad_\cL (g) \circ (g^{-1} g^\circ g) = g^\circ \circ \Ad_\cL (g)
\end{equation}
for all $g^\circ \in G^\circ$. (Here we denote the canonical action of 
$g^\circ$ on $\cL$ simply by $g^\circ$.) Of course we can choose these maps 
such that $\Ad_\cL (g g^\circ) = \Ad_\cL (g) \circ g^\circ$ for 
$g^\circ \in G^\circ$. Notice that 
\begin{equation}\label{eq:2.22}
\Ad_\cL (g^\circ) \text{ coincides with the earlier action of } G^\circ
\text{ on } \dot{\cL} .
\end{equation} 
For $g \in Z_G (\sigma,y)$, $\Ad (g)$ stabilizes $\cP_y^\circ$, and
the map $\Ad_\cL (g)$ induces an operator $H_*^{M(y)^\circ}(\Ad_\cL (g))$
on $H_*^{M(y)^\circ}(\cP_y^\circ,\dot{\cL})$.

\begin{lem}\label{lem:2.12}
For all $h \in \mh H (G^\circ,L,\cL), \gamma \in \cR_{\cL,y,\sigma}$ 
and $g \in \gamma G^\circ \cap Z_G (\sigma,y)$:
\[
H_*^{M(y)^\circ}(\Ad_\cL (g^\circ)) \circ \Delta (h) = 
\Delta (N_\gamma h N_\gamma^{-1}) \circ H_*^{M(y)^\circ}(\Ad_\cL (g^\circ)) 
\in \End_\C \big( H_*^{M(y)^\circ}(\cP_y^\circ,\dot{\cL}) \big) .
\]
\end{lem}
\begin{proof}
By Theorem \ref{thm:2.4}.d the map $H_*^{M(y)^\circ}(\Ad_\cL (g^\circ))$ 
commutes with the action of $\mh H (G^\circ,L,\cL)$ on $H_*^{M(y)^\circ}
(\cP_y^\circ,\dot{\cL})$. Moreover $H_*^{M(y)^\circ}(\Ad_\cL (g^\circ)) = 1$ 
for $g^\circ$ in the connected group $Z_G (\sigma,\rho)^\circ$. 
Thus we get a map
\[
\pi_0 (Z_G (\sigma,y))_{\rho^\circ} \to 
\Aut_\C \big( H_*^{M(y)^\circ}(\cP_y^\circ,\dot{\cL}) \big) 
\]
which sends $\pi_0 (Z_{G^\circ}(\sigma,y))$ to $\Aut_{\mh H (G^\circ,L,\cL)} 
\big( H_*^{M(y)^\circ}(\cP_y^\circ,\dot{\cL}) \big)$.
Recall from \eqref{eq:2.9} that the action of $S(\mf t^* \oplus \C)$ on 
$H_*^{M(y)^\circ}(\cP_y^\circ,\dot{\cL})$ comes from the product with
$H^*_{M(y) \cap G^\circ}(\cP_y^\circ)$. The functoriality of this product 
and \eqref{eq:2.21} entail that 
\[
H_*^{M(y)^\circ}(\Ad_\cL (g^\circ))(\delta \otimes \eta) = 
\Ad (g) \delta \otimes H_*^{M(y)^\circ}(\Ad_\cL (g^\circ)) (\eta)
\]
for $\eta \in H_*^{M(y)^\circ}(\cP_y^\circ,\dot{\cL})$ and $\delta \in 
H^*_{M(y)^\circ \cap G^\circ}(\cP_y^\circ)$. The operators $\Ad (g)$ on\\
$H^*_{M(y) \cap G^\circ}(\cP_y^\circ)$ are trivial for $g \in Z_{G^\circ}
(\sigma,y) \subset M(y) \cap G^\circ$, they factor through
\[
Z_G (\sigma,y) / Z_{G^\circ}(\sigma,y) \cong \cR_{\cL,y,\sigma}.
\]
Similarly, the operators $\Ad (g)$ on $S (\mf t^* \oplus \C) \cong 
H^*_{G^\circ \times \C^\times}(\dot{\mf g}^\circ)$ factor through \\
$Z_G (\sigma,y) / Z_{G^\circ}(\sigma,y)$ and become the natural action
of $\cR_{\cL,y,\sigma}$. Hence
\begin{equation}\label{eq:2.23}
H_*^{M(y)^\circ}(\Ad_\cL (g)) \circ \Delta (\xi) = \Delta (\Ad (\gamma) 
\xi) \circ H_*^{M(y)^\circ}(\Ad_\cL (g))
\end{equation}
for $\gamma \in \cR_{\cL,y,\sigma}, g \in \gamma G^\circ \cap Z_G (\sigma,y))$
and $\xi \in S (\mf t^* \oplus \C)$. By making the appropriate choices, 
we can arrange that the dual map of \eqref{eq:2.20} is 
\[
\Ad_{\cL^*}(g^{-1}) : \Ad (g)^* \dot{\cL}^* \to \dot{\cL}^* . 
\]
It induces $\Ad_{\cL^*} (g^{-1}) : \Ad (g)^* K^\circ \to K^\circ$, where 
$K^\circ$ is $K^*$ but for $G^\circ$. The operators $N_w \; (w \in 
W_\cL^\circ)$ from \eqref{eq:1.1} are $G^\circ \times \C^\times$-equivariant, 
so the operator 
\begin{equation}\label{eq:2.24}
\Ad_{\cL^*}(g^{-1})^{-1} \circ N_w \circ \Ad_{\cL^*} (g^{-1}) 
\in \Aut_{G^\circ \times \C^\times}(K^\circ)
\end{equation}
depends only on the image of $g^{-1}$ in $G / G^\circ$. 
If $\gamma \in \cR_{\cL,y,\sigma}$ and $g \in \gamma G^\circ$, then we see 
from the definition of $N_w$ in \cite[3.4]{Lus2} that \eqref{eq:2.24} is a 
(nonzero) scalar multiple of $N_{\gamma w \gamma^{-1}}$. Consequently 
\[
H_*^{M(y)^\circ} (\Ad_{\cL^*}(g^{-1}) )^{-1} \circ \tilde{\Delta}(N_w) \circ 
H_*^{M(y)^\circ} (\Ad_{\cL^*} (g^{-1})) = 
\lambda (w,\gamma) \tilde{\Delta}(N_\gamma N_w N_\gamma^{-1})  
\]
for some number $\lambda (w,\gamma) \in \C^\times$. Dualizing, we find that 
\begin{equation}\label{eq:2.25}
H_*^{M(y)^\circ} (\Ad_\cL (g) ) \circ \Delta (N_w) \circ H_*^{M(y)^\circ} 
(\Ad_{\cL^*} (g))^{-1} = 
\lambda (w,\gamma) \Delta (N_\gamma N_w^{-1} N_\gamma^{-1}) . 
\end{equation}
Let $\alpha_i \in R(G^\circ,T)$ be a simple root and let $s_i \in W_\cL^\circ$ 
be the associated simple reflection. By the multiplication rules in 
$\mh H (G^\circ,L,\cL)$
\begin{equation}\label{eq:2.26}
0 = \Delta (N_{s_i} \alpha_i - {}^{s_i} \alpha_i N_{s_i} - c_i {\mb r} 
(\alpha_i - {}^{s_i} \alpha_i) / \alpha_i ) = \Delta (N_{s_i} \alpha_i + 
\alpha_i N_{s_i} - 2 c_i {\mb r}) .
\end{equation}
Now we apply \eqref{eq:2.23} and \eqref{eq:2.25} and to this equality, 
and we find 
\begin{equation}\label{eq:2.27}
\begin{split}
0 & = H_*^{M(y)^\circ} (\Ad_\cL (g) ) \circ \Delta (N_{s_i} \alpha_i +
\alpha_i N_{s_i} - 2 c_i {\mb r}) \circ H_*^{M(y)^\circ} (\Ad_\cL (g) )^{-1} \\
& = \Delta (\lambda (s_i,\gamma) N_{\gamma s_i \gamma^{-1}} \: {}^\gamma 
\alpha_i + \lambda (s_i,\gamma) \: {}^\gamma \alpha_i N_{\gamma s_i 
\gamma^{-1}} - 2 c_i {\mb r}) .
\end{split}
\end{equation}
We note that $\alpha_j := {}^\gamma \alpha_i$ is another simple root, 
with reflection $s_j := \gamma s_i \gamma^{-1}$ and $c_j = c_i$. By 
\eqref{eq:2.26} the second line of \eqref{eq:2.27} becomes 
\[
\lambda (s_i,\gamma) \Delta (N_{s_j} \: \alpha_j + \alpha_j N_{s_j} - 
2 c_j {\mb r}) + 2 c_i \Delta (\lambda (s_i,\gamma) {\mb r} -  {\mb r}) = 
2 (\lambda (s_i,\gamma) - 1) c_i \Delta ({\mb r}) .
\]
Recall from \eqref{eq:1.5} that $c_i > 0$. As $\Delta ({\mb r}) = r$ is nonzero 
for  some choices of $(\sigma,r)$, we deduce that $\lambda (s_i,\gamma) = 1$ 
for all $\gamma \in \cR_{\cL,y,\sigma,\rho^\circ}$. In view of \eqref{eq:2.25} 
this implies $\lambda (w,\gamma) = 1$ for all $w \in W_\cL^\circ, \gamma \in 
\cR_{\cL,y,\sigma,\rho^\circ}$. Now \eqref{eq:2.23} and \eqref{eq:2.25} 
provide the desired equalities.
\end{proof}

Lemma \ref{lem:2.12} says that for $g \in \gamma G^\circ \cap Z_G (\sigma,y)$, 
$H_*^{M(y)^\circ}(\Ad_\cL (g^\circ))$ intertwines the standard 
$\mh H( G^\circ,L,\cL)$-modules $E^\circ_{y,\sigma,r}$ and $\gamma \cdot 
E^\circ_{y,\sigma,r}$. However, it does not necessarily map the 
subrepresentation $E^\circ_{y,\sigma,r,\rho^\circ}$ to $\gamma \cdot 
E^\circ_{y,\sigma,r,\rho^\circ}$, even for $\gamma \in \cR_{\cL,y,\sigma,
\rho^\circ}$. Moreover $g \mapsto H_*^{M(y)^\circ} (\Ad_\cL (g^\circ))$ need 
not be multiplicative, by the freedom in \eqref{eq:2.20}. In general it is 
not even possible to make it multiplicative by clever choices in 
\eqref{eq:2.20}. The next lemma takes care of both these inconveniences.

\begin{prop}\label{prop:2.13}
Let $y,\sigma,r,\rho^\circ$ be as in Theorem \ref{thm:2.9}. 
There exists a group homomorphism 
\[
\cR_{\cL,y,\sigma,\rho^\circ} \to \Aut_\C (E^\circ_{y,\sigma,r,\rho^\circ}) :
\gamma \mapsto J^\gamma ,
\]
depending algebraically on $(\sigma,r)$ and unique up to scalars, such that 
\[
J^\gamma (\Delta (N_\gamma^{-1} h N_\gamma) \phi) = \Delta (h) J^\gamma (\phi) 
\qquad h \in \mh H (G^\circ,L,\cL), \phi \in E^\circ_{y,\sigma,r,\rho^\circ} .
\]
\end{prop}
\begin{proof}
For $g \in Z_G (\sigma,y)$ let $I^g$ be as in \eqref{eq:2.13} and let 
$H_*^{M(y)^\circ}(\Ad_\cL (g^\circ))$ be as in Lemma \ref{lem:2.12}. We define
\[
J^g (\phi) = H_*^{M(y)^\circ}(\Ad_\cL (g)) ( \phi \circ (I^g)^{-1} )
\qquad \phi \in E^\circ_{y,\sigma,r,\rho^\circ} = 
\Hom_{\pi_0 (Z_{G^\circ}(\sigma,y))} (\rho^\circ, E^\circ_{y,\sigma,r}).
\]
For $v \in V_{\rho^\circ}$ and $z \in Z_{G^\circ}(\sigma,y)$ we calculate:
\begin{align*}
J^g (\rho^\circ (z) \phi) & = H_*^{M(y)^\circ}(\Ad_\cL (g^\circ)) 
( \phi \circ (I^g)^{-1} \rho^\circ (z) v) \\
& = H_*^{M(y)^\circ}(\Ad_\cL (g)) 
( \phi \circ \rho^\circ (g^{-1} z g ) (I^g)^{-1} v) \\
& = H_*^{M(y)^\circ}(\Ad_\cL (g)) ( (g^{-1} z g) \phi \circ (I^g)^{-1} v) \\
& = H_*^{M(y)^\circ}(\Ad_\cL (g)) H_*^{M(y)^\circ}(\Ad_\cL (g^{-1} z g)) 
( \phi \circ (I^g)^{-1} v) \\ 
& = H_*^{M(y)^\circ}(\Ad_\cL (z)) H_*^{M(y)^\circ}(\Ad_\cL (g)) 
( \phi \circ (I^g)^{-1} v) = z \cdot J^g (\phi)(v) .
\end{align*}
Thus $J^g$ sends $E^\circ_{y,\sigma,r,\rho^\circ}$. It is invertible because
$H_*^{M(y)^\circ}(\Ad_\cL (g^\circ))$ and $I^g$ are. By \eqref{eq:2.22}, 
\eqref{eq:2.13} and the intertwining property of $\phi$, $J^g (\phi) = 
J^{g'}(\phi)$ whenever $g^{-1} g' \in G^\circ$. Hence $J^g \in \Aut_\C 
(E^\circ_{y,\sigma,r,\rho^\circ})$ depends only on the image of $g$ in 
$\cR_\cL \cong G / G^\circ$, and may denote it by $J^\gamma$ when 
$g \in \gamma G^\circ$.

As the $\pi_0 (Z_{G^\circ}(\sigma,y))$-action commutes with that of 
$\mh H (G^\circ,L,\cL)$, we deduce from Lemma \ref{lem:2.12} that
\begin{equation}\label{eq:2.28}
\begin{split}
J^\gamma (\Delta (N_\gamma^{-1} h N_\gamma) \phi) = H_*^{M(y)^\circ}
(\Ad_\cL (g)) \Delta (N_\gamma^{-1} h N_\gamma) (\phi \circ (I^g)^{-1}) = \\
\Delta (h) H_*^{M(y)^\circ}(\Ad_\cL (g)) (\phi \circ (I^g)^{-1}) = 
\Delta (h) J^\gamma (\phi) .
\end{split}
\end{equation}
By Lemma \ref{lem:2.2} and \eqref{eq:2.15} all the vector spaces
$E^\circ_{y,\sigma,r,\rho^\circ}$ can be identified with \\
$\Hom_{\pi_0 (Z_{G^\circ}(\sigma_0,y))} (\rho^\circ,H_* (\cP_y^\circ,
\dot{\cL}))$. In this sense $J^\gamma$ depends algebraically on $(\sigma,r) = 
(\sigma_0 + \textup{d}\gamma_y \matje{r}{0}{0}{-r}, r)$.

Given $y,r \neq 0$, \cite[Theorem 8.17.b]{Lus3} implies that 
$E^\circ_{y,\sigma,r,\rho^\circ}$ is irreducible for all $\sigma$ in a 
Zariski-open nonempty subset of $\{ \sigma \in \mf g : [\sigma,y] = 2 r y \}$. 
For such $\sigma$ \eqref{eq:2.28} and Schur's lemma imply that $J^\gamma$ 
is unique up to scalars. By the algebraic dependence on $(\sigma,r)$, this 
holds for all $(\sigma,r)$. Hence the choice of $\Ad_\cL (g)$ in 
\eqref{eq:2.20} is also unique up to scalars. If we can choose the 
$\Ad_\cL (g)$ such that $\gamma \mapsto J^\gamma$ is multiplicative for at 
least one value of $(\sigma,r)$, then the definition of $J^G$ shows that 
it immediately holds for all $(\sigma,r)$.

For $r = 0$ \eqref{eq:2.28} says that $J^\gamma$ intertwines 
$E^\circ_{y,\sigma,0,\rho^\circ}$ and $\gamma \cdot E^\circ_{y,\sigma,0,
\rho^\circ}$. Then it also intertwines the quotients $M^\circ_{y,\sigma,0,
\rho^\circ}$ and $\gamma \cdot M^\circ_{y,\sigma,0,\rho^\circ}$ from 
Lemma \ref{lem:2.8}. Recall that
\[
M^\circ_{y,\sigma_0,0,\rho^\circ} = 
\ind_{W_\cL^{Q^\circ} \ltimes S (\mf t^*)}^{W_\cL^\circ \ltimes S(\mf t^*)} 
(M^{Q^\circ}_{y,\sigma_0,0,\rho^\circ})
\]
where ${Q^\circ} = Z_{G^\circ}(\sigma_0)$ and $S(\mf t^*)$ acts on 
$w M^{Q^\circ}_{y,\sigma_0,0,\rho^\circ}$ via the character $w \sigma_0$. 
By \eqref{eq:2.28} 
\[
J^\gamma (w M^{Q^\circ}_{y,\sigma_0,0,\rho^\circ}) = 
(\gamma w \gamma^{-1}) M^{Q^\circ}_{y,\sigma_0,0,\rho^\circ} ,
\]
and in particular all the $J^\gamma$ restrict to elements 
\[
J^\gamma_{Q^\circ} \in \Aut_{W_\cL^{Q^\circ} \ltimes 
S(\mf t^*)}(M^{Q^\circ}_{y,\sigma_0,0,\rho^\circ}) =
\Aut_{W_\cL^{Q^\circ}}(M^{Q^\circ}_{y,\sigma_0,0,\rho^\circ}) 
= \Aut_{W_\cL^{Q^\circ}}(M^{Q^\circ}_{y,\rho^\circ}) .
\]
Here $W_\cL^{Q^\circ}$ is the Weyl group of $({Q^\circ},T)$, a group 
normalized by $\cR_{\cL,y,\sigma,\rho^\circ}$. By \cite[Proposition 4.3]{ABPS5} 
we can choose the $J^\gamma$ (which we recall are still unique up to scalars) 
such that $\gamma \mapsto J^\gamma$ is a group homomorphism (for $r = 0$). 
As we noted before, this determines a choice of all the $J^\gamma$ such that
$\gamma \mapsto J^\gamma$ is multiplicative.
\end{proof}

Recall from Theorem \ref{thm:2.9} that the quotient map $E^\circ_{y,\sigma,r,
\rho^\circ} \to M^\circ_{y,\sigma,r,\rho^\circ}$ provides a bijection between
standard modules and $\Irr (\mh H (G^\circ,L,\cL))$. Therefore Proposition 
\ref{prop:2.13} also applies to all irreducible representations of 
$\mh H (G^\circ,L,\cL)$. It expresses a regularity property of geometric 
graded Hecke algebras: the group of automorphisms $\cR_\cL$ of the Dynkin 
diagram of $(G^\circ,T)$ can be lifted to a group of intertwining operators 
between the appropriate irreducible representations. 

With Clifford theory we can obtain a first construction and classification 
of all irreducible representations of $\mh H (G,L,\cL)$:

\begin{lem}\label{lem:2.14}
There exists a bijection
\[
\begin{array}{ccc}
\Irr (\C [\cR_{\cL,y,\sigma,\rho^\circ},\natural_\cL]) & \to &
\big\{ \pi \in \Irr \big( \mh H (G,L,\cL) \big) : \pi |_{\mh H (G^\circ,L,\cL)} 
\text{ contains } M^\circ_{y,\sigma,r,\rho^\circ} \big\} \\
(\tau,V_\tau) & \mapsto & \tau \ltimes M^\circ_{y,\sigma,r,\rho^\circ} 
\end{array}.
\]
Here $\tau \ltimes M^\circ_{y,\sigma,r,\rho^\circ} = 
\ind_{\mh H (G^\circ \cR_{\cL,\sigma,y,\rho^\circ},L,\cL)}^{\mh H (G,L,\cL)}
(V_\tau \otimes M^\circ_{y,\sigma,r,\rho^\circ})$, where
$\mh H (G^\circ,L,\cL)$ acts trivially on $V_\tau$ and 
\[
N_\gamma \cdot (v \otimes m) = \tau (N_\gamma) v \otimes J^\gamma (m)
\qquad \gamma \in \cR_{\cL,y,\sigma,\rho^\circ}, v \in V_\tau, 
m \in M^\circ_{y,\sigma,r,\rho^\circ} .
\]
\end{lem}
\begin{proof}
Let the central extension $\cR_\cL^+ \to \cR_\cL$ and $p_{\natural_\cL}$ 
be as in the proof of Proposition \ref{prop:1.4}, and let $\cR^+$ be the 
inverse image of $\cR_{\cL,y,\sigma,\rho^\circ}$ in $\cR_\cL^+$. As in 
\eqref{eq:1.10}, $\mh H (G^\circ \cR_{\cL,y,\sigma,\rho^\circ},L,\cL)$ 
is the direct summand 
\[
p_{\natural_\cL} \C [\cR^+] \ltimes \mh H (G^\circ,L,\cL) 
\quad \text{ of } \quad \cR^+ \ltimes \mh H (G^\circ,L,\cL) .
\]
By Proposition \ref{prop:2.13} and Clifford theory (in the version 
\cite[Theorem 1.2]{Sol2} or \cite[p. 24]{RaRa}) there is a bijection
\[
\begin{array}{ccc}
\Irr (\cR^+) & \to &
\big\{ \pi \in \Irr \big( \cR^+ \ltimes \mh H (G^\circ,L,\cL) \big) 
: \pi |_{\mh H (G^\circ,L,\cL)} \text{ contains } 
M^\circ_{y,\sigma,r,\rho^\circ} \big\} \\
(\tau,V_\tau) & \mapsto & \ind_{\cR^+ \ltimes \mh H (G^\circ,L,\cL)
}^{\cR_\cL^+ \ltimes \mh H (G^\circ,L,\cL)} 
(V_\tau \otimes M^\circ_{y,\sigma,r,\rho^\circ})
\end{array} .
\]
Restrict this to the modules that are not annihilated by the central 
idempotent $p_{\natural_\cL}$.
\end{proof}

\subsection{Parametrization of irreducible representations} \
\label{par:param}

We start this paragraph with a few further preparatory results.
Let $(y,\sigma,\rho^\circ)$ be as before.

\begin{lem}\label{lem:2.17}
There are isomorphisms of $\pi_0 (Z_G (\sigma,y))_{\rho^\circ}$-representations
\begin{align*}
\ind_{\mh H (G^\circ,L,\cL)}^{\mh H (G^\circ \cR_{\cL,y,\sigma,\rho^\circ},L,
\cL)} (V_{\rho^\circ} \otimes E^\circ_{y,\sigma,r,\rho^\circ}) & \cong
\ind_{\pi_0 (Z_{G^\circ}(\sigma,y))}^{\pi_0 (Z_G (\sigma,y))_{\rho^\circ}}
(V_{\rho^\circ} \otimes E^\circ_{y,\sigma,r,\rho^\circ}) \\
& \cong \C[\cR_{\cL,y,\sigma,\rho^\circ},\natural_\cL^{-1}] \otimes 
V_{\rho^\circ} \otimes E^\circ_{y,\sigma,r,\rho^\circ} .
\end{align*} 
In the last line the action is
\[
g \cdot (N_w \otimes v \otimes \phi) = N_g N_w \otimes I^g (v) \otimes \phi
\]
for $g \in \pi_0 (Z_G (\sigma,y))_{\rho^\circ}, w \in 
\cR_{\cL,y,\sigma,\rho^\circ},
v \in V_{\rho^\circ}$ and $\phi \in E^\circ_{y,\sigma,r,\rho^\circ}$.
\end{lem}
\begin{proof}
Recall that $E_{y,\sigma,r} = \C_{\sigma,r} \underset{H_*^{M(y)^\circ}
(\{y\})}{\otimes} H_*^{M(y)^\circ}(\cP_y,\dot{\cL})$ and that
\begin{equation}\label{eq:2.31}
\cP_y \cap (G^\circ \cR_{\cL,y,\sigma,\rho^\circ} / P) = \cP_y^\circ \times
\cR_{\cL,y,\sigma,\rho^\circ} .
\end{equation}
There are two projective actions of $\cR_{\cL,y,\sigma,\rho^\circ}$ on 
$E_{y,\sigma,r}^{G^\circ \cR_{\cL,y,\sigma,\rho^\circ}}$. The first one 
comes from considering it as the group underlying $\C[\cR_{\cL,y,\sigma,
\rho^\circ},\natural_\cL] \subset \mh H (G,L,\cL)$, and the second one from 
considering it as a quotient of $\pi_0 (Z_G (\sigma,y))_{\rho^\circ}$. 
Both induce a simply transitive permutation of the copies of $\cP_y^\circ$ 
in \eqref{eq:2.31}, the first action by right multiplication and the second 
action by left multiplication. This implies the first stated isomorphism.

The second claim is an instance of \cite[Proposition 1.1.b]{AMS}. 
Here we use Lemma \ref{lem:2.11} to identify the 2-cocycle. 
We note that there is some choice in the second isomorphism of the lemma, 
we can still twist it by a character of $\cR_{\cL,y,\sigma,\rho^\circ}$.
\end{proof}

Notice that the twisted group algebras of $\cR_{\cL,y,\sigma,\rho^\circ}$ 
appearing in Lemmas \ref{lem:2.11} and \ref{lem:2.14} are opposite, but not
necessarily isomorphic. If $(\tau,V_\tau) \in \Irr (\C [\cR_{\cL,y,\sigma,
\rho^\circ},\natural_\cL])$, then $(\tau^*,V_\tau^*) \in (\C [\cR_{\cL,y,
\sigma,\rho^\circ},\natural_\cL^{-1}])$, where
\[
\tau^* (N_\gamma) \lambda = \lambda \circ \tau (N_\gamma^{-1}) 
\qquad \gamma \in \cR_{\cL,y,\sigma,\rho^\circ}, \lambda \in V_\tau^* .
\]
As noted in \cite[Lemma 1.3]{AMS}, this sets up a natural bijection between\\
$\Irr (\C [\cR_{\cL,y,\sigma,\rho^\circ},\natural_\cL])$ and 
$\Irr (\C [\cR_{\cL,y,\sigma,\rho^\circ},\natural_\cL^{-1}])$.

\begin{lem}\label{lem:2.15}
In the notations of Lemma \ref{lem:2.14}, there is an isomorphism of \\
$\mh H (G,L,\cL) $-modules $E_{y,\sigma,r,\rho^\circ \rtimes \tau^*} 
\cong \tau \ltimes E^\circ_{y,\sigma,r,\rho^\circ}$.
\end{lem}
\begin{proof}
By Lemma \ref{lem:2.1}
\[
E_{y,\sigma,r,\rho^\circ \rtimes \tau^*} = \Hom_{\pi_0 (Z_G (\sigma,y))}
\big( \tau^* \ltimes \rho^\circ, \ind_{\mh H (G^\circ,L,\cL}^{\mh H (G,L,\cL)} 
E^\circ_{y,\sigma,r} \big) .
\]
By Frobenius reciprocity this is isomorphic to
\begin{equation}\label{eq:2.29}
\Hom_{\pi_0 (Z_G (\sigma,y))_{\rho^\circ}} \big( \tau^* \otimes \rho^\circ, 
\ind_{\mh H (G^\circ \cR_{\cL,y,\sigma,\rho^\circ},L,\cL)}^{\mh H (G,L,\cL)} 
\ind_{\mh H (G^\circ,L,\cL)}^{\mh H (G^\circ 
\cR_{\cL,y,\sigma,\rho^\circ},L,\cL)} E^\circ_{y,\sigma,r} \big). 
\end{equation}
The action of $\pi_0 (Z_G (\sigma,y))_{\rho^\circ}$ can be constructed entirely 
within $G^\circ \cR_{\cL,y,\sigma,\rho^\circ,L,\cL}$, so we can move the first 
induction outside the brackets. Furthermore we only need the $\rho^\circ
$-isotypical part of $E^\circ_{y,\sigma,r}$, so \eqref{eq:2.29} equals
\begin{equation}\label{eq:2.30}
\ind_{\mh H (G^\circ \cR_{\cL,y,\sigma,\rho^\circ},L,\cL)}^{\mh H (G,L,\cL)} 
\Hom_{\pi_0 (Z_G (\sigma,y))_{\rho^\circ}} \big( \tau^* \otimes \rho^\circ, 
\ind_{\mh H (G^\circ,L,\cL)}^{\mh H (G^\circ \cR_{\cL,y,\sigma,\rho^\circ},
L,\cL)} V_{\rho^\circ} \otimes E^\circ_{y,\sigma,r,\rho^\circ} \big). 
\end{equation}
From Lemma \ref{lem:2.17} and \cite[Proposition 1.1.d]{AMS} we deduce that
\begin{multline}\label{eq:2.33}
\Hom_{\pi_0 (Z_G (\sigma,y))_{\rho^\circ}} \big( \tau^* \otimes \rho^\circ, 
\ind_{\mh H (G^\circ,L,\cL)}^{\mh H (G^\circ \cR_{\cL,y,\sigma,\rho^\circ},
L,\cL)} V_{\rho^\circ} \otimes E^\circ_{y,\sigma,r,\rho^\circ} \big) = \\
\Hom_{\C[\cR_{\cL,y,\sigma,\rho^\circ},\natural_\cL^{-1}]} 
\big( \tau^*, \Hom_{\pi_0 (Z_{G^\circ}(\sigma,y))}(\rho^\circ, 
\C[\cR_{\cL,y,\sigma,\rho^\circ},\natural_\cL^{-1}] \otimes V_{\rho^\circ} 
\otimes E^\circ_{y,\sigma,r,\rho^\circ} )\big) = \\
\Hom_{\C[\cR_{\cL,y,\sigma,\rho^\circ},\natural_\cL^{-1}]} 
\big( \tau^*, \C[\cR_{\cL,y,\sigma, \rho^\circ},\natural_\cL^{-1}] \otimes 
E^\circ_{y,\sigma,r,\rho^\circ} \big).
\end{multline}
Here $\C[\cR_{\cL,y,\sigma,\rho^\circ},\natural_\cL^{-1}]$ fixes
$E^\circ_{y,\sigma,r,\rho^\circ}$ pointwise.
By \cite[Lemma 1.3.c]{AMS} there is an isomorphism of 
$\C[\cR_{\cL,y,\sigma,\rho^\circ},\natural_\cL^{-1}] \times 
\C[\cR_{\cL,y,\sigma,\rho^\circ},\natural_\cL]$-modules
\begin{equation}\label{eq:2.34}
\C[\cR_{\cL,y,\sigma,\rho^\circ},\natural_\cL^{-1}] \cong 
\bigoplus_{\pi \in \Irr (\C[\cR_{\cL,y,\sigma,\rho^\circ},\natural_\cL])}
V_\pi^* \otimes V_\pi .
\end{equation}
Thus the $\mh H (G^\circ \cR_{\cL,y,\sigma,\rho^\circ},L,\cL)$-module 
\eqref{eq:2.33} becomes $V_\tau \otimes E^\circ_{y,\sigma,r,\rho^\circ}$, 
while \eqref{eq:2.29} and \eqref{eq:2.30} become
\begin{equation}\label{eq:2.36}
\ind_{\mh H (G^\circ \cR_{\cL,y,\sigma,\rho^\circ},L,\cL)}^{\mh H (G,L,\cL)} 
( V_\tau \otimes E^\circ_{y,\sigma,r,\rho^\circ} ) .
\end{equation}
The subalgebra $\mh H (G^\circ,L,\cL)$ fixes $V_\tau$ pointwise. To understand
the above $\mh H (G,L,\cL)$-module, it remains to identify the action of
$\C[\cR_{\cL,y,\sigma,\rho^\circ},\natural_\cL]$ on $V_\tau \otimes 
E^\circ_{y,\sigma,r,\rho^\circ}$. For that we return to the first line of
\eqref{eq:2.33}. Taking into account that the actions of $\pi_0 (Z_G (\sigma,y)
)_{\rho^\circ}$ and $\C[\cR_{\cL,y,\sigma,\rho^\circ},\natural_\cL] \subset 
\mh H (G^\circ \cR_{\cL,y,\sigma,\rho^\circ},L,\cL)$ commute, 
\cite[Proposition 1.1.d]{AMS} says that it is isomorphic to
\begin{equation}\label{eq:2.35}
\Hom_{\C[\cR_{\cL,y,\sigma,\rho^\circ}, \natural_\cL^{-1}]} \big( \tau^*, 
\C[\cR_{\cL,y,\sigma,\rho^\circ},\natural_\cL] \otimes 
E^\circ_{y,\sigma,r,\rho^\circ} \big) .
\end{equation}
We have seen in \eqref{eq:2.33} that $\C[\cR_{\cL,y,\sigma,\rho^\circ},
\natural_\cL^{-1}]$ fixes $E^\circ_{y,\sigma,r,\rho^\circ}$ pointwise, and
we know from Theorem \ref{thm:2.4}.d that its action commutes with $\Delta 
(\mh H (G^\circ \cR_{\cL,y,\sigma,\rho^\circ},L,\cL))$. The proof of Lemma 
\ref{lem:2.17} entails that, up to a scalar which depends only on $\gamma
\in \cR_{\cL,y,\sigma,\rho^\circ}$,
\[
N_\gamma \cdot (N \otimes \phi) = N N_\gamma^{-1} \otimes \phi \qquad
N \in \C[\cR_{\cL,y,\sigma,\rho^\circ},
\natural_\cL], \phi \in E^\circ_{y,\sigma,r,\rho^\circ} .
\]
Since this formula already defines an action, the family of scalars (for 
various $\gamma$) must form a character of $\cR_{\cL,y,\sigma,\rho^\circ}$. 
We can make this character trivial by adjusting the choice of the second 
isomorphism in Lemma \ref{lem:2.17}, which means that $\C[\cR_{\cL,y,\sigma,
\rho^\circ},\natural_\cL]$ in \eqref{eq:2.35} becomes a bimodule in the 
standard manner. By \eqref{eq:2.34} for $\C[\cR_{\cL,y,\sigma,\rho^\circ},
\natural_\cL]$, \eqref{eq:2.35} is isomorphic, as $\C[\cR_{\cL,y,\sigma,
\rho^\circ},\natural_\cL]$-module, to $V_\tau \otimes E^\circ_{y,\sigma,r,
\rho^\circ}$. Consequently the $\mh H (G,L,\cL)$ is endowed with the expected 
action of $\C[\cR_{\cL,y,\sigma,\rho^\circ},\natural_\cL]$ on $V_\tau$, 
which means that it can be identified with 
$\tau \ltimes E^\circ_{y,\sigma,r,\rho^\circ}$.
\end{proof}

It will be useful to improve our understanding of standard modules with $r = 0$,
like in Lemma \ref{lem:2.6}.

\begin{lem}\label{lem:2.24}
The $\mh H (G,L,\cL)$-module $E_{y,\sigma,0,\rho}$ is completely reducible
and can be decomposed along the homological degree:
\[
E_{y,\sigma,0,\rho} = \bigoplus\nolimits_n \Hom_{\pi_0 (Z_G (\sigma,y))} 
\big( \rho, H_n (\mc P_y ,\dot{\cL}) \big) . 
\]
\end{lem}
\begin{proof}
By Lemma \ref{lem:2.1}
\begin{equation}\label{eq:2.71}
E_{y,\sigma,0} \cong \ind_{\mh H (G^\circ,L,\cL)}^{\mh H (G,L,\cL)}
E^\circ_{y,\sigma,0} .
\end{equation}
From Lemma \ref{lem:2.6} we know that $E^\circ_{y,\sigma,0}$ is completely
reducible. As $\mh H (G^\circ,L,\cL) = \C [\cR_\cL,\natural_\cL] \ltimes 
\mh H (G^\circ,L,\cL)$ where $\C [\cR_\cL,\natural_\cL]$ is a twisted group
algebra of a finite group acting on $\mh H (G^\circ,L,\cL)$, the induction in 
\eqref{eq:2.71} preserves complete reducibility. 

From Lemma \ref{lem:2.6} we know that $E^\circ_{y,\sigma,0} =
\bigoplus_n H_n (\mc P_y^\circ,\dot{\cL})$. The proof of Lemma \ref{lem:2.1}
shows that 
\[
\ind_{\mh H (G^\circ,L,\cL)}^{\mh H (G,L,\cL)} H_n (\mc P_y^\circ,\dot{\cL})
\cong \C [\cR_\cL,\natural_\cL] \otimes_\C H_n (\mc P_y^\circ,\dot{\cL}) \cong
H_n (\mc P_y,\dot{\cL}) ,
\]
so $E_{y,\sigma,0} = \bigoplus_n H_n (\mc P_y,\dot{\cL})$ as 
$\mh H(G,L,\cL)$-modules.
Since the action of $\pi_0 (Z_G (\sigma,y))$ commutes with that of 
$\mh H (G,L,\cL)$, $E_{y,\sigma,0,\rho} = \Hom_{\pi_0 (Z_G (\sigma,y))} (\rho,
E_{y,\sigma,0})$ is also completely reducible, and the decomposition
according to homological degree persists in $E_{y,\sigma,0,\rho}$.
\end{proof}

We note that the definitions \eqref{eq:2.74} and \eqref{eq:2.18} also
can be used with $G$ instead of $G^\circ$, provided that one involves
the generalized Springer correspondence for disconnected groups from
\cite[\S 4]{AMS}. In this way we define the $\mh H (G,L,\cL)$-module
$M_{y,\sigma_0,0,\rho}$.

Now we are ready to prove the main result of this section. It generalizes
\cite[Corollary 8.18]{Lus5} to disconnected groups $G$.
Recall that Condition \ref{cond:1} is in force.

\begin{thm}\label{thm:2.16}
Let $y \in \mf g$ be nilpotent and let $(\sigma,r)/\!\!\sim \; \in V_y$ 
be semisimple.
Let $\rho \in \Irr \big(\pi_0 (Z_G (\sigma,y))\big)$ be such that 
$\Psi_{Z_G (\sigma_0)}(y,\rho) = (L,\cC_v^L,\cL)$ (up to $G$-conjugation).
\enuma{
\item If $r \neq 0$, then $E_{y,\sigma,r,\rho}$ has a unique irreducible
quotient $\mh H (G,L,\cL)$-module. We call it $M_{y,\sigma,r,\rho}$.
\item If $r = 0$, then $E_{y,\sigma_0,r,\rho}$ has a unique 
irreducible subquotient isomorphic to $M_{y,\sigma_0,0,\rho}$. This subquotient
is the component of $E_{y,\sigma_0,r,\rho}$ in one homological degree
(as in Lemma \ref{lem:2.24}).
\item Parts (a) and (b) set up a canonical bijection between $\Irr_{r}
(\mh H (G,L,\cL))$ and the $G$-orbits of triples $(y,\sigma,\rho)$ as above.
\item The two sets from part (c) are canonically in bijection with the 
collection of $G$-orbits of triples $(y,\sigma_0,\rho)$ as in Proposition 
\ref{prop:2.7}. 
(The only difference is that $\sigma_0 \in Z_{\mf g}(y)$ instead of 
$(\sigma,r) \in \mathrm{Lie}(Z_{G \times \C^\times}(y))$. That is, $(y,
\sigma_0,\rho)$ is obtained from $(y,\sigma,\rho)$ via Lemma \ref{lem:2.2}.) 
} 
\end{thm}
\begin{proof}
Let $\rho^\circ$ be an irreducible constituent of $\rho |_{\pi_0 (Z_{G^\circ}
(\sigma,y))}$. By Lemma \ref{lem:2.11} there is a unique $\tau^* \in \Irr 
(\C [\cR_{\cL,y,\sigma,\rho^\circ}, \natural_\cL^{-1}])$ such that 
$\rho \cong \rho^\circ \rtimes \tau^*$.\\
(a) From Lemma \ref{lem:2.15} we know that 
\[
E_{y,\sigma,r,\rho} \cong \tau \ltimes E^\circ_{y,\sigma,r,\rho^\circ} .
\]
By Lemma \ref{lem:2.14} it has the irreducible quotient 
\begin{equation}\label{eq:2.43}
\tau \ltimes M^\circ_{y,\sigma,r,\rho^\circ} =  
(\tau \ltimes E^\circ_{y,\sigma,r,\rho^\circ}) / (\tau \ltimes N^\circ ) ,
\end{equation}
where $N^\circ = \ker (E^\circ_{y,\sigma,r,\rho^\circ} \to M^\circ_{y,\sigma,
r,\rho^\circ})$. Hence $\tau \ltimes N^\circ$ is a maximal proper submodule of 
$E_{y,\sigma,r,\rho^\circ}$. We define 
\begin{align*}
& I_M = \{ V \in \Irr_r (\mh H (G^\circ,L,\cL)) : V \text{ is a constituent of }
\tau \ltimes M^\circ_{y,\sigma,r,\rho^\circ} \} , \\
& I_N = \{ V \in \Irr_r (\mh H (G^\circ,L,\cL)) : V \text{ is a constituent of }
\tau \ltimes N^\circ \} . 
\end{align*}
Recall from Theorem \ref{thm:2.9}.d that all the irreducible 
$\mh H (G^\circ,L,\cL)$-constituents of $N^\circ$ are of the form
$M^\circ_{y',\sigma',r,\rho'^\circ}$, where $\dim \cC_{y'}^{G^\circ} >
\dim \cC_y^{G^\circ}$. Since $\cR_\cL$ acts by algebraic automorphisms on
$G^\circ$, the same holds for all $\mh H (G^\circ,L,\cL)$-constituents of
$\tau \ltimes N^\circ$. Hence $I_M$ and $I_N$ are disjoint. Moreover these
sets are finite, so by Wedderburn's theorem about irreducible representations
the canonical map
\[
\mh H (G^\circ,L,\cL) \to \bigoplus\nolimits_{V \in I_M} \End (V) \oplus
\bigoplus\nolimits_{V \in I_N} \End (V)
\]
is surjective. In particular there exists an element of $\mh H (G^\circ,L,
\cL)$ which annihilates all $V \in I_N$ and fixes all $V \in I_M$ pointwise.
By Theorem \ref{thm:2.4}.d $\tau \ltimes N^\circ$ has finite length, so a 
suitable power $h^\circ$ of that element annihilates $\tau \ltimes N^\circ$. 
Since $\tau \ltimes M^\circ_{y,\sigma,r,\rho^\circ}$ is completely reducible as
$\mh H (G^\circ,L,\cL)$-module, $h^\circ$ acts as the identity on it.

Choose a basis $\cB$ of $\C [\cR_\cL,\natural_\cL] \underset{\C[\cR_{\cL,y,
\sigma,\rho^\circ},\natural_\cL]}{\otimes} V_\tau$, consisting of elements
of the form $b = N_\gamma \otimes v$ with $\gamma \in \cR_\cL$
and $v \in V_\tau$. Since $\tau \ltimes M^\circ_{y,\sigma,r,\rho^\circ}$ is
irreducible, we can find for all $b,b' \in \cB$ an element $h_{b b'} \in 
\mh H (G,L,\cL)$ which maps $b' M^\circ_{y,\sigma,r,\rho^\circ}$ bijectively 
to $b M^\circ_{y,\sigma,r,\rho^\circ}$ and annihilates all the other 
subspaces $b'' M^\circ_{y,\sigma,r,\rho^\circ}$. 

Consider any $x \in E_{y,\sigma,r,\rho^
\Hom_{\pi_0 (Z_{G^\circ} 
(s,u))} (\rho^\circ, H_{d(u)} (\mc B_{G^\circ}^{s,u},\C) )\circ} \setminus \tau \ltimes 
N^\circ$. Write it in terms of $\cB$ as $x = \sum_{b \in \cB} b \otimes x_b$
with $x_b \in E^\circ_{y,\sigma,r,\rho}$. For at least one $b' \in \cB ,\, 
x_{b'} \in E^\circ_{y,\sigma,r,\rho} \setminus N^\circ$. Then 
\[
h^\circ h_{b b'} x = b \otimes v' \text{ for some } 
v' \in E^\circ_{y,\sigma,r,\rho^\circ} \setminus N^\circ .
\]
As $\mh H (G^\circ,L,\cL)$-representation 
\[
b E^\circ_{y,\sigma,r,\rho^\circ} = (N_\gamma \otimes v) 
E^\circ_{y,\sigma,r,\rho^\circ} \cong \gamma \cdot 
E^\circ_{y,\sigma,r,\rho^\circ} ,
\]
which has the unique maximal proper submodule $\gamma \cdot N^\circ 
\cong b N^\circ$. Hence 
\[
\mh H (G^\circ,L,\cL) h^\circ h_{b b'} x = b E^\circ_{y,\sigma,r,\rho^\circ} .
\]
This works for every $b \in \cB$, so $\mh H (G,L,\cL) x = E_{y,\sigma,r,\rho}$.
Consequently there is no other maximal proper submodule of 
$E_{y,\sigma,r,\rho}$ besides $\tau \ltimes N^\circ$.\\
(b) Put $Q = Z_G (\sigma_0)$. By Lemma \ref{lem:2.1} and \eqref{eq:2.12} 
\[
E_{y,\sigma_0,0} \cong \ind_{\mh H (Q,L,\cL)}^{\mh H (G,L,\cL)} 
E^Q_{y,\sigma_0,0} = \ind_{\mh H (Q^\circ,L,\cL)}^{\mh H (G,L,\cL)} 
E^{Q^\circ}_{y,\sigma_0,0} .
\]
Now Theorem \ref{thm:2.4}.d and \eqref{eq:2.37} (but for $G$) imply 
\[
E_{y,\sigma_0,0,\rho} \cong 
\ind_{\mh H (Q,L,\cL)}^{\mh H (G,L,\cL)} E^Q_{y,\sigma_0,0,\rho}.  
\]
By Lemma \ref{lem:2.15} $E^Q_{y,\sigma_0,0,\rho^\circ \rtimes \tau^*} 
\cong \tau \ltimes E^{Q^\circ}_{y,\sigma_0,0,\rho^\circ}$, whereas 
\cite[(54)]{AMS} shows that $M^Q_{y,\rho^\circ \rtimes \tau^*} \cong \tau 
\ltimes M^{Q^\circ}_{y,\rho^\circ}$ as $\C[W_\cL^Q,\natural_\cL^Q]$-modules.
Decreeing that $S(\mf t^*)$ acts trivially on $V_\tau$, we obtain an 
isomorphism of $\C[W_\cL^Q,\natural_\cL^Q] \ltimes S (\mf t^*)$-modules
\begin{equation}\label{eq:2.56}
M^Q_{y,\sigma_0,0,\rho^\circ \rtimes \tau^*} \cong \tau \ltimes 
M^{Q^\circ}_{y,\sigma_0,0,\rho^\circ} .
\end{equation}
From Lemma \ref{lem:2.8} we know that $ E^{Q^\circ}_{y,\sigma_0,0,
\rho^\circ}$ has a direct summand isomorphic to $M^{Q^\circ}_{y,\sigma_0,0,
\rho^\circ}$. Hence there is a surjective $\mh H (G,L,\cL)$-module map
\begin{equation}\label{eq:2.39}
E_{y,\sigma_0,0,\rho} \cong \ind_{\mh H (Q,L,\cL)}^{\mh H (G,L,\cL)} 
(\tau \ltimes E^{Q^\circ}_{y,\sigma_0,0,\rho^\circ}) \to
\ind_{\mh H (Q,L,\cL)}^{\mh H (G,L,\cL)} (\tau 
\ltimes M^{Q^\circ}_{y,\sigma_0,0,\rho^\circ}) \cong M_{y,\sigma_0,0,\rho}.
\end{equation}
The same argument as for part (a) shows that there exists a $h^\circ \in 
\mh H (G^\circ,L,\cL)$ which annihilates $\ker (E_{y,\sigma_0,0,\rho} \to
M_{y,\sigma_0,0,\rho})$ and acts as the identity on $M_{y,\sigma_0,0,\rho}$. 
Therefore $M_{y,\sigma_0,0,\rho}$ appears with multiplicity one in 
$E_{y,\sigma_0,0,\rho}$. By the complete reducibility from Lemma \ref{lem:2.24}, 
it appears as a direct summand.

Recall from from Lemma \ref{lem:2.15} and \eqref{eq:2.38} that there are
isomorphisms of $\mh H (G,L,\cL)$-modules 
\[
E_{y,\sigma_0,0,\rho} \cong \tau \ltimes E^\circ_{y,\sigma_0,0,\rho^\circ} 
\cong \tau \ltimes \ind_{\mh H (Q^\circ,L,\cL)}^{\mh H (G,L,\cL)} 
E^{Q^\circ}_{y,\sigma_0,0,\rho^\circ} .
\]
From these and \eqref{eq:2.39} we deduce
\begin{equation}\label{eq:2.72}
M_{y,\sigma,0,\rho^\circ \rtimes \tau^*} \cong 
\tau \ltimes M^\circ_{y,\sigma,0,\rho^\circ} .
\end{equation}
Combining these with Lemma \ref{lem:2.8}, we see that 
$M_{y,\sigma,0,\rho^\circ \rtimes \tau^*}$ is the component of 
$E_{y,\sigma_0,0,\rho}$ in one homological degree.\\
(c) For $r \neq 0$, part (a) and Lemma \ref{lem:2.15} 
induce an isomorphism of $\mh H (G,L,\cL)$-modules 
\begin{equation}\label{eq:2.40}
M_{y,\sigma,r,\rho^\circ \rtimes \tau^*} \cong 
\tau \ltimes M^\circ_{y,\sigma,r,\rho^\circ} .
\end{equation}
From Lemma \ref{lem:2.14} we see that the irreducible modules \eqref{eq:2.40} 
and \eqref{eq:2.72} exhaust \\ $\Irr (\mh H (G,L,\cL))$. By \cite[Theorem 1.2]{AMS} 
and \cite[Theorem 1.2]{Sol2} two such representations are isomorphic if and only 
if there is a $\gamma \in \cR_\cL$ such that 
\begin{equation}\label{eq:2.41}
M^\circ_{y,\sigma,r,\rho^\circ} \cong \gamma \cdot M^\circ_{y',\sigma',
r,\rho'^\circ} \quad \text{and} \quad \tau \cong \gamma \cdot \tau' .
\end{equation}
By Theorem \ref{thm:2.9}.c the first isomorphism means that $(y,\sigma,
\rho^\circ)$ and $(y',\sigma',\rho'^\circ)$ are $G$-conjugate, while the 
second is equivalent to $\tau^*$ and $\tau'^*$ being associated under the 
action of $G / G^\circ$. With Lemma \ref{lem:2.11} we see that \eqref{eq:2.41} 
is equivalent to: 
\[
(y,\sigma,\rho = \rho^\circ \rtimes \tau^*) \quad \text{and} \quad (y',\sigma',
\rho' = \rho'^\circ \rtimes \tau'^*) \quad \text{are } G\text{-conjugate.}
\]
This yields the bijection between $\Irr (\mh H (G,L,\cL))$ and the indicated set 
of parameters. It is canonical because $M_{y,\sigma,r,\rho}$ does not depend 
on any arbitrary choices, in particular the 2-cocycles from the 
previous paragraph do not appear in $\rho$.\\
(d) Apply Lemma \ref{lem:2.2}.b to part (c).
\end{proof}

Recall that all the above was proven under Condition \ref{cond:1}. Now we
want to lift this condition, so we consider a group $G$ which does not
necessarily equal $G^\circ N_G (P,\cL)$. In \eqref{eq:1.12} we saw that
$\mh H (G,L,\cL)$ remains as in this section, but the parameters for irreducible
representations could change when we replace $G^\circ N_G (P,\cL)$ by $G$.

\begin{lem}\label{lem:2.19}
The parametrizations of $\Irr_r (\mh H (G,L,\cL))$ obtained in Theorem
\ref{thm:2.16} remain valid without Condition \ref{cond:1}.
\end{lem}
\begin{proof}
By the definition of $N_G (P,\cL)$, no element of $G \setminus G^\circ N_G (P,\cL)$
can stabilize the $G^\circ N_G (P,\cL)$-orbit of $(L,\cC_v^L,\cL)$. So,
when we replace $G^\circ N_G (P,\cL)$ by $G$, the orbit of the cuspidal support 
$(L,\cC_v^L,\cL)$ becomes $[G : G^\circ N_G (P,\cL)]$ times larger. More
precisely, $G \cdot (L,\cC_v^L,\cL)$ can be written as a disjoint union of
$[G : G^\circ N_G (P,\cL)]$ orbits for $G^\circ N_G (P,\cL)$ with representatives
$(L,\cC_v^L,\cL')$, where $\cL' = \Ad (g)^* \cL$ for some $g \in N_G (P,L)$.

Let $(y,\sigma_0,\rho)$ be as in Theorem \ref{thm:2.16}, for the group
$G^\circ N_G (P,\cL)$. By Theorem \ref{thm:2.16}.d the stabilizer of 
$G^\circ N_G (P,\cL) \cdot (y,\sigma_0,\rho)$ in $G$ equals that of 
$G^\circ N_G (P,\cL) \cdot (L,\cC_v^L,\cL)$, so it is $G^\circ N_G (P,\cL)$.
In particular the $Z_G (\sigma_0,y)$-stabilizer of $\rho$ is precisely
$Z_{G^\circ N_G (P,\cL)}(\sigma_0,y)$, which implies that 
\[
\rho^+ = \ind_{Z_{G^\circ N_G (P,\cL)}(\sigma_0,y)}^{Z_G (\sigma_0,y)} (\rho)
\]
is an irreducible $\pi_0 (Z_G (\sigma_0,y))$-representation. By 
\cite[Theorem 4.8.a]{AMS} 
\[
\Psi_{Z_G (\sigma_0)}(y,\rho^+) = (L,\cC_v^L,\cL) 
\quad \text{up to } G\text{-conjugation.}
\]
From $G \cdot (\sigma_0,y,\rho^+)$ we can recover
$G^\circ N_G (P,\cL) \cdot (y,\sigma_0,\rho)$ as the unique 
$G^\circ N_G (P,\cL)$-orbit contained in it with cuspidal support $(L,\cC_v^L,\cL)$
up to $G^\circ N_G (P,\cL)$-conjugation. 

Consequently the canonical map $(y,\sigma_0,\rho) \mapsto (y,\sigma_0,\rho^+)$
provides a bijection between the triples in Theorem \ref{thm:2.16}.d for 
$G^\circ N_G (P,\cL)$, and the same triples for $G$. With Lemma \ref{lem:2.2}
(which is independent of Condition \ref{cond:1}) we can replace $(y,\sigma_0,\rho^+)$
by $(y,\sigma,\rho^+)$, obtaining the same triples as in Theorem \ref{thm:2.16}.c,
but for $G$.
\end{proof}

In Theorem \ref{thm:2.24} we showed that the assignment $(\sigma,y,r) \mapsto
E_{y,\sigma,r}$ is compatible with parabolic induction. That cannot be true for
the modules $E_{y,\sigma,r,\rho}$, if only because $\rho$ is not a correct part
of the data when $G$ is replaced by a Levi subgroup. Nevertheless a weaker 
version of Theorem \ref{thm:2.24} holds for $E_{y,\sigma,r,\rho}$ and 
$M_{y,\sigma,r,\rho}$.

Let $Q \subset G$ be an algebraic subgroup such that $Q \cap G^\circ$ is a Levi
subgroup of $G^\circ$ and $L \subset Q^\circ$. Let $y,\sigma,r,\rho$ be as in
Theorem \ref{thm:2.16}, with $\sigma,y \in \mf q = \Lie (Q)$. By \cite[\S 3.2]{Ree}
the natural map
\begin{equation}
\pi_0 (Z_Q (\sigma,y)) = \pi_0 ( Z_{Q \cap Z_G (\sigma_0)}(y)) \to
\pi_0 (Z_{Z_G (\sigma_0)}(y)) = \pi_0 (Z_G (\sigma,y)) 
\end{equation}
is injective, so we can consider the left hand side as a subgroup of the right
hand side. Let $\rho^Q \in \Irr \big( \pi_0 (Z_Q (\sigma,y)) \big)$ be such that 
$\Psi_{Z_Q (\sigma_0)} (y,\rho^Q) = (L,\cC_v^L,\cL)$. Then $E_{y,\sigma,r,\rho}, 
M_{y,\sigma,r,\rho}, E^Q_{y,\sigma,r,\rho^Q}$ and $M^Q_{y,\sigma,r,\rho^Q}$ are 
defined. 

\begin{prop}\label{prop:2.25}
\textbf{Erratum.} \texttt{For this proposition to hold (with the same proof) we
need an extra condition $r=0$ or $\epsilon (\sigma,r) \neq 0$, see the appendix.}
\enuma{
\item There is a natural isomorphism of $\mh H (G,L,\cL)$-modules
\[
\mh H (G,L,\cL) \underset{\mh H (Q,L,\cL)}{\otimes} E^Q_{y,\sigma,r,\rho^Q} \cong 
\bigoplus\nolimits_\rho \Hom_{\pi_0 (Z_Q (\sigma,y))} (\rho^Q ,\rho) \otimes 
E_{y,\sigma,r,\rho} ,
\]
where the sum runs over all $\rho \in \Irr \big( \pi_0 (Z_G (\sigma,y)) \big)$
with $\Psi_{Z_G (\sigma_0)}(y,\rho) = (L,\cC_v^L,\cL)$.
\item For $r=0$ part (a) contains an isomorphism of $S(\mf t^*) \rtimes \C[W_\cL ,
\natural_\cL]$-modules
\[
\mh H (G,L,\cL) \underset{\mh H (Q,L,\cL)}{\otimes} M^Q_{y,\sigma,0,\rho^Q} \cong 
\bigoplus\nolimits_\rho \Hom_{\pi_0 (Z_Q (\sigma,y))} (\rho^Q ,\rho) \otimes 
M_{y,\sigma,0,\rho} .
\]
\item The multiplicity of $M_{y,\sigma,r,\rho}$ in $\mh H (G,L,\cL) 
\underset{\mh H (Q,L,\cL)}{\otimes} E^Q_{y,\sigma,r,\rho^Q}$ is
$[\rho^Q : \rho]_{\pi_0 (Z_Q (\sigma,y))}$. \\
It already appears that many times as a 
quotient, via $E^Q_{y,\sigma,r,\rho^Q} \to M^Q_{y,\sigma,r,\rho^Q}$. 
More precisely, there is a natural isomorphism
\[
\Hom_{\mh H (Q,L,\cL)} (M^Q_{y,\sigma,r,\rho^Q}, M_{y,\sigma,r,\rho}) \cong
\Hom_{\pi_0 (Z_Q (\sigma,y))} (\rho^Q ,\rho)^* .
\]
} 
\end{prop}
\textbf{Remark.} When we set $(\sigma,r) = (0,0)$, part (b) gives a natural
isomorphism of $\C[W_\cL ,\natural_\cL]$-modules
\[
\C[W_\cL ,\natural_\cL] \underset{\C[W_\cL^Q ,\natural_\cL]}{\otimes} 
M^Q_{y,\rho^Q} \cong \bigoplus\nolimits_\rho \Hom_{\pi_0 (Z_Q (y))} 
(\rho^Q ,\rho) \otimes M_{y,\rho} .
\]
Consequently $[M^Q_{y,\rho^Q} : M_{y,\rho}]_{\C[W_\cL^Q ,\natural_\cL]} =
[\rho^Q : \rho]_{\pi_0 (Z_Q (y))}$. As the modules $M_{y,\rho}$ and $M^Q_{y,\rho^Q}$
are obtained with the generalized Springer correpondence for disconnected groups from 
\cite[Theorem 4.7]{AMS}, this solves the issue with the multiplicities mentioned
in \cite[Theorem 4.8.b]{AMS}.
\begin{proof}
(a) By Theorem \ref{thm:2.24}.b
\begin{equation}\label{eq:2.57}
\mh H (G,L,\cL) \underset{\mh H (Q,L,\cL)}{\otimes} E^Q_{y,\sigma,r,\rho^Q} =
\Hom_{\pi_0 (Z_Q (\sigma,y))} (\rho^Q, E_{y,\sigma,r}) .
\end{equation}
With Frobenius reciprocity we can rewrite this as
\[
\Hom_{\pi_0 (Z_Q (\sigma,y))} \big( \ind_{\pi_0 (Z_Q (\sigma,y))}^{\pi_0 
(Z_G (\sigma,y))} \rho^Q, E_{y,\sigma,r} \big) = 
\big( \big( \ind_{\pi_0 (Z_Q (\sigma,y))}^{\pi_0 (Z_G (\sigma,y))
} V_{\rho^Q} \big)^* \otimes E_{y,\sigma,r} \big)^{\pi_0 (Z_G (\sigma,y))} .
\]
Similarly $E_{y,\sigma,r,\rho} = (V_\rho^* \otimes E_{y,\sigma,r} 
)^{\pi_0 (Z_G (\sigma,y))}$. Again by Frobenius reciprocity
\begin{multline}\label{eq:2.66}
\Hom_{\pi_0 (Z_G (\sigma,y))} \big( V_\rho^*,\big( \ind_{\pi_0 (Z_Q (\sigma,y))}^{
\pi_0 (Z_G (\sigma,y))} V_{\rho^Q} \big)^* \big) = \\
\Hom_{\pi_0 (Z_G (\sigma,y))} \big( \ind_{\pi_0 (Z_Q (\sigma,y))}^{
\pi_0 (Z_G (\sigma,y))} \rho^Q , \rho \big) 
= \Hom_{\pi_0 (Z_Q (\sigma,y))} (\rho^Q ,\rho) .
\end{multline}
(b) Now we assume that $r = 0$. From Theorem \ref{thm:2.16}.c we know that 
$M_{y,\sigma,0,\rho}$ is the component of $E_{y,\sigma,0,\rho}$ in one homological degree.
By Lemma \ref{lem:2.8} this degree, say $n^G$, does not depend on $\rho$. Similarly 
\[
M^Q_{y,\sigma,0,\rho^Q} = \Hom_{\pi_0 (Z_Q (\sigma,y))} 
\big( \rho^Q, H_{n^Q} (\mc P_y^Q ,\dot{\cL}) \big) .
\]
The isomorphism in part (a) comes eventually from Theorem \ref{thm:2.24}.b, so by 
\eqref{eq:2.67} it changes all homological degrees by a fixed amount 
$d = \dim \mc P_y - \dim P_y^Q$. Thus part (a) restricts to
\begin{align}
\nonumber \mh H (G,L,\cL) \underset{\mh H (Q,L,\cL)}{\otimes} M^Q_{y,\sigma,r,\rho^Q} 
= \mh H (G,L,\cL) \underset{\mh H (Q,L,\cL)}{\otimes} 
\Hom_{\pi_0 (Z_Q (\sigma,y))} \big( \rho^Q, H_{n^Q} (\mc P_y^Q,\dot{\cL}) \big) \\
\label{eq:2.77} \cong \bigoplus\nolimits_\rho \Hom_{\pi_0 (Z_Q (\sigma,y))} (\rho^Q ,\rho) 
\otimes \Hom_{\pi_0 (Z_G (\sigma,y))} \big( \rho, H_{n^Q + d} (\mc P_y,\dot{\cL}) \big) 
\end{align}
We want to show that $n^Q + d = n^G$, for then \eqref{eq:2.77} becomes the desired 
isomorphism. This is easily seen from the explicit formula given in Lemma \ref{lem:2.8},
but we prefer an argument that does not use \cite{LusCS}.
Since $n^G$ does not depend on $\rho$, it suffices to consider one $\rho$.
By \cite[Theorem 4.8.a]{AMS} we can pick $\rho$ such that $\Hom_{\pi_0 (Z_Q (\sigma,y))} 
(\rho^Q ,\rho) \neq 0$, while maintaining the condition on the cuspidal support.

By \eqref{eq:2.56}, \eqref{eq:2.39} and \eqref{eq:2.63} the $(\sigma,0)$-weight 
space of $M_{y,\sigma,0,\rho^\circ \rtimes \tau^*}$ is
\[
\tau \ltimes M_{y,\rho^\circ} \in \Irr (\C [W_{\cL,\sigma},\natural_\cL]) .
\]
For the same reasons the $(\sigma,0)$-weight space of \eqref{eq:2.77} is
\[
\ind_{\C [W_{\cL,\sigma}^Q,\natural_\cL]}^{\C 
[W_{\cL,\sigma},\natural_\cL]} (\tau^Q \ltimes M_{y,\rho^{Q^\circ}} ) \in 
\Mod ( \C [W_{\cL,\sigma},\natural_\cL] ) .
\]
Here $\tau \ltimes M_{y,\rho^\circ}$ is the representation attached to
$(y,\rho = \rho^\circ \rtimes \tau^*)$ by the generalized Springer correspondence
for $Z_G (\sigma)$ from \cite[\S 4]{AMS}. In the same way, only for $Z_Q (\sigma)$,
$\tau^Q \ltimes M_{y,\rho^{Q^\circ}}$ is related to $(y,\rho^Q = \rho^{Q^\circ}
\rtimes \tau^Q)$. 

As $\rho^Q$ appears in $\rho$, \cite[Proposition 4.8.b]{AMS} guarantees that 
$\tau^Q \ltimes M_{y,\rho^{Q^\circ}}$ appears in $\tau \ltimes M_{y,\rho^\circ}$.
Hence the $\C [W_{\cL,\sigma},\natural_\cL]$-module $\tau \ltimes M_{y,\rho^\circ}$
appears in \eqref{eq:2.77}. In view of the irreducibility of $M_{y,\sigma,0,\rho}$,
this implies that $M_{y,\sigma,0,\rho}$ is a quotient of 
\[
\Hom_{\pi_0 (Z_G (\sigma,y))} \big( \rho, H_{n^Q + d} (\mc P_y,\dot{\cL}) \big)
\subset E_{y,\sigma,0,\rho} .
\]
By Theorem \ref{thm:2.9}.d and \eqref{eq:2.72}, this is only possible if 
$n^Q + d = n^G$.\\
(c) From Theorem \ref{thm:2.16} we know that $M_{y,\sigma,r,\rho}$ appears with
multiplicity one in $E_{y,\sigma,r,\rho}$. It follows from \cite[Corollary 10.7 and 
Proposition 10.12]{Lus5} that all other irreducible constituents of the standard
module $E_{y,\sigma,r,\rho}$ are of the form $M_{y',\sigma,r,\rho'}$, where 
$\cC_{y'}^G$ is a nilpotent orbit of larger dimension then $\cC_y^G$. Together with
part (a) this shows the indicated multiplicity is 
\[
\dim \Hom_{\pi_0 (Z_Q (\sigma,y))} (\rho^Q,\rho) = 
[\rho^Q : \rho]_{\pi_0 (Z_Q (\sigma,y))} .
\]
Now we assume that $r \neq 0$, so that $M_{y,\sigma,r,\rho}$ is the 
unique irreducible quotient of $E_{y,\sigma,r,\rho}$. 

For every $\rho$ as with $\Psi_{Z_G (\sigma_0)}(y,\rho) = (L,\cC_v^L,\cL)$ we choose 
an element $f_\rho \in E_{y,\sigma,r,\rho}$ with $f_\rho \neq 0$ in $M_{y,\sigma,r,\rho}$, 
and we choose a basis $\{ b_{\rho,i} \}_i$ of $\Hom_{\pi_0 (Z_Q (\sigma,y))} (\rho^Q ,\rho)$.
The set $F := \{ b_{\rho,i} \otimes f_\rho \}_{\rho,i}$ generates the right hand 
side of part (a) as a $\mh H (G,L,\cL)$-module, and no proper subset of it
has the same property. Via the canonical isomorphism of part (a) we consider
$F$ as a subset of the left hand side. Suppose that one element $b_{\rho,i} \otimes 
f_\rho$ belongs to
\begin{equation}\label{eq:2.59}
\mh H (G,L,\cL) \underset{\mh H (Q,L,\cL)}{\otimes} 
\ker \big( E^Q_{y,\sigma,r,\rho^Q} \to M^Q_{y,\sigma,r,\rho^Q} \big) .
\end{equation}
The remaining elements of $F$ generate $\mh H (G,L,\cL) \underset{\mh H (Q,L,\cL)}{
\otimes} M^Q_{y,\sigma,r,\rho^Q}$. Since $M^Q_{y,\sigma,r,\rho^Q}$ is the 
unique irreducible quotient of $E^Q_{y,\sigma,r,\rho^Q}$, they also generate the 
modules in part~(a). This contradiction shows that all elements of $F$ are nonzero in\\
$\mh H (G,L,\cL) \underset{\mh H (Q,L,\cL)}{\otimes} M^Q_{y,\sigma,r,\rho^Q}$,
and that \eqref{eq:2.59} is contained in
\[
\bigoplus\nolimits_\rho \Hom_{\pi_0 (Z_Q (\sigma,y))} (\rho^Q ,\rho) \otimes 
\ker \big( E_{y,\sigma,r,\rho} \to M_{y,\sigma,r,\rho} \big) . 
\]
Consequently the canonical surjection
\[
\mh H (G,L,\cL) \underset{\mh H (Q,L,\cL)}{\otimes} E^Q_{y,\sigma,r,\rho^Q} \to
\bigoplus\nolimits_\rho \Hom_{\pi_0 (Z_Q (\sigma,y))} (\rho^Q ,\rho) \otimes 
M_{y,\sigma,r,\rho} 
\]
factors through $\mh H (G,L,\cL) \underset{\mh H (Q,L,\cL)}{\otimes} 
M^Q_{y,\sigma,r,\rho^Q}$. We deduce natural isomorphisms
\begin{multline*}
\Hom_{\pi_0 (Z_Q (\sigma,y))} (\rho^Q ,\rho)^*  \cong \Hom_{\mh H (G,L,\cL)} 
\big( \mh H (G,L,\cL) \underset{\mh H (Q,L,\cL)}{\otimes} E^Q_{y,\sigma,r,\rho^Q}, 
M_{y\sigma,r,\rho} \big) \cong \\
\Hom_{\mh H (G,L,\cL)} \big( \mh H (G,L,\cL) \!\! \underset{\mh H (Q,L,\cL)}{\otimes}
\!\! M^Q_{y,\sigma,r,\rho^Q}, M_{y\sigma,r,\rho} \big) 
\cong \Hom_{\mh H (Q,L,\cL)} (M^Q_{y,\sigma,r,\rho^Q}, M_{y,\sigma,r,\rho}) .
\end{multline*}
For $r=0$ we can apply the functor $\Hom_{\mh H (G,L,\cL)} (?,M_{y,\sigma,0,\rho})$ 
to part (b). A computation analogous to the above yields the desired result.
\end{proof}

Depending on the circumstances, it might be useful to present the parameters
from Theorem \ref{thm:2.16} and Lemma \ref{lem:2.19} in another way. 
If one is primarily interested in
the algebra $\mh H (G,L,\cL) = \mh H (\mf t, W_\cL, c\mb r, \natural_\cL)$, then 
it is natural to involve the Lie algebra $\mf t$. On the other hand, for studying
the parameter space some simplification can be achieved by combining $y$ and 
$\sigma$ in a single element of $\mf g$. Of course that is done with the Jordan 
decomposition $x = x_S + x_N$, where $x_S$ (respectively $x_N$) denotes the 
semisimple (respectively nilpotent) part of $x \in \mf g$ \cite[Theorem 4.4.20]{Spr}.

\begin{cor}\label{cor:2.18}
In the setting of Lemma \ref{lem:2.19}, there exists a canonical bijection 
between the following sets:
\begin{itemize}
\item $\Irr_{r} (\mh H (G,L,\cL))$;
\item $N_G (L)/L$-orbits of triples $(\sigma_0,\cC,\cF)$ where $\sigma_0 \in 
\mf t$, $\cC$ is a nilpotent $Z_G (\sigma_0)$-orbit in $Z_{\mf g}(\sigma_0)$ 
and $\cF$ is an irreducible $Z_G (\sigma_0)$-equivariant local system on $\cC$
such that $\Psi_{Z_G (\sigma_0)}(\cC,\cF) = (L,\cC_v^L,\cL)$ (up to
$Z_G (\sigma_0)$-conjugacy);
\item $G$-orbits of pairs $(x,\rho)$ with $x \in \mf g$ and $\rho \in
\Irr \big( \pi_0 (Z_G (x)) \big)$ such that \\
$\Psi_{Z_G (x_S)}(x_N,\rho) = (L,\cC_v^L,\cL)$ (up to $G$-conjugacy).
\end{itemize}
\end{cor}
\begin{proof}
By Proposition \ref{prop:2.3}.c we may assume that $\sigma$ and $\sigma_0$ lie 
in $\mf t$. Upon requiring that, the $G$-orbit of $\sigma$ (or 
$\sigma_0$) reduces to a $N_G (L)/L$-orbit in $\mf t$. The nilpotent element $y$
lies in $Z_{\mf g}(\sigma_0)$, and only its $Z_G (\sigma_0)$-orbit matters.
The data of $\rho \in \Irr \big( \pi_0 (Z_G (\sigma_0,y)) \big)$ are equivalent
to that of an irreducible $Z_G (\sigma_0)$-equivariant local system $\cF$ on 
$\cC_y^{Z_G (\sigma_0)}$. Now Theorem \ref{thm:2.16}.d provides a canonical 
bijection between the first two sets. 

We put $x = \sigma_0 + y \in \mf g$. By the Jordan decomposition every element
of $\mf g$ is of this form, and $Z_G (x) = Z_G (y,\sigma_0)$. Again 
Theorem \ref{thm:2.16}.d yields the desired bijection, between the first and
third sets.
\end{proof}

We note that the third set of Corollary \ref{cor:2.18} is included in the set
of all $G$-orbits of pairs $(x,\rho)$ with $x \in \mf g$ and $\rho \in 
\Irr \big( \pi_0 (Z_G (x)) \big)$. It follows that the latter set is 
canonically in bijection with
\begin{equation}\label{eq:2.42}
\bigsqcup\nolimits_{(L,\cC_v^L,\cL)} \Irr_r (\mh H (G,L,\cL)) , 
\end{equation}
where the disjoint union runs over all cuspidal supports for $G$ 
(up to $G$-conjugacy).

\subsection{Tempered representations and the discrete series} \
\label{par:temp}

In this paragraph we study two analytic properties of $\mh H (G,L,\cL)$-modules,
temperedness and discrete series. Of course these are well-known for
representations of reductive groups over local fields, and the definition
in our context is designed to mimick those notions.

The complex vector space $\mf t = X_* (T) \otimes_\Z \C$ has a canonical real
form $\mf t_\R = X_* (T) \otimes_\Z \R$. The decomposition of an element
$x \in \mf t$ along $\mf t = \mf t_\R \oplus i \mf t_\R$ will be written as
$x = \Re (x) + i \Im (x)$.
We define the positive cones
\begin{align*}
& \mf t_\R^+ := \{ x \in \mf t_\R : \inp{x}{\alpha} \geq 0 \; 
\forall \alpha \in R(P,T) \} , \\
& \mf t_\R^{*+} := \{ \lambda \in \mf t_\R^* : \inp{\alpha^\vee}{\lambda} 
\geq 0 \; \forall \alpha \in R(P,T) \} .
\end{align*}
The antidual of $\mf t_\R^{*+}$ is the obtuse negative cone 
\[
\mf t_\R^- := \{ x \in \mf t_\R : \inp{x}{\lambda} \leq 0 \; \forall
\lambda \in \mf t_\R^{*+} \} .
\]
It can also be described as 
\[
\mf t_\R^- = \big\{ \sum\nolimits_{\alpha \in R(P,T)} x_\alpha \alpha^\vee : 
x_\alpha \leq 0 \big\} . 
\]
The interior $\mf t_\R^{--}$ of $\mf t_\R^-$ is given by
\begin{equation}\label{eq:2.46}
\mf t_\R^{--} = \left\{
\begin{array}{ll}
\{ \sum_{\alpha \in R(P,T)} x_\alpha \alpha^\vee : x_\alpha < 0 \} &
\text{if } R(G,T)^\vee \text{ spans } \mf t_\R \\
\emptyset & \text{otherwise.}
\end{array}
\right.
\end{equation}
Let $(\pi,V)$ be a $\mh H (\mf t, W_\cL,c \mb r, \natural)$-module. We call 
$x \in \mf t$ a weight of $V$ if there is a $v \in V 
\setminus \{0\}$ such that $\pi (\xi) v = \xi (x) v$ for all $\xi \in 
S (\mf t^*)$. This is equivalent to requiring that the generalized weight space
\begin{equation}\label{eq:2.50}
V_x := \{ v \in V : (\pi (\xi) - \xi (x))^n v = 0 \text{ for some } n \in \N \}
\end{equation}
is nonzero. Since $S(\mf t^*)$ is commutative, $V$ is the direct sum of its 
generalized weight spaces whenever it has finite dimension. We denote the set of 
weights of $(\pi,V)$ by Wt$(\pi,V)$, Wt$(V)$ or Wt$(\pi)$.

\begin{defn}\label{def:2.temp}
Let $(\pi,V)$ be a finite dimensional $\mh H (\mf t, W_\cL,c \mb r, \natural)
$-module. We call it tempered if $\Re (\mathrm{Wt}(\pi,V)) \subset \mf t_\R^-$.
We call it discrete series if $\Re (\mathrm{Wt}(\pi,V)) \subset \mf t_\R^{--}$.

Similarly we say that $(\pi,V)$ is anti-tempered (respectively
anti-discrete series) if $\Re (\mathrm{Wt}(\pi,V)) \subset - \mf t_\R^-$ 
(respectively $\subset - \mf t_\R^{--}$).

We denote the set of irreducible tempered representations of this algebra by
$\Irr_\temp (\mh H (\mf t, W_\cL,c \mb r, \natural))$.
\end{defn}

\begin{thm}\label{thm:2.19}
Let $y,\sigma,\rho$ be as in Corollary \ref{cor:2.18}, 
with $\sigma, \sigma_0 \in \mf t$.
\enuma{
\item Suppose that $\Re (r) \leq 0$. The following are equivalent:
\begin{enumerate}
\item $E_{y,\sigma,r,\rho}$ is tempered;
\item $M_{y,\sigma,r,\rho}$ is tempered;
\item $\sigma_0 \in i \mf t_\R$.
\end{enumerate}
\item Suppose that $\Re (r) \geq 0$. Then part (a) remains valid if we
replace tempered by anti-tempered.
} 
\end{thm}
\begin{proof}
(a) Choose $\tau$ and $\rho^\circ \in \Irr \big( \pi_0 (Z_G (\sigma,y)) \big)$ 
as before, so that $\rho = \rho^\circ \rtimes \tau^*$. By Clifford theory and 
Lemmas \ref{lem:2.14} and \ref{lem:2.15}
\begin{equation}\label{eq:2.45}
\mathrm{Wt}(M_{y,\sigma,r,\rho}) = 
\cR_\cL \mathrm{Wt}(M^\circ_{y,\sigma,r,\rho^\circ}) ,
\end{equation}
and similarly for $E^\circ_{y,\sigma,r,\rho^\circ}$. Since $\cR_\cL$ stabilizes
$\mf t_\R^-$, it follows that $E_{y,\sigma,r,\rho}$ (respectively 
$M_{y,\sigma,r,\rho}$) is tempered if and only if $E^\circ_{y,\sigma,r,\rho^\circ}$ 
(respectively $M^\circ_{y,\sigma,r,\rho^\circ}$ is tempered. This reduces the 
claim to the case where $G$ is connected. 

From now on we assume that $\cR_\cL = 1$. From Proposition \ref{prop:1.4} we see that
\[
\mh H (G^\circ,L,\cL) = \mh H (G^\circ_\der ,L \cap G_\der,\cL) \otimes S (Z(\mf g)^*) . 
\]
Write $\sigma_0 = \sigma_{0,\der} + z_0$ with $\sigma_{0,\der} \in \mathrm{Lie} 
(G_\der)$ and $z_0 \in Z(\mf g)$. By Proposition \ref{prop:2.3}.b both  
$M^\circ_{y,\sigma,r,\rho^\circ}$ and $E^\circ_{y,\sigma,r,\rho^\circ}$ admit
the $S(Z(\mf g)^*)$-character $z_0$. By definition $\mf t_\R^- \cap Z(\mf g) = \{0\}$.
Thus $M^\circ_{y,\sigma,r,\rho^\circ}$ and $E^\circ_{y,\sigma,r,\rho^\circ}$ are
tempered as $Z(\mf Z (g)^*)$-modules if and only if $\Re (z_0) = 0$, or
equivalently $z_0 \in X_* (Z(G^\circ)^\circ) \otimes_\Z i \R$. This achieves 
further reduction, to the case where $G = G^\circ$ is semisimple.

When $\Re (r) < 0$, we will apply \cite[Theorem 1.21]{Lus7}. It says that the 
following are equivalent:
\begin{itemize}
\item[(i)'] $E^\circ_{y,\sigma,r,\rho^\circ}$ is $\tau$-tempered
(where $\tau$ refers to the homomorphism $\Re : \C \to \R$);
\item[(ii)'] $M^\circ_{y,\sigma,r,\rho^\circ}$ is $\tau$-tempered;
\item[(iii)'] All the eigenvalues of $\ad (\sigma_0) : \mf g \to \mf g$ 
are purely imaginary.
\end{itemize}
As $\mf t$-module, $\mf g$ is the direct sum of the weight spaces $\mf g_\alpha$ with 
$\alpha \in R(G,T) \cup \{0\}$. We note that $R(G,T) \cup \{0\} \subset X^* (T) \subset 
\Hom (\mf t,\R)$ and $R(G,T)$ spans $\mf t_\R^*$ (for $G$ is semisimple). Hence
(iii)' is equivalent to (iii). 

As $\Re (r) < 0$, the condition for $\tau$-temperedness of a module $E$ 
\cite[1.20]{Lus7} becomes
\begin{equation}\label{eq:2.44}
\Re (\lambda) \leq 0 \text{ for any eigenvalue of } \xi_V \text{ on } E. 
\end{equation}
Here $\xi_V \in \mf t^*$ is determined by an irreducible finite dimensional 
$\mf g$-module $V$ which contains a unique line $\C v$ annihilated by $\mf u$. Then
$\xi_V$ is the character by which $\mf t$ acts on $\C v$. When we vary $V$, $\xi_V$ 
runs through a set of dominant weights which spans $\mf t_\R^{*+}$ over $\R_{\geq 0}$. 
Hence the condition \eqref{eq:2.44} is equivalent to
$\Re (\mathrm{Wt}(E)) \subset \mf t_\R^-$. In other words, temperedness is the
same as $\tau$-temperedness when $\Re (r) < 0$, (i) is the same as (i)' and
(ii) is equivalent to (ii)'. Thus \cite[Theorem 1.21]{Lus7} is our required
result in this setting.

It remains to settle the case $\Re (r) = 0, G = G^\circ$ semisimple. 
Assume (iii). When we vary $r$ and keep $\sigma_0$ fixed, the weights of 
$(E^\circ_{y,\sigma,r,\rho^\circ})$ depend algebraically on $r$. We have already 
shown that $\Re (\mathrm{Wt}(E^\circ_{y,\sigma,r,\rho^\circ})) \subset \mf t_\R^-$ 
when $\Re (r) < 0$. Clearly $\mf t_\R^-$ is closed in $\mf t_\R$, so by continuity
this property remains valid when $\Re (r) = 0$. That proves (i), upon which
(ii) follows immediately. 

Conversely, suppose that (iii) does not hold. We may assume that $\gamma_y :
\SL_2 (\C) \to G$ has image in $Z_G (\sigma_0)$. Recall from Proposition
\ref{prop:2.3}.b that
\[
\mathrm{Wt}( M^\circ_{y,\sigma,r,\rho^\circ} ) \subset
W_\cL \big( \sigma_0 + \textup{d}\gamma_y \matje{r}{0}{0}{-r} \big) .
\]
In particular we can find a $w \in W_\cL$ such that $w \big( \sigma_0 + 
\textup{d}\gamma_y \matje{r}{0}{0}{-r} \big)$ is a $S(\mf t^*)$-weight of
$M^\circ_{y,\sigma,r,\rho^\circ}$. The map $z \mapsto \gamma_y 
\matje{z}{0}{0}{z^{-1}}$ is a cocharacter of $T$ and $r \in i \R$, so 
$\textup{d}\gamma_y \matje{r}{0}{0}{-r} \in i \mf t_\R$. By our assumption
$\sigma_0 + \textup{d}\gamma_y \matje{r}{0}{0}{-r} \in \mf t \setminus i \mf t_\R$,
and this does not change upon applying $w \in W_\cL$. Hence 
$M^\circ_{y,\sigma,r,\rho^\circ}$ is not tempered. We proved that
(ii) implies (iii) when $\Re (r) = 0$.\\
(b) This is completely analogous to part (a), when we interpret $\tau$-tempered
with $\tau = -\Re : \C \to \R$.
\end{proof}

From Definition \ref{def:2.temp} and \eqref{eq:2.46} we immediately see that
$\mh H (G,L,\cL)$ has no discrete series representations if $R(G,T)$ does not
span $\mf t_\R^*$. That is equivalent to $Z(\mf g) \neq 0$. Therefore we only 
formulate a criterium for discrete series when $G^\circ$ is semisimple.

\begin{thm}\label{thm:2.20}
Let $G^\circ$ be semisimple. Let $y,\sigma,\rho$ be as in Corollary \ref{cor:2.18}, 
with $\sigma, \sigma_0 \in \mf t$.
\enuma{
\item Suppose that $\Re (r) < 0$. The following are equivalent:
\begin{enumerate}
\item $M_{y,\sigma,r,\rho}$ is discrete series;
\item $y$ is distinguished in $\mf g$, that is, it is not contained in any
proper Levi subalgebra of $\mf g$.
\end{enumerate}
Moreover, if these conditions are fulfilled, then $\sigma_0 = 0$ and
$E_{y,\sigma,r,\rho} = M_{y,\sigma,r,\rho}$.
\item Suppose that $\Re (r) > 0$. Then part (a) remains valid if we replace
(i) by: $M_{y,\sigma,r,\rho}$ is anti-discrete series.
\item For $\Re (r) = 0$ there are no (anti-)discrete series representations
on which $\mb r$ acts as $r$.
} 
\end{thm}
\begin{proof}
(a) Since $[\sigma_0,y] = 0$ and $\mf g$ is semisimple, $\sigma_0 = 0$ whenever
$y$ is distinguished. 

In view of \eqref{eq:2.45} it suffices to prove the equivalence of (i) and
(ii) when $G$ is connected, so we assume that for the moment. 
We can reformulate (ii) as:
\[
\inp{x}{\alpha} < 0 \text{ for all } x \in \mathrm{Wt}(M^\circ_{y,\sigma,r,
\rho^\circ}) \text{ and all } \alpha \in R(P,T) .
\]
The same argument as for temperedness shows that this is equivalent to 
$M^\circ_{y,\sigma,r,\rho^\circ}$ being $\tau$-square integrable with $\tau = 
\Re$, in the sense of \cite{Lus7}. By \cite[Theorem 1.22]{Lus7} that in turn is
equivalent to (i). The same result also that $E^\circ_{y,\sigma,r,\rho^\circ} = 
M^\circ_{y,\sigma,r,\rho^\circ}$ when (i) and (ii) hold.

The last statement can be lifted from $G^\circ$ to $G$ by
\eqref{eq:2.43} and Lemma \ref{lem:2.15}:
\[
E_{y,\sigma,r,\rho} \cong \tau \ltimes E_{y,\sigma,r,\rho^\circ} =
\tau \ltimes M^\circ_{y,\sigma,r,\rho^\circ} \cong M_{y,\sigma,r,\rho} .
\]
(b) This can be shown in the same way as part (a), when we consider $\tau$-square
integrable with $\tau = -\Re : \C \to \R$.\\
(c) Suppose that $V$ is a discrete series $\mh H (G,L,\cL)$-module on which
$\mb r$ acts as $r \in i \R$. By definition $\dim V < \infty$, so $V$ has an
irreducible subrepresentation, say $M_{y,\sigma,r,\rho}$. Its weights are a
subset of those of $V$, so it is also discrete series. By \eqref{eq:2.45} this
means that $M^\circ_{y,\sigma,r,\rho^\circ} \in \Irr (\mh H (G^\circ,L,\cL))$
is discrete series. 

In particular it is tempered, so by Theorem \ref{thm:2.19} $\sigma_0 \in i 
\mf t_\R$. As $\gamma_y : \SL_2 (\C) \to G^\circ$ is algebraic and $r \in i \R$, 
d$\gamma_y \matje{r}{0}{0}{-r} \in i \mf t_\R$ as well. From Proposition
\ref{prop:2.3}.b we know that 
\[
\mathrm{Wt}(M^\circ_{y,\sigma,r,\rho^\circ}) \subset W_\cL^\circ \sigma =
W_\cL^\circ \big( \sigma_0 + \textup{d}\gamma_y \matje{r}{0}{0}{-r} \big)
\subset i \mf t_\R .
\]
Consequently $\Re (x) = 0 \notin \mf t_\R^{--}$ for every 
$x \in \mathrm{Wt}(M^\circ_{y,\sigma,r,\rho^\circ})$. This contradicts the
definition of discrete series.
\end{proof}

When $R(G,T)$ does not span $\mf t_\R^*$, it is sometimes useful to relax 
the notion of the discrete series in the following way. 
\begin{defn}\label{defn:2.23}
Let $(\pi,V)$ be a finite dimensional $\mh H (\mf t, W_\cL,c \mb r, \natural)
$-module, and let $\mf t' \subset \mf t$ be the $\C$-span of the coroots for 
$W_\cL^\circ$. We say that $(\pi,V)$ is essentially (anti-) discrete series if 
its restriction to $\mh H (\mf t', W_\cL,c\mb r)$ is (anti-) discrete series.
\end{defn}

\begin{cor}\label{cor:2.21}
Let $r \in \C$ with $\Re (r) < 0$, and let $y,\sigma,\rho$ be as in Corollary 
\ref{cor:2.18}, with $\sigma, \sigma_0 \in \mf t$. Then $M_{y,\sigma,r,\rho}$ 
is essentially discrete series if and only if $y$ is distinguished in $\mf g$.

When $\Re (r) > 0$, the same holds with essentially anti-discrete series. 
\end{cor}
\begin{proof}
Fix $r \in \C$ with $\Re (r) < 0$. Notice that for algebra under consideration 
$\mf t'$ as in Definition \eqref{defn:2.23} equals $\mf t \cap \mf g_\der$.
Namely, the roots of the semisimple Lie algebra $\mf g_\der$ span the linear dual
of any Cartan subalgebra, and hence also of $\mf t$.

Recall from \eqref{eq:1.11} that  
\[
\mh H (G^\circ,L,\cL) = 
\mh H (\mf t \cap \mf g_\der ,W_\cL^\circ,c \mb r ) \otimes S (Z(\mf g)^*) . 
\]
The restriction of $M_{y,\sigma,r,\rho} = M_{y,\sigma,r,\rho^\circ \rtimes \tau^*}$
to $\mh H (G^\circ,L,\cL)$ is
\[
V_\tau \otimes M_{y,\sigma - z_0,r,\rho^\circ} \otimes \C_{z_0}, \text{ where }
\sigma = (\sigma - z_0) + z_0 \in (\mf t \cap \mf g_\der) \oplus Z(\mf g) .
\]
The action on $V_\tau$ is trivial and there is no condition on the character
$z_0$ by which $S(Z(\mf g)^*)$ acts. Hence $M_{y,\sigma,r,\rho}$ is essentially
discrete series if and only if
\[
M_{y,\sigma - z_0,r,\rho^\circ} \in \Irr (\mh H (\mf t \cap \mf g_\der ,W_\cL^\circ,c \mb r ))
\]
is discrete series. By Theorem \ref{thm:2.20} that is equivalent to $y$ being 
distinguished in $\mf g_\der$. Since $\mf g  = \mf g_\der \oplus Z(\mf g)$, that 
is the case if and only if $y$ is distinguished in $\mf g$.

The case $\Re (r) > 0$ can be shown in the same way.
\end{proof}

Unfortunately Theorems \ref{thm:2.19}, \ref{thm:2.20} and Corollary \ref{cor:2.21}
do not work as we would like them when $\Re (r) > 0$, the prefix "anti" should
rather not be there. In the Langlands program $r$ will typically be $\log (q)$,
where $q$ is cardinality of a finite field, so $r \in \R_{>0}$ is the default.
This problem comes from \cite{Lus7} and can be traced back to Lusztig's conventions
for the generalized Springer correspondence in \cite{Lus2}, see also Remark 
\ref{rem:1.L}.

To make the properties of $\mh H (G,L,\cL)$-modules fit with those of Langlands
parameters, we need a small adjustment. Extend the sign representation of the
Weyl group $W_\cL^\circ$ to a character of $W_\cL = W_\cL^\circ \rtimes \cR_\cL$
by means of the trivial representation of $\cR_\cL$. Then $N_w \mapsto 
\mathrm{sign} (w) N_w$ extends linearly to an involution of $\C [W_\cL,\natural_\cL]$.

The Iwahori--Matsumoto involution of $\mh H (G,L,\cL)$ is defined as the unique
algebra automorphism such that
\begin{equation}\label{eq:2.53}
\IM (N_w) = \mathrm{sign} (w) N_w ,\qquad \IM (\mb r) = \mb r, 
\qquad \IM ( \xi ) = - \xi \quad (\xi \in \mf t^*) .
\end{equation}
Notice that IM preserves the braid relation 
\[
N_{s_i} \xi - {}^{s_i}\xi N_{s_i} = c_i {\mb r} (\xi - {}^{s_i} \xi) / \alpha_i ,
\]
for $\alpha_i$ is also multiplied by -1.
We also note that the Iwahori--Matsumoto involutions for various graded Hecke
algebras are compatible with parabolic induction. Suppose that $Q \subset G$ is
as in Proposition \ref{prop:2.25} and let $V$ be any $\mh H (Q,L,\cL)$-module.
There is a canonical isomorphism of $\mh H (G,L,\cL)$-modules
\begin{equation}\label{eq:2.62}
\begin{array}{ccc} 
\mh H (G,L,\cL) \underset{\mh H (Q,L,\cL)}{\otimes} \IM^* ( V ) & \to &
\IM^* \big( \mh H (G,L,\cL) \underset{\mh H (Q,L,\cL)}{\otimes} V \big) \\
h \otimes v & \mapsto & \IM (h) \otimes v 
\end{array}.
\end{equation}
This allows us to identify the two modules, and then Proposition \ref{prop:2.25}
remains valid upon composition with IM.

Clearly IM has the effect $x \leftrightarrow -x$ on $S(\mf t^*)$-weights of
$\mh H (G,L,\cL)$-representations. Hence IM exchanges tempered with anti-tempered
representations, and discrete series with anti-discrete series representations.
For $\Re (r) \geq 0$ Theorem \ref{thm:2.19} yields equivalences
\begin{equation}\label{eq:2.47}
\IM^* E_{y,\sigma,r,\rho} \text{ is tempered } \Longleftrightarrow \;
\IM^* M_{y,\sigma,r,\rho} \text{ is tempered } \Longleftrightarrow \;
\sigma_0 \in i \mf t_\R .
\end{equation}
For $\Re (r) > 0$ Corollary \ref{cor:2.21} says that
\begin{equation}\label{eq:2.48}
\IM^* M_{y,\sigma,r,\rho} \text{ is essentially discrete series }
\Longleftrightarrow \; y \text{ is distinguished in } \mf g .
\end{equation}
We note that $\IM^*$ changes central characters of these representations: 
by Proposition \ref{prop:2.3}.b both $\IM^* E_{y,\sigma,r,\rho}$ and 
$\IM^* M_{y,\sigma,r,\rho}$ 
\begin{equation}\label{eq:2.49}
\text{admit the central character } (-W_\cL \sigma ,r) \in \mf t / W_\cL \times \C.
\end{equation}
Composition with the Iwahori--Matsumoto involution corresponds to two changes
in the previous setup:
\begin{itemize}
\item In \eqref{eq:1.1} the action of $\C [W_\cL,\natural_\cL^{-1}]$
on $K^*$ is twisted by the sign character of $W_\cL$, that is, we use a normalization
different from that of Lusztig in \cite{Lus2}.
\item The action \eqref{eq:2.9} of $\mf t^* \subset H^*_{G \times \C^\times}(\dot{\mf g})$ 
on standard modules is adjusted by a factor -1.
\end{itemize}
To $r \in \C$ and a triple $(y,\sigma_0,\rho)$ as in Theorem \ref{thm:2.16}.c we will
associate the irreducible representation $\IM^* 
M_{y,\textup{d}\gamma_y \matje{r}{0}{0}{-r} - \sigma_0, r ,\rho}$. This parametrization
of $\Irr_r (\mh H (G,L,\cL))$ is in some respects more suitable than that in Theorem 
\ref{thm:2.16}, for example to study tempered representations. 

We use it here to highlight the relation with extended quotients. Recall that
$W_\cL$ acts linearly on $\mf t$ and that $\C [W_\cL,\natural_\cL] \subset
\mh H (G,L,\cL)$. We write
\[
\tilde{\mf t}_{\natural_\cL} = \{ (x,\pi_x) : x \in \mf t, \pi_x \in
\Irr (\C [(W_\cL)_x,\natural_\cL] ) \} .
\]
The group $W_\cL$ acts on $\tilde{\mf t}_{\natural_\cL}$ by
\[
w \cdot (x,\pi_x) = (w x, w^* \pi_x) \text{ where }
(w^* \pi_x) (N_v) = \pi_x (N_w^{-1} N_v N_w) \text{ for } v \in (W_\cL)_{wx} .
\]
The twisted extended quotient of $\mf t$ by $W_\cL$ (with respect to $\natural_\cL$)
is defined as
\begin{equation}\label{eq:2.51}
(\mf t \q W_\cL )_{\natural_\cL} = \tilde{\mf t}_{\natural_\cL} / W_\cL .
\end{equation}

\begin{thm}\label{thm:2.22}
Let $r \in \C$. There exists a canonical bijection
\[
\mu_{G,L\cL} : (\mf t \q W_\cL )_{\natural_\cL} \to \Irr_r (\mh H (G,L,\cL))
\]
such that:
\begin{itemize}
\item $\mu_{G,L,\cL}( i \mf t_\R \q W_\cL )_{\natural_\cL} = 
\Irr_{r,\temp} (\mh H (G,L,\cL))$ when $\Re (r) \geq 0$.\\
For $\Re (r) \leq 0$ it is the anti-tempered part of $\Irr_r (\mh H (G,L,\cL))$.
\item The central character of $\mu_{G,L,\cL}(x,\pi_x)$ is $\big( W_\cL \big( x + 
\textup{d}\gamma \matje{r}{0}{0}{-r} \big), r \big)$ for some algebraic homomorphism
$\gamma : \SL_2 (\C) \to Z_G (x)^\circ$.
\end{itemize}
\end{thm}
\textbf{Remark.} This establishes a version of the ABPS-conjectures \cite[\S 15]{ABPS2}
for the twisted graded Hecke algebra $\mh H (G,L,\cL)$.
\begin{proof}
By \cite[Lemma 2.3]{ABPS7} there exists a canonical bijection
\[
\begin{array}{ccc}
(\mf t \q W_\cL )_{\natural_\cL} & \to & \Irr (S(\mf t^*) \rtimes \C [W_\cL,\natural_\cL]) \\
(x,\pi_x) & \mapsto & \C_x \rtimes \pi_x = \ind_{S(\mf t^*) \rtimes \C [(W_\cL)_x,
\natural_\cL]}^{S(\mf t^*) \rtimes \C [W_\cL,\natural_\cL]} (\C_x \otimes V_{\pi_x}) 
\end{array} .
\]
We can consider $\C_x \rtimes \pi_x$ as an irreducible $\mh H (G,L,\cL)$-representation
with central character $(W_\cL x, 0)$. By Lemma \ref{lem:2.19} there are $y,\rho$,
unique up to $Z_G (x)$-conjugation, such that 
$\C_x \rtimes \pi_x \cong \IM^* M_{y,-x,0,\rho}$. Choose an algebraic homomorphism
$\gamma_y : \SL_2 (\C) \to Z_G (x)^\circ$ with d$\gamma_y \matje{0}{1}{0}{0} = y$.
Now we can define
\[
\mu_{G,L,\cL}(x,\pi_x) = 
\IM^* M_{y,\textup{d}\gamma_y \matje{r}{0}{0}{-r} - x, r,\rho} .
\]
This is canonical because all the above choices are unique up to conjugation. By 
\eqref{eq:2.49} its central character is $\big( W_\cL \big( x -
\textup{d}\gamma_y \matje{r}{0}{0}{-r} \big), r \big)$. Define $\gamma : 
\SL_2 (\C) \to Z_G (x)^\circ$ by
\[
\gamma (g) = \gamma_y \matje{0}{1}{-1}{0} \gamma_y (g) \gamma_y \matje{0}{-1}{1}{0}, 
\]
it is associated to the unipotent element $\gamma_y \matje{1}{0}{1}{1}$. 
As $\textup{d}\gamma \matje{r}{0}{0}{-r} = - \textup{d}\gamma_y \matje{r}{0}{0}{-r}$, 
the central character of $\mu_{G,L,\cL}$ attains the desired form.

The claims about temperedness follow from Theorem \ref{thm:2.19} and \eqref{eq:2.47}.
\end{proof}

\section{The twisted graded Hecke algebra of a cuspidal quasi-support}
\label{sec:3}

In disconnected reductive groups one sometimes has to deal with disconnected
variations on Levi subgroups. Here we will generalize the results of the
previous two sections to that setting. 

Recall \cite{AMS} that a quasi-Levi subgroup of $G$ is a group of the form
$M = Z_G (Z(L)^\circ)$, where $L$ is a Levi subgroup of $G^\circ$. Thus
$Z (M)^\circ = Z(L)^\circ$ and $M \longleftrightarrow L = M^\circ$ is a bijection
between quasi-Levi subgroups of $G$ and the Levi subgroups of $G^\circ$.

\begin{defn}\label{defn:3.5}
A cuspidal quasi-support for $G$ is a triple $(M,\cC_v^M,q\cL)$ where:
\begin{itemize}
\item $M$ is a quasi-Levi subgroup of $G$;
\item $\cC_v^M$ is a nilpotent $\Ad(M)$-orbit in $\mf m = \mathrm{Lie}(M)$.
\item $q \cL$ is a $M$-equivariant cuspidal local system on $\cC_v^M$, i.e. as 
$M^\circ$-equivariant local system it is a direct sum of cuspidal local systems.
\end{itemize}
We denote the $G$-conjugacy class of $(M,v,q \cL)$ by $[M,v,\cL]_G$.
With this cuspidal quasi-support we associate the groups
\begin{equation}\label{eq:3.15}
N_G (q \cL) = \mathrm{Stab}_{N_G (M)}(q \cL) \quad \text{and} \quad
W_{q \cL} = N_G (q \cL) / M .
\end{equation}
\end{defn}
Let $\cL$ be an irreducible constituent of $q\cL$ as $M^\circ$-equivariant
local system on $\cC_v^{M^\circ} = \cC_v^M$. Then
\[
W_\cL^\circ = N_{G^\circ}(M^\circ) / M^\circ \cong N_{G^\circ} (M^\circ) M / M
\]
is a a subgroup of $W_{q\cL}$. It is normal because $G^\circ$ is normal in $G$.

Let $P^\circ$ be a parabolic subgroup of $G^\circ$ with Levi decomposition
$P^\circ = M^\circ \ltimes U$. The definition of $M$ entails that it normalizes
$U$, so
\[
P := M \ltimes U
\]
is a again a group. We put
\begin{align*}
& N_G (P,q\cL) = N_G (P,M) \cap N_G (q\cL) , \\
& \cR_{q\cL} = N_G (P,q\cL) / M .
\end{align*}
The same proof as for Lemma \ref{lem:1.1}.b shows that 
\begin{equation}\label{eq:3.1}
W_{q\cL} = W_\cL^\circ \rtimes \cR_{q\cL} .
\end{equation}
We define $\dot{\mf g}$ as before, but with $M$ instead of $L$, 
and with the new $P$. We put
\begin{equation}\label{eq:3.14}
K = (\mathrm{pr}_1)_! \dot{q \cL} \quad \text{and} \quad
K^* = (\mathrm{pr}_1)_! \dot{q \cL^*}, 
\end{equation}
these are perverse sheaves on $\mf g$. Considering $(\mathrm{pr}_1)_! \dot{q \cL}$ 
as a perverse sheaf on $\mf g_{RS}$, \cite[Lemma 5.4]{AMS} says that
\[
\End_{\mathcal{P}_{G}(\mf g_{RS})}((\mathrm{pr}_1)_! \dot{q \cL} ) 
\cong \C[W_{q \cL},\natural_{q \cL}],
\]
where $\natural_{q\cL} : (W_{q\cL} / W_\cL^\circ )^2 \to \C^\times$ is a suitable 
2-cocycle. As in \eqref{eq:1.E}
\begin{equation}\label{eq:3.60}
\End^+_{\mathcal{P}_{G}(\mf g_{RS})}((\mathrm{pr}_1)_! \dot{q \cL} ) 
\cong \C[\cR_{q \cL},\natural_{q \cL}] .
\end{equation}
To $(M,\cC_v^M,q\cL)$ we associate the twisted graded Hecke algebra
\[
\mh H (G,M,q\cL) := \mh H (\mf t, W_{q\cL},c \mb r,\natural_{q \cL}) ,
\]
where the parameters $c_i$ are as in \eqref{eq:1.5}. As in Lemma \ref{lem:1.7}, we
can consider it as 
\[
\mh H (G,M,q\cL) = \mh H (\mf t, W_\cL^\circ,c \mb r) \rtimes 
\End_{\mathcal{P}_{G}(\mf g_{RS})}((\mathrm{pr}_1)_! \dot{q \cL} ) ,
\]
and then it depends canonically on $(G,M,q\cL)$.
We note that \eqref{eq:3.1} implies
\begin{equation}\label{eq:3.13}
\mh H (G^\circ N_G (P,q\cL),M,q\cL) = \mh H (G,M,q\cL) .
\end{equation}
All the material from Proposition \ref{prop:1.2} up to and including
Theorem \ref{thm:2.4}, and the parts of \cite{Lus3} on which it is based, 
extend to this situation with the above substitutions. We will use these results 
also for $\mh H (G,M,q\cL)$.

To generalize the remainder of Section \ref{sec:reps} we need to assume that:
\begin{cond}\label{cond:3}
The group $G$ equals $G^\circ N_G (P,q\cL)$. 
\end{cond}
By \eqref{eq:3.13} this imposes no further restriction on the collection of 
twisted graded Hecke algebras under consideration. Let us write
\[
\cP_y^\circ = \{ gP \in G^\circ M / P : \Ad (g^{-1}) y \in \cC_v^M + \mf u \} 
= \cP_y \cap G^\circ M / P .
\]
Condition \ref{cond:3} guarantees that $\cP_y = \cP_y^\circ \times \cR_{q\cL}$ 
as $M(y)$-varieties. With these minor modifications Lemma \ref{lem:2.1} also 
goes through: there is an isomorphism of $\mh H (G,M,q\cL)$-modules
\begin{equation}\label{eq:3.2}
H_*^{M(y)^\circ}(\cP_y,\dot{q\cL}) \cong \ind_{\mh H (G^\circ M,M,q\cL)}^{\mh H 
(G,M,q\cL)} H_*^{M(y)^\circ}(\cP_y^\circ,\dot{q\cL}) . 
\end{equation}
We note that $N_G (q\cL) \cap G^\circ = N_{G^\circ}(M^\circ)$, for by 
\cite[Theorem 9.2]{Lus2} $N_G (M^\circ)$ stabilizes all $M^\circ$-equivariant
cuspidal local systems contained in $q\cL$. Hence
\begin{equation}\label{eq:3.3}
N_{G^\circ M}(q\cL)/M \cong N_{G^\circ}(q\cL)/M^\circ = 
N_{G^\circ}(M^\circ) / M^\circ = W_\cL^\circ .
\end{equation}
Moreover the 2-cocycles $\natural_{q\cL}$ and $\natural_\cL$ are trivial on 
$W_\cL^\circ$, so we can 
\begin{equation}\label{eq:3.11}
\text{identify} \quad \mh H (G^\circ M,M,q\cL) 
\quad \text{with} \quad \mh H (G^\circ,L,\cL) .
\end{equation}
We already performed the construction and parametrization of $\mh H (G^\circ,L,\cL)$ 
in Theorem \ref{thm:2.9}, but now we want it in terms of $M$ and $q\cL$. To this end 
we need to recall how $q\cL$ can be constructed from $\cL$. Let $M_\cL$ be the 
stabilizer in $M$ of $(\cC_v^{M^\circ},\cL)$. Let $K_M$ be like $K$, 
but for $M$. About this perverse sheaf on $\mf m$ \cite[Proposition 4.5]{AMS} says
\[
\End_{\mathcal{P}_{G}(\mf m_{RS})} (K_M) \cong \C [M_\cL / M^\circ,\natural_\cL] .
\]
By \cite[(63)]{AMS} there is a unique 
$\rho_M \in \Irr (\C [M_\cL / M^\circ,\natural_\cL^{-1}] )$ such that 
\begin{equation}\label{eq:3.4}
q \cL = \Hom_{\C [M_\cL / M^\circ,\natural_\cL]} (\rho_M^* ,K_M) . 
\end{equation}
From the proof of \cite[Proposition 3.5]{AMS} we see that the stalk of 
\eqref{eq:3.4} at $v \in \cC_v^M$, considered as $Z_M (v)$-representation, is
\[
(q \cL)_v \cong \ind_{Z_M (v)_{\cL_v}}^{Z_M (v)} (\rho_M \otimes \cL_v) =
\ind_{Z_{M_\cL} (v)}^{Z_M (v)} (\rho_M \otimes \cL_v) .
\]
Here $Z_M (v)_{\cL_v}$ denotes the stabilizer of $\cL_v \in \Irr (Z_{M^\circ}(v))$
in $Z_M (v)$. The same holds for other elements in the $M$-conjugacy class of 
$v$, so as $M$-equivariant sheaves
\begin{equation}\label{eq:3.5}
q \cL \cong \ind_{M_\cL}^M (\rho_M \otimes \cL) .
\end{equation}
We recall from \cite[(64)]{AMS} that the cuspidal support map $\Psi_G$ has a "quasi"
version $q \Psi_G$, which associates to every pair $(y,\rho)$ with $y \in G^\circ$ 
unipotent and $\rho \in \Irr_{\pi_0 (Z_G (y))}$ a cuspidal quasi-support.

\begin{lem}\label{lem:3.1}
Let $y \in \mf g$ be nilpotent such that $\cP_y$ is nonempty. Then $M$ stabilizes
the $M^\circ$-orbit of $y$.
\end{lem}
\begin{proof}
From \cite[Theorem 6.5]{Lus2} and \cite[(64)]{AMS} we deduce that there exists a
$\rho \in \Irr \big( \pi_0 (Z_G (y)) \big)$ such that $q\Psi_G (y,\rho) = 
(M,\cC_v^M,q\cL)$ (up to $G$-conjugacy). Now \cite[Lemma 7.6]{AMS} says that there 
exist algebraic homomorphisms $\gamma_y, \gamma_v : \SL_2 (\C) \to M^\circ$ such that 
\begin{equation}\label{eq:3.7}
\textup{d}\gamma_y \matje{0}{1}{0}{0} = y ,\quad 
\textup{d}\gamma_v \matje{0}{1}{0}{0} = v
\quad \text{and} \quad \textup{d}\gamma_v \matje{1}{0}{0}{-1} - 
\textup{d}\gamma_y \matje{1}{0}{0}{-1} \in \mathrm{Lie}(Z(M)^\circ) .
\end{equation}
In view of \cite[Proposition 5.6.4]{Car} the $G^\circ$-conjugacy class of $y$ (resp. 
the $M^\circ$-orbit of $v$) is completely determined by the $G^\circ$-conjugacy class 
of $\textup{d}\gamma_y \matje{1}{0}{0}{-1}$ (resp. the $M^\circ$-conjugacy class of
$\textup{d}\gamma_v \matje{1}{0}{0}{-1}$). By \cite[Theorem 3.1.a]{AMS} 
$\cC_v^{M^\circ} = \cC_v^M$. It follows that for every $m \in M$ there is a 
$m_0 \in M^\circ$ such that
\[
\Ad (m) \textup{d}\gamma_v \matje{1}{0}{0}{-1} = 
\Ad (m_0) \textup{d}\gamma_v \matje{1}{0}{0}{-1} .
\]
We calculate, using \eqref{eq:3.7}:
\begin{multline*}
\Ad (m) \textup{d}\gamma_y \matje{1}{0}{0}{-1} = \Ad (m) \textup{d}\gamma_v 
\matje{1}{0}{0}{-1} + \Ad (m) \big( \textup{d}\gamma_y \matje{1}{0}{0}{-1} -
\textup{d}\gamma_v \matje{1}{0}{0}{-1} \big) \\
\Ad (m) \textup{d}\gamma_v \matje{1}{0}{0}{-1} + \Ad (m) \big( \textup{d}\gamma_y 
\matje{1}{0}{0}{-1} - \textup{d}\gamma_v \matje{1}{0}{0}{-1} \big) = 
\Ad (m_0) \textup{d}\gamma_y \matje{1}{0}{0}{-1} .
\end{multline*}
This implies that $m$ stabilizes the $M^\circ$-orbit of $y$.
\end{proof}

\begin{lem}\label{lem:3.2}
Let $\sigma_0 \in \mf t = \mathrm{Lie}(Z(M)^\circ)$ be semisimple, write $Q = Z_G 
(\sigma_0)$ and let $y \in Z_{\mf g}(\sigma_0) = \mathrm{Lie}(Q)$ be nilpotent. 
The map $\rho^\circ \mapsto \rho^\circ \ltimes \rho_M$ is a bijection between 
the following sets:
\begin{align*}
& \big\{ \rho^\circ \in \Irr \big( \pi_0 (Z_Q (y)) \big) : \Psi_Q (y,\rho^\circ) = 
(M^\circ,\cC_v^M,\cL) \text{ up to } G^\circ\text{-conjugation} \big\} , \\
& \big\{ \tau^\circ \in \Irr \big( \pi_0 (Z_{QM} (y)) \big) : q\Psi_{QM} 
(y,\tau^\circ) = (M,\cC_v^M,q\cL) \text{ up to } G^\circ M\text{-conjugation} \big\}.
\end{align*}
\end{lem}
\begin{proof}
Notice that $M^\circ \subset Q$, for $\sigma \in \mf t$. By \cite[Theorem 9.2]{Lus2}
there is a canonical bijection
\[
\Sigma_\cL : \Psi_Q^{-1} (M^\circ,\cC_v^M,\cL) \to \Irr (W_\cL^Q) .
\]
Similarly, by \cite[Lemma 5.3 and Theorem 5.5.a]{AMS} there is a canonical bijection
\[
q \Sigma_{q \cL} : q \Psi_{QM}^{-1} (M,\cC_v^M,q\cL) \to \Irr (N_{QM}(M,q\cL ) / M) 
\cong \Irr (W_\cL^Q) ,
\]
where we used \eqref{eq:3.3} for the last identification. Composing these two,
we obtain a bijection
\begin{equation}\label{eq:3.6}
q \Sigma_{q\cL}^{-1} \circ \Sigma_\cL : \Psi_Q^{-1} (M^\circ,\cC_v^M,\cL) \to
q \Psi_{QM}^{-1} (M,\cC_v^M,q\cL) .
\end{equation}
Since $\cL$ is a subsheaf of $\cL$ and the $W_\cL^\circ$-action on $q\cL$ extends that
on $\cL$, $\Sigma_\cL (y,\rho^\circ)$ is contained in $q \Sigma_{q\cL}(y,\tau^\circ)$
for some $\tau^\circ$. Hence \eqref{eq:3.6} preserves the fibers over $y$. This 
provides a canonical bijection between the two sets figuring in the lemma.

The action of $W_\cL^Q$ on $q \cL$ and the sheafs associated to it for $\Sigma_\cL$
and $q\Sigma_{q \cL}$ comes from $Q \subset G^\circ$, so it fixes the part 
$\ind_{M_\cL}^M \rho_M$ in \eqref{eq:3.5}. Now it follows from the descriptions of 
$\Sigma_\cL$ and $q \Sigma_{q\cL}$ in \cite[\S 5]{AMS} that
\begin{equation}\label{eq:3.9}
q \Sigma_{q\cL}^{-1} \circ \Sigma_\cL (y,\rho^\circ) = 
\big( y, \big( \ind_{M_\cL}^M (\rho_M \otimes \rho^\circ ) \big)_y \big) = 
\big( y, \ind_{Z_{M_\cL}(y)}^{Z_M (y)} (\rho_M \otimes \rho^\circ ) \big) .
\end{equation}
For the same reasons the action of $\pi_0 (Z_Q (y))$ on \eqref{eq:3.9} fixes the
$\ind_{M_\cL}^M \rho_M$ part pointwise, and sees only $\rho^\circ$. To analyse
the right hand side as representation of $\pi_0 (Z_{QM}(y))$, we investigate
$Z_M (y) / Z_{M^\circ}(y)$. Using Lemma \ref{lem:3.1} we find 
\begin{equation}\label{eq:3.8}
\begin{split}
\pi_0 ( Z_{QM}(y) ) / \pi_0 (Z_Q (y)) & = Z_{QM} (y) / Z_Q (y) \; = \;
Z_{QM}(y) / Z_{Q M^\circ}(y) \\
& \cong \mathrm{Stab}_{M / M^\circ} (\Ad (Q M^\circ) y) \; = \; M / M^\circ \\
& = \mathrm{Stab}_{M / M^\circ} (\Ad (M^\circ) y) \; \cong \; Z_M (y) / Z_{M^\circ}(y).
\end{split}
\end{equation}
With \eqref{eq:3.8} we can identify the representation on the right hand side of
\eqref{eq:3.9} with ,
\begin{equation}\label{eq:3.10}
\ind_{\pi_0 (Z_{Q M_\cL}(y))}^{\pi_0 (Z_{QM} (y))} (\rho_M \otimes \rho^\circ ) .
\end{equation}
We already knew that it is irreducible, so $\pi_0 (Z_{Q M_\cL}(y))$ must be the
stabilizer of $\rho^\circ \in \Irr \big( \pi_0 (Z_Q (y)) \big)$ in $\pi_0 (Z_{QM} 
(y)) / \pi_0 (Z_Q (y)) \cong M / M^\circ$. In other words, \eqref{eq:3.10} equals 
$\rho_M \ltimes \rho^\circ \in \Irr \big( \pi_0 (Z_{QM} (y)) \big)$.
\end{proof}

\begin{lem}\label{lem:3.3}
Let $\sigma_0, y, \rho^\circ$ be as in Lemma \ref{lem:3.2}, and define $\sigma 
\in \mf g$ as in Lemma \ref{lem:2.2}. With the identification \eqref{eq:3.11},
the $\mh H (G^\circ M,M,q\cL)$-module $E^\circ_{y,\sigma,r,\rho^\circ \rtimes \rho_M}$
is canonically isomorphic to the $\mh H (G^\circ,M^\circ,\cL)$-module
$E^\circ_{y,\sigma,r,\rho^\circ}$.
\end{lem}
\begin{proof}
Let us recall that
\begin{equation}\label{eq:3.12}
\begin{array}{ll}
E^\circ_{y,\sigma,r,\rho^\circ} & = \; \Hom_{\pi_0 (Z_{G^\circ}(\sigma_0,y))}
\big( \rho^\circ, \C_{\sigma,r} \underset{H_*^{M(y)^\circ}(\{y\}) }{\otimes} 
H_*^{M(y)^\circ} (\cP_y^\circ,\dot{\cL}) \big) , \\
E^\circ_{y,\sigma,r,\rho^\circ \rtimes \rho_M}\hspace{-2mm} & 
= \; \Hom_{\pi_0 (Z_{G^\circ M}(\sigma_0,y))} \big( \rho^\circ \rtimes \rho_M, 
\C_{\sigma,r} \!\! \underset{H_*^{M(y)^\circ}(\{y\}) }{\otimes} \hspace{-3mm}
H_*^{M(y)^\circ} (\cP_y^\circ,\dot{q \cL}) \big) .
\end{array}
\end{equation}
Here the first $\cP_y^\circ$ is a subset of $G^\circ / P^\circ$, whereas the second 
$\cP_y^\circ$ is contained in $G^\circ M / P$. Yet they are canonically isomorphic via
$G^\circ / P^\circ \isom G^\circ P / P = G^\circ M / P$. By \eqref{eq:3.5}
\begin{align*}
\C_{\sigma,r} \underset{H_*^{M(y)^\circ}(\{y\}) }{\otimes}  H_*^{M(y)^\circ} 
(\cP_y^\circ,\dot{q \cL}) & \cong \ind_{M_\cL}^M \big( \rho_M \otimes \C_{\sigma,r} 
\underset{H_*^{M(y)^\circ}(\{y\})}{\otimes} 
H_*^{M(y)^\circ} (\cP_y^\circ,\dot{\cL}) \big) \\
& = \ind_{M_\cL}^M \big( \rho_M \otimes E^\circ_{y,\sigma,r} \big)
\end{align*}
From this and Proposition \cite[prop:1.1.d]{AMS} we see that
\[
E^\circ_{y,\sigma,r,\rho^\circ \rtimes \rho_M} \cong
\Hom_{\C[M_\cL / M^\circ,\natural_\cL^{-1}]} \Big( \rho_M, 
\Hom_{\pi_0 (Z_{G^\circ}(\sigma_0,y))} \big( \rho^\circ, \ind_{M_\cL}^M 
\big( \rho_M \otimes E^\circ_{y,\sigma,r} \big) \big) \Big) .
\]
Recall from Proposition \ref{prop:2.5} that $\rho^\circ$ only sees the cuspidal
support $(M^\circ,v,\cL)$. In the above expression the part of 
$\ind_{M_\cL}^M$ associated to $M \setminus M_\cL$ gives rise to cuspidal supports 
$(M^\circ,v,m \cdot \cL)$ with $m \cdot \cL \not\cong \cL$, so this part does not
contribute to $E^\circ_{y,\sigma,r,\rho^\circ \rtimes \rho_M}$. We conclude that
\[
\begin{split}
E^\circ_{y,\sigma,r,\rho^\circ \rtimes \rho_M} \cong
\Hom_{\C[M_\cL / M^\circ,\natural_\cL^{-1}]} \Big( \rho_M, 
\Hom_{\pi_0 (Z_{G^\circ}(\sigma_0,y))} \big( \rho^\circ, 
\rho_M \otimes E^\circ_{y,\sigma,r} \big) \Big) = \\
\Hom_{\C[M_\cL / M^\circ,\natural_\cL^{-1}]} \Big( \rho_M, \rho_M \otimes 
E^\circ_{y,\sigma,r,\rho^\circ} \Big) = E^\circ_{y,\sigma,r,\rho^\circ} . \qedhere
\end{split}
\]
\end{proof}

We note that, as a consequence of Lemmas \ref{lem:3.2}, \ref{lem:3.3} and Theorem 
\ref{thm:2.9}, Theorem \ref{thm:2.9} is also valid with $G^\circ$ replaced by 
$G^\circ M$, $L$ by $M$ and $\cL$ by $q\cL$. Knowing this and assuming Condition
\ref{cond:3}, we can use Clifford
theory to relate $\Irr (\mh H (G,M,q\cL))$ to $\Irr (\mh H (G^\circ M,M,q\cL))$.
All of Paragraphs \ref{par:intop}--\ref{par:temp} remains valid in the
setting of the current section. Let us summarise the most important results,
analogues of Theorem \ref{thm:2.16} and Corollary \ref{cor:2.18}.
In view of Lemma \ref{lem:2.19} we do not need Condition \ref{cond:3} anymore
once we have obtained these results. Therefore we state them without assuming
Condition \ref{cond:3}.

\begin{thm}\label{thm:3.4}
Fix $r \in \C$. 
\enuma{
\item Let $y,\sigma \in \mf g$ with $y$ nilpotent, $\sigma$ semisimple
and $[\sigma,y] = 2r y$. Let $\tau \in \\ \Irr \big( \pi_0 (Z_G (\sigma_0,y)) \big)$
such that $q \Psi_{Z_G (\sigma_0)} (y,\tau) = (M,\cC_v^M,q\cL)$ (up to 
$G$-conjugation). With these data we associate the 
$\mh H (G^\circ N_G (P,q\cL),M,q\cL)$-module
\[
E_{y,\sigma,r,\tau} = \Hom_{\pi_0 (Z_{G^\circ N_G (P,q\cL)}(\sigma_0,y))}
\big( \tau, \C_{\sigma,r} \underset{H_*^{M(y)^\circ}(\{y\}) }{\otimes} 
H_*^{M(y)^\circ} (\cP_y^\circ,\dot{q\cL}) \big) .
\]
Via \eqref{eq:3.13} we consider it also as a $\mh H (G,M,q\cL)$-module.

Then the $\mh H (G,M,q\cL)$-module $E_{y,\sigma,r,\tau}$ has a 
distinguished irreducible quotient $M_{y,\sigma,r,\tau}$, which appears with 
multiplicity one in $E_{y,\sigma,r,\tau}$.
\item The map $M_{y,\sigma,r,\tau} \longleftrightarrow (y,\sigma,\tau)$ gives
a bijection between $\Irr_r (\mh H (G,M,q\cL))$ and $G$-conjugacy classes of
triples as in part (a).
\item The set $\Irr_r (\mh H (G,M,q\cL))$ is also canonically in bijection with
the following two sets:
\begin{itemize}
\item $G$-orbits of pairs $(x,\tau)$ with $x \in \mf g$ and $\tau \in \Irr \big(
\pi_0 (Z_G (x)) \big)$ such that $q \Psi_{Z_G (x_S)} (x_N,\tau) = (M,\cC_v^M,q\cL)$
up to $G$-conjugacy.
\item $N_G (M)/M$-orbits of triples $(\sigma_0,\cC,\cF)$, with $\sigma_0 \in \mf t$,
$\cC$ a nilpotent $Z_G (\sigma_0)$-orbit in $Z_{\mf g}(\sigma_0)$ and $\cF$ a
$Z_G (\sigma_0)$-equivariant cuspidal local system on $\cC$ such that
$q \Psi_{Z_G (\sigma_0)} (\cC,\cF) = (M,\cC_v^M,q\cL)$ up to $G$-conjugacy.
\end{itemize}
}
\end{thm}

\appendix

\section{Compatibility with parabolic induction}

It has turned out that Theorem \ref{thm:2.24} is not correct as stated. The maps
given there have almost all the claimed properties, the only problem is that
usually they are not surjective. In this appendix (written in June 2018) we will 
repair that, by proving a weaker version of the Theorem. 

Given $Q$ as on page \pageref{eq:2.54}, $P Q^\circ$ is a parabolic subgroup of
$G^\circ$ with $Q^\circ$ as Levi factor. The unipotent radical $\mc R_u (P Q^\circ)$
is normalized by $Q^\circ$, so its Lie algebra $\mf u_Q = \mr{Lie} (\mc R_u (P Q^\circ))$
is stable under the adjoint actions of $Q^\circ$ and $\mf q$. In particular 
ad$(y)$ acts on $\mf u_Q$. We denote the cokernel of ad$(y) : \mf u_Q \to \mf u_Q$ by
$_y \mf u_Q$. For $v \in \mf u_Q$ and $(\sigma,r) \in \mr{Lie}(M(y)^\circ)$ we have
\[
[\sigma, [y,v]] = [y, [\sigma,v]] + [[\sigma,y], v] = 
[y, [\sigma,v]] + [2r y,v] \in \mr{ad}(y)\mf u_Q .
\]
Hence $\mr{ad}(y)$ descends to a linear map ${}_y \mf u_Q \to {}_y \mf u_Q$. 
Following Lusztig \cite[\S 1.16]{Lus7}, we define 
\[
\begin{array}{cccc}
\epsilon : & \mr{Lie}(M^Q (y)^\circ) & \to & \C \\
& (\sigma,r) & \mapsto & \det( \mr{ad}(\sigma) - 2r : {}_y \mf u_Q \to {}_y \mf u_Q) 
\end{array}.
\]
It is easily seen that $\epsilon$ is invariant under the adjoint action of $M^Q (y)$, so it 
defines an element of $H^*_{M^Q (y)}(\{y\})$. For a given $y$, all the parameters 
$(y,\sigma,r)$ for which parabolic induction from $\mh H (Q,L,\mc L)$ to 
$\mh H (G,L,\mc L)$ can behave problematically, are zeros of $\epsilon$. 

For any closed subgroup $C$ of $M^Q (y)^\circ$, $\epsilon$ yields an element $\epsilon_C$ of 
$H^*_C (\{y\})$ (by restriction). We recall from \cite[Proposition 7.5]{Lus3} that there is
a natural isomorphism 
\begin{equation}\label{eq:A.1}
H^C_* (\mc P_y, \dot{\mc L}) \cong H_C^* (\{y\}) \underset{H^*_{M^Q (y)}(\{y\})}{\otimes}
H^{M (y)^\circ}_* (\mc P_y, \dot{\mc L}) .
\end{equation}
Here $H_C^* (\{y\})$ acts on the first tensor leg and $\mh H (G,L,\mc L)$ acts on the second 
tensor leg. By Theorem \ref{thm:2.4}.b these actions commute, and $H^C (\mc P_y, \dot{\mc L})$
becomes a module over $H_C^* (\{y\}) \otimes_\C \mh H (G,L,\mc L)$.

Now we can formulate an improved version of Theorem \ref{thm:2.24}.

\begin{thm}\label{thm:A.1}
Let $Q$ and $y$ be as on page \pageref{eq:2.54}, and let $C$ be a maximal torus of 
$M^Q (y)^\circ$.
\enuma{
\item The map \eqref{eq:2.54} induces an injection of $\mh H (G,L,\cL)$-modules
\[
\mh H (G,L,\cL) \underset{\mh H (Q,L,\cL)}{\otimes} H_*^C (\mc P_y^Q,\dot{\cL})
\to H_*^C (\mc P_y,\dot{\cL}) .
\]
It respects the actions of $H_C^* (\{y\})$ and its image contains 
$\epsilon_C H_*^C (\mc P_y,\dot{\cL})$.
\item Let $(\sigma,r) / \!\! \sim \in V_y^Q$ be such that $\epsilon (\sigma,r) \neq 0$
or $r=0$. The map \eqref{eq:2.54} induces an isomorphism of $\mh H (G,L,\cL)$-modules
\[
\mh H (G,L,\cL) \underset{\mh H (Q,L,\cL)}{\otimes} E^Q_{y,\sigma,r} 
\to E_{y,\sigma,r} ,
\]
which respects the actions of $\pi_0 (M^Q (y))_\sigma$.
}
\end{thm}
\begin{proof}
(a) The given proof of Theorem \ref{thm:2.24} is valid with only one modification.
Namely, the diagram \eqref{eq:2.55} does not commute. A careful consideration of
\cite[\S 2]{Lus7} shows failure to do so stems from the difference between certain
maps $i_!$ and $(p^*)^ {-1}$, where $p$ is the projection of a vector bundle on its
base space and $i$ is the zero section of the same vector bundle. In \cite[Lemma 2.18]{Lus7}
this difference is identified as multiplication by $\epsilon_C$. \\
(b) For $(\sigma, r) \in \mr{Lie}(C)$ with $\epsilon (\sigma,r) \neq 0$, the proof
of Theorem \ref{thm:2.24}.b needs only one small adjustment. From \eqref{eq:A.1} we get
\begin{multline*}
\C_{\sigma,r} \underset{H_C^* (\{y\})}{\otimes} \epsilon_C H_*^C (\mc P_y, \dot{\cL}) \cong 
\C_{\sigma,r} \underset{H_C^* (\{y\})}{\otimes} \epsilon_C H^*_C (\{y\}) \underset{
H_{M(y)^\circ}^* (\{y\})}{\otimes} H_*^{M (y)^\circ} (\mc P_y, \dot{\cL}) \cong \\
\C_{\sigma,r} \underset{H_{M(y)^\circ}^* (\{y\})}{\otimes} H_*^{M (y)^\circ} 
(\mc P_y, \dot{\cL}) = E_{y,\sigma,r} .
\end{multline*}
The difference with before is the appearance of $\epsilon_C$, with that and the above
the proof of Theorem \ref{thm:2.24}.b goes through.

For $r = 0$, we know from \eqref{eq:2.12} that
\[
\mh H (G^\circ, L,\mc L) \underset{\mh H (Z_{G^\circ} (\sigma^\circ ),L,\mc L)}{\otimes}
E^{Z_{G^\circ} (\sigma^\circ )}_{y,\sigma_0,0} = E^\circ_{y,\sigma_0,0} .
\]
Let $Q_y \subset Z_{G^\circ} (\sigma^\circ )$ be a Levi subgroup which is minimal for the 
property that it contains $L$ and $\exp (y)$. Then $S(\mf t^* \oplus \C)$ acts on both
\begin{equation}\label{eq:A.2}
E^{Z_{G^\circ} (\sigma^\circ )}_{y,\sigma_0,0} \quad \mr{and} \quad
\mh H (Z_{G^\circ} (\sigma^\circ ), L,\mc L) \underset{\mh H (Q_y,L,\mc L)}{\otimes}
E^{Q_y}_{y,\sigma_0,0} 
\end{equation}
by evaluation at $(\sigma_0,0)$. Hence the structure of these two 
$\mh H (Z_{G^\circ} (\sigma^\circ ), L,\mc L)$-modules is completely determined by
the action of $\C [ W_\cL^{Z_{G^\circ} (\sigma^\circ )} ]$. But by Theorem \ref{thm:2.4}.c
\begin{equation}\label{eq:A.3}
E^{Z_{G^\circ} (\sigma^\circ )}_{y,\sigma,r} \quad \mr{and} \quad
\mh H (Z_{G^\circ} (\sigma^\circ ), L,\mc L) \underset{\mh H (Q_y,L,\mc L)}{\otimes}
E^{Q_y}_{y,\sigma,r} 
\end{equation}
do not depend on $(\sigma,r)$ as $\C [ W_\cL^{Z_{G^\circ} (\sigma^\circ )} ]$-modules.
From a case with $\epsilon (\sigma,r)$ we see that these two 
$W_\cL^{Z_{G^\circ} (\sigma^\circ )}$-representations are naturally isomorphic. Together
with \eqref{eq:2.12} that gives a natural isomorphism of $\mh H (G^\circ,L,\mc L)$-modules
\begin{equation}\label{eq:A.4}
\mh H (G^\circ, L,\mc L) \underset{\mh H (Q_y,L,\mc L)}{\otimes}
E^{Q_y}_{y,\sigma_0,0} \to E^\circ_{y,\sigma_0,0} . 
\end{equation}
By the transitivity of induction, \eqref{eq:A.4} entails that
\begin{equation}\label{eq:A.5}
\mh H (G^\circ, L,\mc L) \underset{\mh H (Q^\circ,L,\mc L)}{\otimes}
E^{Q^\circ}_{y,\sigma_0,0} \cong E^\circ_{y,\sigma_0,0} .
\end{equation}
From \eqref{eq:A.5} and Lemma \ref{lem:2.1} we get natural isomorphisms of 
$\mh H (G^\circ,L,\mc L)$-modules
\begin{multline*}
E_{y,\sigma_0,0} \cong \mh H (G, L,\mc L) \underset{\mh H (G^\circ,L,\mc L)}{\otimes}
E^\circ_{y,\sigma_0,0} \cong \mh H (G, L,\mc L) \underset{\mh H (Q^\circ,L,\mc L)}{\otimes}
E^{Q^\circ}_{y,\sigma_0,0} \\
\mh H (G, L,\mc L) \underset{\mh H (Q,L,\mc L)}{\otimes}
\mh H (Q, L,\mc L) \underset{\mh H (Q^\circ,L,\mc L)}{\otimes}
E^{Q^\circ}_{y,\sigma_0,0} \cong \mh H (G, L,\mc L) \underset{\mh H (Q,L,\mc L)}{\otimes}
E^Q_{y,\sigma_0,0} . 
\end{multline*}
Here the composed isomorphism is still induced by \eqref{eq:2.54}, so just as in the case
$\epsilon (\sigma,r) \neq 0$ it is $\pi_0 (M^Q (y))_\sigma$-equivariant.
\end{proof}

There is just result in the paper that uses Theorem \ref{thm:2.24}, namely Proposition
\ref{prop:2.25}. We have to replace that by a version which involves only the cases of
Theorem \ref{thm:2.24} covered by Theorem \ref{thm:A.1}. Fortunately, under the extra
condition $r=0$ or $\epsilon (\sigma,r) \neq 0$ the proof of Proposition \ref{prop:2.25}
goes through, when we replace the references to Theorem \ref{thm:2.24} by references to
Theorem \ref{thm:A.1}.

Since $\epsilon$ is a polynomial function, its zero set is a subvariety of smaller 
dimension (say of $V_y$). Nevertheless, we also want to explicitly exhibit a large class
of parameters $(y,\sigma,r)$ on which $\epsilon$ does not vanish. By Proposition 
\ref{prop:2.3}.c it suffices to do so under the assumption that $\sigma ,\sigma_0 \in \mf t$.

Let us call $x \in \mf t$ (strictly) positive with respect to $P Q^\circ$ if 
$\Re (\alpha (t))$ is (strictly) positive for all $\alpha \in R (\mc R_u (P Q^\circ),T)$.
We say that $x$ is (strictly) negative with respect to $P Q^\circ$ if $-x$ is (strictly)
positive.

\begin{lem}\label{lem:A.2}
Let $y \in \mf q$ be nilpotent and let $(\sigma,r) \in \mf t \oplus \C$ with
$[\sigma,y] = 2r y$. Suppose that $\sigma = \sigma_0 + \textup d \gamma_y \matje{r}{0}{0}{-r}$
as in \eqref{eq:2.52}, with $\sigma_0 \in Z_{\mf t}(y)$. Assume furthermore that one 
of the following holds:
\begin{itemize}
\item $\Re (r) > 0$ and $\sigma_0$ is negative with respect to $P Q^\circ$;
\item $\Re (r) < 0$ and $\sigma_0$ is positive with respect to $P Q^\circ$;
\item $\Re (r) = 0$ and $\sigma_0$ is strictly positive or strictly negative
with respect to $P Q^\circ$.
\end{itemize}
Then $\epsilon (\sigma,r) \neq 0$.
\end{lem}
\begin{proof}
Via d$\gamma_y : \mf{sl}_2 (\C) \to \mf q$, $\mf u_Q$ becomes a finite dimensional
$\mf{sl}_2 (\C)$-module. Since $\sigma_0 \in \mf t$ commutes with $y$, it commutes with
d$\gamma_y (\mf{sl}_2 (\C))$. For any eigenvalue $\lambda \in \C$ of $\sigma_0$, let
$_\lambda \mf u_Q$ be the eigenspace in $\mf u_Q$. 

For $n \in \Z_{\geq 0}$ let $\mr{Sym}^n (\C^2)$ be the unique irreducible 
$\mf{sl}_2 (\C)$-module of dimension $n+1$. We decompose the $\mf{sl}_2 (\C)$-module 
$_\lambda \mf u_Q$ as
\[
_\lambda \mf u_Q = \bigoplus\nolimits_{n \geq 0} \mr{Sym}^n (\C^2)^{\mu (\lambda,n)} 
\quad \text{with} \quad \mu (\lambda,n) \in \Z_{\geq 0}.
\] 
The cokernel of $\matje{0}{1}{0}{0}$ on $\mr{Sym}^n (\C^2)$ is the lowest weight space
$W_{-n}$ in that representation, on which $\matje{r}{0}{0}{-r}$ acts as $-nr$. Hence
$\sigma$ acts on
\[
\mr{coker}(\mr{ad}(y) : _\lambda \mf u_Q \to _\lambda \mf u_Q) \cong
\bigoplus_{n \geq 0} W_{-n}^{\mu (\lambda,r)} \quad \text{as} \quad 
\bigoplus_{n \geq 0} (\lambda - nr) \mr{Id}_{W_n^{\mu (\lambda,r)}}.
\]
Consequently
\[
(\mr{ad}(\sigma) - 2r)\big|_{_\lambda \mf u_Q} = 
\bigoplus_{\lambda \in \C} \bigoplus_{n \geq 0} 
(\lambda - (n+2)r) \mr{Id}_{W_n^{\mu (\lambda,r)}}.
\]
By definition then 
\[
\epsilon (\sigma,r) = \prod_{\lambda \in \C} 
\prod_{n \geq 0} ( \lambda - (n+2) r )^{\mu (\lambda,n)} .
\]
When $\Re (r) > 0$ and $\sigma_0$ is negative with respect to $P Q^\circ$,
$\Re (\lambda - (n+2) r) < 0$ for all eigenvalues $\lambda$ of $\sigma_0$ on $\mf u_Q$,
and in particular $\epsilon (\sigma,r) \neq 0$.

Similarly, we see that $\epsilon (\sigma,r) \neq 0$ in the other two possible cases in the lemma.
\end{proof}

As an application of Lemma \ref{lem:A.2}, we prove a result in the spirit of the Langlands
classification for graded Hecke algebras \cite{Eve}. It highlights a procedure to obtain
irreducible $\mh H (G,L,\mc L)$-modules from irreducible tempered modules of a parabolic
subalgebra $\mh H (Q,L,\mc L)$: twist by a central character which is strictly positive
with respect to $P Q^\circ$, induce parabolically and then take the unique irreducible
quotient.

\begin{prop}\label{prop:A.3}
Let $y,\sigma,r,\rho$ be as in \ref{thm:2.16}
\enuma{
\item If $\Re (r) \neq 0$ and $\sigma_0 \in i \mf t_\R + Z(\mf g)$, then
$M_{y,\sigma,r,\rho} = E_{y,\sigma,r,\rho}$.
\item Suppose that $\Re (r) > 0$ and $\sigma,\sigma_0 \in \mf t$ such that $\Re (\sigma_0)$
is negative with respect to $P$. Then $\Re (\sigma_0)$ is strictly negative with
respect to $P Q^\circ$, where $Q = Z_G (\Re(\sigma_0))$. Further $M_{y,\sigma,r,\rho}$ is 
the unique irreducible quotient of $\mh H (G,L,\mc L) \underset{\mh H (Q,L,\mc L)}{\otimes} 
M^Q_{y,\sigma,r,\rho}$.
\item In the setting of part (b), $\IM^* M_{y,\sigma,r,\rho}$ is the unique irreducible 
quotient of 
\[
\IM^* \big( \mh H (G,L,\mc L) \underset{\mh H (Q,L,\mc L)}{\otimes} M^Q_{y,\sigma,r,\rho} \big)
\cong \mh H (G,L,\mc L) \underset{\mh H (Q,L,\mc L)}{\otimes} \IM^* (M^Q_{y,\sigma,r,\rho}) .
\]
Furthermore $\IM^* (M^Q_{y,\sigma,r,\rho})$ comes from the twist of a tempered 
$\mh H (Q^\circ,L,\mc L)$-module by a strictly positive character of $S(Z(\mf q^*))$.
}
\end{prop}
\textbf{Remark.} By \eqref{eq:2.48} the extra condition in part (a) holds for instance when
$\Re (r) > 0$ and $\IM^* (M_{y,\sigma,r,\rho})$ is tempered. By Proposition \ref{prop:2.3} 
and Lemma \ref{lem:1.1} every parameter $(y,\sigma)$ is 
$G^\circ$-conjugate to one with the properties as in (b) and (c).
\begin{proof}
(a) Write $\sigma_0 = \sigma_{0,\der} + z_0$ with $\sigma_{0,\der} \in \mf g_\der$ and 
$z_0 \in Z(\mf g)$. Then, as in the proof of \ref{cor:2.21},
\[
E^\circ_{y,\sigma,r,\rho^\circ} = E^\circ_{y,\sigma- z_0,r,\rho^\circ} \otimes \C_{z_0}
\quad \text{and} \quad 
M^\circ_{y,\sigma,r,\rho^\circ} = M^\circ_{y,\sigma- z_0,r,\rho^\circ} \otimes \C_{z_0} .
\]
By \cite[Theorem 1.21]{Lus7} $E^\circ_{y,\sigma - z_0,r,\rho^\circ} = 
M^\circ_{y,\sigma - z_0,r,\rho^\circ}$ as $\mh H (G_\der,L \cap G_\der,\mc L)$-modules, so
$E^\circ_{y,\sigma,r,\rho^\circ} = M^\circ_{y,\sigma,r,\rho^\circ}$ as 
$\mh H (G^\circ,L,\mc L)$-modules. Together with Lemma \ref{lem:2.15} and \eqref{eq:2.71} 
this gives $E_{y,\sigma,r,\rho} = M_{y,\sigma,r,\rho}$.\\
(b) Notice that $Z_G (\sigma,y) = Z_Q (\sigma,y)$, so by \cite[Theorem 4.8.a]{AMS} $\rho$ 
is a valid enhancement of the parameter $(\sigma,y)$ for $\mh H (Q,L,\mc L)$.

By construction $\Re (\sigma_0)$ is strictly negative with respect to
$P Q^\circ$. Now Lemma \ref{lem:A.2} says that we may apply Proposition \ref{prop:2.25}. 
That and part (a) yield
\[
\mh H (G,L,\mc L) \underset{\mh H (Q,L,\mc L)}{\otimes} M^Q_{y,\sigma,r,\rho} =
\mh H (G,L,\mc L) \underset{\mh H (Q,L,\mc L)}{\otimes} E^Q_{y,\sigma,r,\rho} =
E_{y,\sigma,r,\rho} .
\]
Now apply Theorem \ref{thm:2.16}.b.\\
(c) The first statement follows from part (b) and \ref{eq:2.62}. Write
\[
M^{Q^\circ}_{y,\sigma,r,\rho^\circ} = M^{Q^\circ}_{y,\sigma- z_0,r,\rho^\circ} \otimes \C_{z_0}
= M^{Q^\circ}_{y,\sigma- z_0,r,\rho^\circ} \otimes \big( \C_{z_0 - \Re (z_0)} \otimes
\C_{\Re (z_0)} \big)
\]
as in the proof of part (a), with $Q$ in the role of $G$. By Theorem \ref{thm:2.19}.b
$M^{Q^\circ}_{y,\sigma- z_0,r,\rho^\circ} \otimes \C_{z_0 - \Re (z_0)}$ is anti-tempered.
The definition of $Q$ entails that $\Re (z_0) = \Re (\sigma_0)$, which we know is strictly 
negative. Hence 
\[
\IM^* (M^{Q^\circ}_{y,\sigma,r,\rho^\circ}) = \IM^* (M^{Q^\circ}_{y,\sigma- z_0,r,\rho^\circ} 
\otimes \C_{z_0 - \Re (z_0)}) \otimes \C_{-\Re (\sigma_0)},
\]
where the right hand side is the twist of a tempered $\mh H (Q^\circ,L,\mc L)$-module by the 
strictly positive character $-\Re (\sigma_0)$ of $S(Z(\mf q^*))$. By \eqref{eq:2.53}
\begin{equation}\label{eq:A.2}
\IM^* (M^Q_{y,\sigma,r\rho}) = \tau \ltimes \IM^* (M^{Q^\circ}_{y,\sigma,r,\rho^\circ}) . 
\qedhere
\end{equation}
\end{proof}
 
We note that by Lemma \ref{lem:2.14} $S(Z(\mf q^*))$ acts on \eqref{eq:A.2} by the characters 
$\gamma (-\Re (\sigma_0))$ with $\gamma \in \mf R^Q$. Since $\mf R^Q$ normalizes $P Q^\circ$, 
it preserves the strict positivity of $-\Re (\sigma_0)$. In this sense 
$\IM^* (M^Q_{y,\sigma,r\rho})$ is essentially the twist of a tempered 
$\mh H (Q,L,\mc L)$-module by a strictly positive central character.

\end{document}